\documentclass[11pt, a4paper, twoside, reqno]{amsart}
\usepackage[utf8]{inputenc}
\usepackage[T1]{fontenc}
\usepackage{lmodern}
\usepackage{microtype}
\usepackage[frak=boondox]{mathalpha}
\usepackage{dsfont}
\usepackage{bbold}
\usepackage{upgreek}
\usepackage{mathrsfs}
\usepackage{euscript}
\usepackage{amssymb}
\usepackage{mathabx}
\usepackage[top=1in, bottom=1in, left=1in, right=1in]{geometry}
\usepackage{enumitem}
\usepackage{parskip}
\usepackage{float}
\usepackage{caption}
\usepackage{tabto}
\usepackage{color}
\usepackage{mdframed}
\usepackage{hyphenat}
\usepackage{comment}
\usepackage{setspace}
\usepackage{etoolbox}
\usepackage{amsmath, amssymb, amsthm, mathtools}
\usepackage{extarrows}
\usepackage{tikz-cd}
\usepackage[all,cmtip]{xy} 
\usepackage{graphicx}

\usepackage[numbers]{natbib}
\setcitestyle{open={},close={}}
\makeatletter
\renewcommand{\@biblabel}[1]{[#1]\hfill}
\makeatother

\usepackage[hidelinks]{hyperref}
\usepackage[nameinlink]{cleveref}
\newtheoremstyle{mystyle}
  {\topsep}    
  {\topsep}   
  {\itshape}   
  {}           
  {\bfseries}   
  {.}          
  {.5em}        
  {}            

\newtheoremstyle{spacedremark} 
  {\topsep}    
  {\topsep}     
  {\normalfont} 
  {}           
  {\bfseries}   
  {.}          
  {.5em}        
  {}               

\newtheorem{thm}{Theorem}[subsection]
\crefname{thm}{Theorem}{Theorems}

\newtheorem{lem}[thm]{Lemma}
\crefname{lem}{Lemma}{Lemmas}

\crefname{Q}{Question}{Questions}

\newtheorem{prop}[thm]{Proposition}
\crefname{prop}{Proposition}{Propositions}

\newtheorem{cor}[thm]{Corollary}
\crefname{cor}{Corollary}{Corollaries}

\crefname{conj}{Conjecture}{Conjectures}

\newtheorem{defn}[thm]{Definition}
\crefname{defn}{Definition}{Definitions}

\newtheorem{inner-recall-star}{Theorem}
\crefname{inner-recall-star}{Theorem}{Theorems}
\newenvironment{recall*}[1]
  {\begin{inner-recall-star}[see \Cref{#1}]}
  {\end{inner-recall-star}}

\theoremstyle{spacedremark}

\newtheorem{rem}[thm]{Remark}
\crefname{rem}{Remark}{Remarks}
\AtEndEnvironment{rem}{\hfill\qedsymbol}

\crefname{con}{Construction}{Constructions}

\newtheorem{coex}[thm]{Counterexample}
\crefname{coex}{Counterexample}{Counterexamples}
\AtEndEnvironment{coex}{\hfill\qedsymbol}
\newtheorem{coex*}[thm]{Counterexamples}
\crefname{coex*}{Counterexample}{Counterexamples}

\crefname{obs}{Observation}{Observations}

\newtheorem{ex}[thm]{Example}
\crefname{ex}{Example}{Examples}

\newtheorem*{thm*}{Theorem}

\DeclareMathOperator*{\colim}{colim}

\newcommand{\dcolim}{\varinjlim}
\newcommand{\plim}{\varprojlim}

\DeclareMathAlphabet{\duc}{U}{dutchcal}{m}{n}
\SetMathAlphabet{\duc}{bold}{U}{dutchcal}{b}{n}
\DeclareFontFamily{U}{BOONDOX-calo}{\skewchar\font=45 }
\DeclareFontShape{U}{BOONDOX-calo}{m}{n}{<-> s*[1.0] BOONDOX-r-calo}{}
\DeclareFontShape{U}{BOONDOX-calo}{b}{n}{<-> s*[1.0] BOONDOX-b-calo}{}
\DeclareMathAlphabet{\mal}{U}{BOONDOX-calo}{m}{n}
\SetMathAlphabet{\mal}{bold}{U}{BOONDOX-calo}{b}{n}

\newcommand*{\mcap}{\mathbin{\scalebox{.9}[1]{\ensuremath{\bigcap}}}}

\newcommand{\esc}{\EuScript}

\setlist[enumerate]{leftmargin=*, nosep}
\begin{document}

\title[Birational Algebraic Topology]{Birational Algebraic Topology}

\author[Dipankar Maity]{Dipankar Maity}
\address{Department of Mathematical Sciences \\
Indian Institute of Science Education and Research (IISER) Mohali \\ Knowledge city, Sector 81, SAS Nagar, Manauli PO 140306 \\
India}
\email{homotopytypes@gmail.com}

\begin{abstract}
Over any qcqs scheme $S$, we show that inverting dense open immersions in $\mathcal{H}^{\mathbb{A}^1}(S)$, $log\mbox{-}\mathcal{H}^\Box(S)$, and $\mathcal{P}_\Sigma(\mathrm{Sm}_S)$ yields the same birational $\infty$-category $\mathcal{H}^b(S)$. Unlike in $\mathcal{H}^{\mathbb{A}^1}(S)$, all Nisnevich truncations and homotopy sheaves (including $\pi_0^{nis}$) of birational spaces remain birational, and hence $\mathcal{H}^b(S)$ is Postnikov complete. Moreover, the birational localization $L_{bir}$ commutes with the bar construction, ensuring effective delooping for every grouplike object in $\mathcal{H}^b(S)$ and yielding an unstable birational higher-connectivity theorem. Building on these results, we establish birational analogs of the van Kampen theorem, the higher Galois theory of coverings, and the Freudenthal suspension theorem. When $S$ is geometrically unibranch with finitely many generic points, we show that $\mathcal{H}^b(S)$ is in fact an $\infty$-topos of homotopy dimension zero.

Over a perfect field, we identify discrete birational group sheaves with strongly $\mathbb{A}^1$-invariant sheaves whose $\mathbb{G}_m$-contraction vanishes, and deduce that $S^{2,1}$-nullification $L^{2,1}$ agrees with $L_{bir}$ on pointed $L^{2,1}$-connected spaces. Finally, over any field, we prove that $L_{bir}$ preserves generic values of $\pi_0^{\mathbb{A}^1}$ on motivic colimits of proper schemes; over perfect fields, it therefore detects their $\mathbb{A}^1$-connectivity and yields an equivalence $\mathbb{H}_0^{\mathbb{A}^1} \simeq \mathbb{H}_0^b$ for them, answering a question of Asok and Haesemeyer in greater generality.
\end{abstract}

\maketitle

\tableofcontents

\section{Introduction to Birational Algebraic Topology}

\subsection{The aims and philosophy of birational motivic homotopy theory}
Birational geometry studies birational morphisms of schemes. Written $X\dashrightarrow Y$, a birational morphism from $X$ to $Y$ is a span $X \xhookleftarrow{i} U \xhookrightarrow{j} Y$ of dense open immersions (i.e., both $i$ and $j$ have homeomorphic images that are dense in their codomain and induce isomorphisms $i^*\mathcal{O}_{X}\simeq \mathcal{O}_U\simeq j^*\mathcal{O}_Y$ of structure sheaves). One of the classical questions in birational geometry is the rationality problem: how can one determine whether a given variety is birational to $\mathbb{A}^n$? However clean the statement may seem, the rationality problem has no universal solution.

The issue is that birational morphisms are not morphisms in any sufficiently well-behaved $1$-category that suitably enlarges the category of (smooth) schemes. Moreover, the rationality question is a rigid one, in the sense that it is really about isomorphism classes. Problems of this type often require one to consider the `possible' $(2, 1)$-category of such spans, which extends the $1$-category of (smooth) schemes in a suitable manner. 

One of the best techniques for obtaining negative answers to such problems is to construct invariants that induce isomorphisms on such spans. However, an appropriate, well-behaved invariant must satisfy the properties of a good cohomology theory, for which even $2$-categories are insufficient. One shall rather build a category that has all higher categorical colimits. The `minimal' such category, which contains all smooth schemes and is a natural home for such spans, is precisely the point of the birational homotopy category. This is the localization of $\mathcal{P}_\Sigma(\mathrm{Sm}_S)$ (see [\cite{lurie2009higher}, Definition~5.5.8.8] for the notation) at the set of dense open immersions, denoted $\mathcal{H}^{b}(S):=L_{dense}\mathcal{P}_{\Sigma}(\mathrm{Sm}_S)$.

Frankly, if one really wants cohomology theories of sufficiently good nature, a priori, it is not clear why our birational category should enforce them. Therefore, it is suggestive to start with a good cohomological extension rather than $\mathcal{P}_\Sigma(\mathrm{Sm}_S)$. One of the best and most popular such unstable candidates is, no doubt, the $\mathbb{A}^1$-homotopy category $\mathcal{H}^{\mathbb{A}^1}(S)$, constructed by Morel and Voevodsky [\cite{morel19991}] as the $\mathbb{A}^1$-localization of the Nisnevich $\infty$-topos $\mathcal{P}_{nis}(\mathrm{Sm}_S)$. So the correct candidate for birational (motivic) homotopy shall rather be $\mathcal{H}^{b\mathbb{A}^1}(S):=L_{dense}\mathcal{H}^{\mathbb{A}^1}(S)$. Historically, this construction first appeared in [\cite{pelaez2014unstable}] to study a possible slice filtration of the unstable motivic homotopy category, and has had a significant impact on motivic slice computations (see [\cite{bachmann2019voevodsky}]).
 
With a different goal in mind, Choudhury and Roy [\cite{choudhury2022characterisation}] considered both constructions over a field $k$ and showed in Theorem~3.4 of \textit{loc.\ cit.} that they coincide: homotopy invariance and Nisnevich locality are automatic in $\mathcal{H}^{b}(k)$. This result also appeared in a work of Cisinski and Kahn [\cite{cisinski2023homotopy}, Corollary~2.3]. However, both contain a technical oversight: they rely on the unrestricted $\infty$-category of presheaves $\mathcal{P}(\mathrm{Sm}_k)$ rather than $\mathcal{P}_\Sigma(\mathrm{Sm}_k)$, so that their $\mathbf{H}_b(k)$ (which evaluates to $L_{dense}\mathcal{P}(\mathrm{Sm}_k)$ in our setting) contains objects that fail to be local with respect to disjoint covers. In \S2 of this paper, we address this gap by providing a rigorously corrected version of this strategy while generalizing it uniformly to all qcqs base schemes $S$. Moreover, the proof of automatic $\mathbb{A}^1$-invariance of birational-local spaces used in both sources is taken directly from [\cite{MR3404383}, Appendix~A], which uses a blow-up construction to first establish $\mathbb{P}^1$-invariance and then deduces $\mathbb{A}^1$-invariance from the dense open immersion $\mathbb{A}^1\hookrightarrow \mathbb{P}^1$. In this paper, we offer an entirely straightforward slope-and-ray argument on $\mathbb{A}^2$ to establish automatic $\mathbb{A}^1$-invariance directly while simultaneously working over arbitrary base schemes (\Cref{bir local is rational local}). Together, these yield an immediate identification of the birational homotopy category and the birational motivic homotopy category: $\mathcal{H}^b(S)=\mathcal{H}^{b\mathbb{A}^1}(S)$ (\Cref{Motivic equivalences are Birational equivalences}; see also [\cite{sfbat}, Theorem~2.3.6]). In fact, in \S\ref{2.3}, we show that even if one starts with the logarithmic motivic homotopy category $\mathrm{log}\mbox{-}\mathcal{H}^\Box(S)$ of [\cite{binda2023logarithmic}], localizing at dense open immersions of log schemes yields nothing new: $\mathrm{log}\mbox{-}\mathcal{H}^b(S)\simeq \mathcal{H}^b(S)$ (\Cref{Equivalence of ordinary and logarithmic birational motivic homotopy categories}). We may therefore unambiguously denote the birational motivic homotopy category by $\mathcal{H}^b(S)$:
$$\mathrm{log}\mbox{-}\mathcal{H}^b(S)\simeq \mathcal{H}^b(S)\simeq \mathcal{H}^{b\mathbb{A}^1}(S).$$
There is a second, deeper structural reason to pass from $\mathcal{H}^{\mathbb{A}^1}(S)$ to $\mathcal{H}^b(S)$. Recall that the $\mathbb{A}^1$-homotopy category of [\cite{morel19991}] is obtained from the category of smooth schemes as the universal presentable $\infty$-category satisfying Nisnevich descent and $\mathbb{A}^1$-invariance [\cite{robalo2012noncommutative}, Theorem~5.2]. More precisely, we have the following universal property.

\begin{inner-recall-star}[\textup{[\cite{robalo2012noncommutative}, Theorem~2.30]}]\label{robalo}
    Let $S$ be a qcqs scheme. Then the functor
   $$h_S^{mot}(-):=L_{mot}h_S(-)\colon \mathrm{Sm}_S\to\mathcal{H}^{\mathbb{A}^1}(S)$$
    is universal among functors to a presentable $\infty$-category that are $\mathbb{A}^1$-invariant and send Nisnevich squares to pushout squares. In other words, the copresheaf
    \[
        \widehat{\mathrm{Cat}_{\infty}}^{\mathbb{A}^1,nis}(\mathrm{Sm}_S,-)\colon
        \mathrm{Pr}^L\to\widehat{\mathrm{Spc}}
    \]
    sending a presentable $\infty$-category $\mathcal{D}$ to the (large) full subcategory of $\widehat{\mathrm{Cat}_\infty}(\mathrm{Sm}_S,\mathcal{D})^{\simeq}$ consisting of functors satisfying the two properties above, is represented by
    $$h_S^{mot}(-)\colon \mathrm{Sm}_S\to\mathcal{H}^{\mathbb{A}^1}(S).$$
\end{inner-recall-star}

In the classical case of smooth manifolds, the analogous construction yields a well-known answer. Starting with the category of smooth manifolds, impose open-cover descent and $\mathbb{R}^1$-invariance, and you recover $\mathrm{Spc}$:

\begin{inner-recall-star}[\textup{[\cite{dugger1999sheaves}, Remark~3.4.10]}]\label{dugger}
    Let $\mathrm{SmMan}$ be the category of smooth (real) manifolds (without boundary). Then the copresheaf
    \[
        \widehat{\mathrm{Cat}_{\infty}}^{\mathbb{R}^1,open}(\mathrm{SmMan},-)\colon
        \mathrm{Pr}^L\to\widehat{\mathrm{Spc}}
    \]
    sending a presentable $\infty$-category $\mathcal{D}$ to the full subcategory of $\widehat{\mathrm{Cat}_\infty}(\mathrm{SmMan},\mathcal{D})^{\simeq}$ consisting of functors that contract the real line $\mathbb{R}^1$ and send twofold open-cover squares to pushout squares, is represented by the canonical embedding
    $$ h:\mathrm{SmMan}\hookrightarrow \mathrm{Spc}$$ sending a manifold to its underlying topological space.
\end{inner-recall-star}

This justifies the construction of the motivic homotopy category as the homotopical modification of algebraic geometry. There is, however, a stark difference between the two worlds: while the classical construction yields an $\infty$-topos, namely $\mathrm{Spc}$, the motivic homotopy category $\mathcal{H}^{\mathbb{A}^1}(S)$ of a non-empty qcqs scheme $S$ is never an $\infty$-topos [\cite{inhot}, Theorem 3.2.39]. This observation, though well known among experts, appeared explicitly for the first time in [\cite{spitzweck2012motivic}, Remark~3.2] as a remark. It was subsequently developed in the language of model categories in [\cite{MR3890956}], where it was established that the motivic homotopy category of a field is not a model topos. 

The problem is that not every discrete group object in the motivic homotopy category admits the internal Eilenberg-MacLane behavior enjoyed by group objects in an $\infty$-topos. Indeed, in \textit{loc.\ cit.}, the author proves that, when the base is the spectrum of a field, internal Eilenberg-MacLane objects arise only from strongly $\mathbb{A}^1$-invariant sheaves of groups. The point is then that not every $\mathbb{A}^1$-invariant sheaf of groups is strongly $\mathbb{A}^1$-invariant.

In the early days of motivic homotopy theory, this distinction was not entirely clear. The first such counterexample appeared in the work of Choudhury [\cite{choudhury2014connectivity}], where the author shows that the reduced free abelian group sheaf generated by $\mathbb{G}_m$ is $\mathbb{A}^1$-invariant but not strongly $\mathbb{A}^1$-invariant. This was developed in a more general setup in [\cite{inhot}, \S3.2], where the author shows that, while free groups generated by many discrete motivic spaces (for example, \'etale weakly unramified sheaves) are always motivic, they do not admit motivic deloopings unless the generating sheaf is birational. On the other hand, in \S3.2 of the present paper, we establish that birational sheaves of groups are not only strongly $\mathbb{A}^1$-invariant (\Cref{bir is strong a1}), but also admit birational deloopings (\Cref{B_bir is B}). These observations suggest that birational localization might retain, and perhaps even enhance, the topos-theoretic aspects of the motivic homotopy category.

Let us take a moment to investigate more closely where this failure stems from. A first objection could be that, while in the classical case we have considered the absolute situation of smooth manifolds over a point, in the motivic case we have so far stated only the relative version. One might then argue that a base scheme of positive dimension introduces additional homotopical complexity, and that this is what causes the problem. Unfortunately, the problems of the last two paragraphs are in fact shown over fields; i.e., even when the base is assumed to be a field, the motivic homotopy category need not be an $\infty$-topos.

The next question, then, is why the classical result (\Cref{dugger}) works. The key is that smooth manifolds are not merely locally diffeomorphic to $\mathbb{R}^n$; they are "very" locally diffeomorphic to $\mathbb{R}^n$. In other words, they admit good open covers, i.e., covers in which every nonempty finite intersection is contractible. Consequently, the \v{C}ech nerve of such a cover is made up of pieces that become contractible under $\mathbb{R}^1$-invariance. The existence of these covers, ultimately ensured by the existence of Riemannian metrics on smooth manifolds, is the key ingredient in the classical argument.

In $\mathrm{Sm}_k$, such a phenomenon does not occur. Although every smooth scheme admits an \'etale cover by schemes admitting \'etale maps to affine spaces $\mathbb{A}^n$, these covers need not be good in the analogous sense: their \v{C}ech nerves need not consist of $\mathbb{A}^1$-contractible schemes. Thus, even \'etale locally, smooth schemes can retain nontrivial $\mathbb{A}^1$-homotopical complexity. The problem is not dimension by itself, but the fact that, unlike in the classical situation, this homotopical dimension cannot in general be eliminated merely by passing to a suitable cover.

It seems rather difficult to find a topology in which such a dimension reduction can be achieved. Algebraic geometry, however, offers another way to approach the problem. Instead of searching for covers that eliminate the homotopical complexity, one can directly identify schemes with their generic points, which by default have $0$-homotopy dimension. This is precisely what birational localization is designed to do: it identifies dense open subschemes with the schemes containing them and, heuristically, reduces the problem to generic points. This suggests that birational localization might recover the missing topos-theoretic structure. 

In fact, throughout this paper, we will see that, over arbitrary qcqs base schemes, the birational homotopy category $\mathcal{H}^b(S)$ is very close to being an $\infty$-topos, in the sense that it enjoys many of the structural properties of an $\infty$-topos (\Cref{birational truncation}, \Cref{bir em}), and that it is indeed an $\infty$-topos when the base scheme is geometrically unibranch with finitely many generic points (for example, when the base is the spectrum of a field; see \Cref{Hbir a topos}), enjoying the above-mentioned $0$-homotopy dimension:

\begin{inner-recall-star}[Compare with \Cref{dugger}, \Cref{robalo}]\label[cor]{bir yoneda}
  Let $S$ be a qcqs, geometrically unibranch scheme with finitely many generic points. Then
 the `Birational Yoneda embedding' $$h^b_S(-):=L_{bir}h_S(-)\colon \mathrm{Sm}_S\to \mathcal{H}^{b}(S)$$
  contracts the affine line, sends Nisnevich (and in fact, cdh) squares to pushout squares, and has as its target an $\infty$-topos of homotopy dimension $0$.
\end{inner-recall-star}
Although it is not clear whether $\mathcal{H}^b$ is universal for this property, it is clear that it is the best-suited category for birational and weight-$0$ calculations. Having said that, it is also important to understand to what extent classical birational problems can be resolved entirely via $\mathbb{A}^1$-homotopy theory. This question has been raised and partially answered in various accounts. The starting point of this line of study is probably [\cite{asok2011smooth}], where the authors show that, for (smooth) proper varieties over $k$, the sections of the $\mathbb{A}^1$-homotopy path component over $\mathcal{F}_k$ capture Manin's notion of $R$-equivalent points. However, it is known that $\pi_0^{\mathbb{A}^1}X$ (even for a smooth proper variety $X$) is not birational [\cite{amit}, Example~4.8], so we are forced to consider the `universal birationalization' $\pi_0^{b\mathbb{A}^1}X$ [\cite{asok2011smooth}], [\cite{asok2010birational}], [\cite{asok2011stable}], [\cite{koizumi2021zeroth}]. In [\cite{choudhury2022characterisation}, Theorem~3.4] and [\cite{cisinski2023homotopy}, Proposition~1.9], $\pi_0^bX$ was identified with $\pi_0^{b\mathbb{A}^1}X$ for smooth proper varieties $X/k$, and this remains valid despite the $\mathcal{P}_\Sigma$ correction as we  show in \Cref{pi0bX=pi0ba1X}. In fact, in \Cref{pi0bX=pi0ba1X} we remove the smoothness hypothesis via an alternative universal-property approach (\Cref{pi0bX=pi0ba1X}).

This allows us to answer the question raised in the previous paragraph: we show that in the full subcategory $\mathcal{H}_{prop}^{\mathbb{A}^1}(k)\subset \mathcal{H}^{\mathbb{A}^1}(k)$, generated under small colimits by proper $k$-schemes (which we call the category of proper motivic spaces), several motivic and birational results coincide. For example, we show that, on $\mathcal{H}_{prop}^{\mathbb{A}^1}(k)$ over a perfect field $k$, $L_{bir}$ preserves and detects connectivity (\Cref{Hprop connectivity detection}, \Cref{proper schemes are a1 connected iff bir connected}) and that Morel's $0$-th $\mathbb{A}^1$-homology sheaf satisfies $\mathbb{H}_0^{\mathbb{A}^1}(\esc{X})\xrightarrow{\sim} \mathbb{Z}_\Sigma(\pi_0^b\esc{X})$ (\Cref{Ha1 is bir loc for prop}). This provides an affirmative answer to (a generalized version of) a question raised in [\cite{asok2011stable}, Remark~2.9], which was verified by explicit calculations for smooth proper schemes in [\cite{koizumi2021zeroth}, Theorem~4.5]; our proof is, however, completely self-contained. Moreover, this yields that, for proper motivic spaces $\esc{X}$, the map $\esc{X}\to \mathbb{H}_0^{\mathbb{A}^1}(\esc{X})$ is the universal morphism to a birational sheaf of abelian groups, identifying $ \mathbb{H}_0^{\mathbb{A}^1}(\esc{X})\simeq  \mathbb{H}_0^b(\esc{X})$ (\Cref{ha1 is universal}). It must, however, be noted that proper motivic spaces are not necessarily birational (consider the trivial example $L_{mot}\mathbb{P}^1$). Birational localization thus yields genuinely new birational invariants that remain invisible to the motivic localization, even on proper spaces.

\subsection{Summary of the paper}
Over a general qcqs base scheme $S$, let $B_0(S)$ denote the set of dense open immersions in $\mathrm{Sm}_S$. We define the \emph{Birational homotopy category} as $\mathcal{H}^b(S):=L_{dense}\mathcal{P}_{\Sigma}(\mathrm{Sm}_S)$ and the \emph{Birational Motivic homotopy category} as $\mathcal{H}^{b\mathbb{A}^1}(S):=L_{dense}\mathcal{H}^{\mathbb{A}^1}(S)$ (denoted $L^0_{bir}\mathcal{H}^{\mathbb{A}^1}(S)$ in [\cite{bachmann2019voevodsky}]), with localization functor $L_{bir}$. In \S2, we show that localization at $B_0(S)$ on $\mathcal{P}_\Sigma(\mathrm{Sm}_S)$ automatically enforces both Nisnevich descent and $\mathbb{A}^1$-invariance (both proofs are self-contained):

\begin{inner-recall-star}[see \Cref{Motivic equivalences are Birational equivalences}]
For any qcqs scheme $S$, $\mathcal{H}^{b\mathbb{A}^1}(S) = \mathcal{H}^b(S)$. 
\end{inner-recall-star}

Furthermore, if $\mathrm{log}\mbox{-}\mathcal{H}^{b}(S)$ denotes the localization of the logarithmic motivic homotopy category $\mathrm{log}\mbox{-}\mathcal{H}^\Box(S)$ [\cite{binda2023logarithmic}] at all dense open immersions of log schemes, we show in \S\ref{2.3} that:

\begin{inner-recall-star}[see \Cref{Equivalence of ordinary and logarithmic birational motivic homotopy categories}]
There is a canonical equivalence $\mathrm{log}\mbox{-}\mathcal{H}^{b}(S)\simeq \mathcal{H}^b(S)$ induced by the forgetful functor $\rho\colon \mathrm{SmlSm}_S\to \mathrm{Sm}_S$, $(X,D)\mapsto X$.
\end{inner-recall-star}

Using the simpler model $\mathcal{H}^b(S)=L_{dense}\mathcal{P}_\Sigma(S)$, we establish in \S3.1 that $\mathcal{H}^b(S)$ is closed under sifted colimits in $\mathcal{P}(S)$ (\Cref{bir is closed under sifted colimits}) and that over geometrically unibranch bases with finitely many generic points, $L_{bir}\simeq L_\Sigma L_{dense}$ preserves all small colimits (\Cref{form for lbir}, \Cref{closed under coproducts}), yielding our main $\infty$-topos theorem:

\begin{recall*}{Hbir a topos}
Let $S$ be a qcqs, geometrically unibranch scheme with finitely many generic points (for example, the spectrum of a field or any normal Noetherian scheme). Then $\mathcal{H}^{b\mathbb{A}^1}(S)$ is an $\infty$-topos. Moreover,
\begin{enumerate}
    \item $\mathcal{H}^{b\mathbb{A}^1}(S)$ has homotopy dimension $0$.
    \item $L_{bir}\colon\mathcal{P}_{nis}(S)\to \mathcal{H}^{b\mathbb{A}^1}(S)$ defines (the exceptional part of) an essential geometric morphism of $\infty$-topoi.
\end{enumerate}
\end{recall*}

In \S3.2, we prove that $L_{bir}$ is cartesian (\Cref{Lbir is cartesian}). Combining finite-product preservation with sifted-colimit preservation yields:

\begin{recall*}{bir bar commute}
Over an arbitrary qcqs scheme $S$, $L_{bir}$ commutes with the two-sided (Nisnevich) bar construction: $$L_{bir}\mathrm{B}_{nis}(-,-,-)\simeq L_{bir}\mathrm{B}(-,-,-)\simeq \mathrm{B}(L_{bir}-,L_{bir}-,L_{bir}-).$$
\end{recall*}

Together with \Cref{birational truncation}, this shows that even over an arbitrary qcqs base scheme $S$, $\mathcal{H}^b(S)$ enjoys key internal $\infty$-topos-theoretic properties that fail in $\mathcal{H}^{\mathbb{A}^1}(S)$:

\begin{inner-recall-star}[see \Cref{bir em} and \Cref{coh dim of bir}]
Over an arbitrary qcqs base scheme $S$, the loop functor $$\Omega\colon \mathcal{H}^{b}(S)_{\geq 1} \xrightarrow{} \mathrm{Grp}(\mathcal{H}^{b}(S))$$ is an equivalence with inverse $\mathrm{B}_{bir}\simeq \mathrm{B}|_{\mathrm{Grp}(\mathcal{H}^b(S))}$, and for all $n\geq 0$ there are canonical equivalences $$\mathrm{B}^n_{bir}:n\text{-}\mathrm{Grp}_S^b \simeq \mathcal{EM}_n(\mathcal{H}^b(S)): \pi_n^{bir}.$$ Here $\pi_n^{bir}$ is the internal homotopy group of the presentable category $\mathcal{H}^{b\mathbb{A}^1}(S)$ which happens to coincide  with $\pi_n|_{\mathcal{H}^{b}(S)}$.
\end{inner-recall-star}

In \S3.3, we show that over a perfect field $k$, a sheaf of groups $G$ is birational if and only if $G\in \mathrm{Grp}_k^{\mathbb{A}^1}$ and $G_{-1}=0$ (\Cref{gm contraction 0}). On connected motivic spaces, this identifies $L_{bir}$ with the $S^{2,1}$-nullification functor $L^{2,1}$ of [\cite{asok2023p}], for which we construct an explicit sequential colimit model (\Cref{double localization}, \Cref{model for L_bir on connected}):

\begin{recall*}{detecting n birationality by gm contraction}
Let $k$ be a perfect field. Then the following conditions for a Nisnevich-locally connected pointed motivic space $\esc{X}$ over $k$ are equivalent:
    \begin{enumerate}
        \item $\esc{X}$ is birational local.
        \item $\pi_i^{nis}\esc{X}$ is birational local for all $i$. 
        \item $(\pi_i^{nis}\esc{X})_{-1}=0$ for all $i$.
        \item $\esc{X}$ is $\mathbb{P}^1$ local (equivalently $S^{2,1}$ null \textup{[\cite{asok2023p}]}).
    \end{enumerate}
In other words, $\tau^{nis}_{\geq 1}\mathcal{H}^b(k)= \tau^{nis}_{\geq 1}\big(\mathcal{H}^{\mathbb{A}^1}(k)\mcap\mathcal{H}^{\mathbb{P}^1}(k)\big)$, and ${L_{bir}}_{\big|\tau_{\geq 1}^{nis}\mathcal{H}^{\mathbb{A}^1}(k)}\simeq {L^{2,1}}_{\big|\tau_{\geq 1}^{nis}\mathcal{H}^{\mathbb{A}^1}(k)}$.
\end{recall*}

In \S3.4, commutation of $L_{bir}$ with the bar construction immediately implies a higher connectivity theorem for $L_{bir}$ independent of the dimension of the base scheme:

\begin{inner-recall-star}[\Cref{bir preserves a1 conn}]
Let $S$ be the spectrum of a henselian local ring. Then for all $n\geq -2$, $L_{bir}$ takes Nisnevich locally $n$-connected (in fact, $L^{2,1}$-$n$-connected) $S$-spaces to sectionwise $n$-connected $S$-spaces. 
\end{inner-recall-star}
Although $\pi_0^{nis}\to\pi_0^b$ is not an (effective) epimorphism (\Cref{Lbir not effective coex}; unlike the ordinary motivic case [\cite{morel19991}, Corollary 2.3.22]), we show that when $S$ is the spectrum of an ample field $k$, then $L_{bir}$ detects $(-1)$-connectivity (\Cref{large-field-points}). In particular, every birationally connected space has a $k$-point.
 
In \S3.5, we study the interaction of $L_{bir}$ with effective epimorphisms from the internal-external perspective and establish the Birational van Kampen theorem (\Cref{bvk}) along with a local-to-global Birational-$n$-connectivity theorem (\Cref{local to global bir conn}). Dually, the interaction between internal and external $n$-truncated morphisms establishes a full-fledged Galois theory of birational covering, including the monodromy classification (\Cref{mono class}) and $n$-Galois correspondence (\Cref{galois cor}). In the same section, we also establish the Birational Freudenthal Suspension Theorem (\Cref{Birational Freudenthal suspension theorem}). These yield general comprehensive methods for computing new birational homotopy groups from old ones. In \S3.6, we provide direct algebraic ways of computing them:

\begin{inner-recall-star}[see \Cref{pi0bX=pi0ba1X}, \Cref{uni bir of pia1x}]\label{thm 12}
Let $k$ be a field. For any proper scheme $X/k$, there is a canonical isomorphism $\pi^{b}_0X\simeq\pi_0^{b\mathbb{A}^1}X$, so $\pi^{\mathbb{A}^1}_0X\to\pi_0^{b\mathbb{A}^1}X$ is initial among maps to a birational sheaf of sets, and $\pi_0^{b\mathbb{A}^1}$ is a stable birational invariant of smooth proper varieties.
\end{inner-recall-star}
\begin{inner-recall-star}[see \Cref{formula to compute piib}, \Cref{GW-non-birational}]
For any pointed $\mathbb{A}^1$-$n$-connected space $\mathcal{X}$ ($n \ge 1$), the following is a right exact sequence:
   $$(\pi_{n+1}^{\mathbb{A}^1}\mathcal{X})_{-1} \otimes^{\mathbb{A}^1} \underline{K}_1^{MW} \xrightarrow{\;\varepsilon\;} \pi_{n+1}^{\mathbb{A}^1}\mathcal{X}\longrightarrow \pi_{n+1}^b(\mathcal{X}) \to 0.$$
This yields, $\pi_0^b(\underline{K}^{MW}_0) = \underline{\mathbb{Z}}$ while $\pi_0^b(\underline{K}_n^{MW}) = 0$ for $n \ge 1$.
\end{inner-recall-star}

Applying \Cref{thm 12} to the full subcategory $\mathcal{H}_{prop}^{\mathbb{A}^1}(k)\subset \mathcal{H}^{\mathbb{A}^1}(k)$ generated under small colimits by proper schemes (with or without the smoothness hypotheses) over a field $k$ yields:

\begin{recall*}{proper schemes are a1 connected iff bir connected}
When $k$ is perfect, every object of $\mathcal{H}_{prop}^{\mathbb{A}^1}(k)$, and in particular every (ind-)proper scheme over $k$ is $\mathbb{A}^1$-connected if and only if it is birationally connected.
\end{recall*}
\begin{inner-recall-star}[see \Cref{Ha1 is bir loc for prop}]
For every $\esc{X}\in \mathcal{H}_{prop}^{\mathbb{A}^1}(k)$, there is a canonical isomorphism $\mathbb{H}_0^{\mathbb{A}^1}(\esc{X})\xrightarrow{\sim} \mathbb{Z}_\Sigma(\pi_0^b\esc{X})=\mathbb{H}^b_0(\esc{X})\in \mathrm{Ab}_k^b$, making $\esc{X}\to \mathbb{H}_0^{\mathbb{A}^1}(\esc{X})$ initial among maps to a birational sheaf of abelian groups. 
\end{inner-recall-star}

In \S4, we work over fields and conclude by contrasting birationally rigid schemes (\S4.1) with birationally contractible ones (\S4.2). We show that birationally rigid schemes are either $0$-dimensional in the smooth affine case or $\mathbb{A}^1$-rigid in the proper case (\Cref{dim 0 is bir rigid}, \Cref{A1 rigid proper iff bir rigid}). On the other hand, we prove that birationally connected Severi-Brauer varieties and smooth proper curves are necessarily isomorphic to $\mathbb{P}_k^n$ (\Cref{severi}, \Cref{bir conn curve}), while observing that, unlike in the $\mathbb{A}^1$-local case, birationally contractible (proper) surfaces are not unique. 

\subsection{Notations and terminologies}\label{1.2}

We shall freely use the language of $\infty$-categories without referring to any particular model for the theory; a standard reference is [\cite{lurie2009higher}]. With a choice of Grothendieck universes, the terms large and small will be treated as usual. The term $\mathrm{Map}_{-}(-,-)$ will always denote the $\infty$-groupoid of morphisms, while $[-,-]_-$ will denote the path component $\pi_0\mathrm{Map}_{-}(-,-)$. $\mathrm{Spc}$ (or $\widehat{\mathrm{Spc}}$) will stand for the $\infty$-category of (large) spaces, $\mathrm{Pr}^L$ for large presentable $\infty$-categories, and $\widehat{\mathrm{Cat}_{\infty}}$ for all large $\infty$-categories.

For an $\infty$-category $\mathcal{C}$, $\mathcal{P}(\mathcal{C})$ denotes the $\infty$-category of presheaves of spaces on $\mathcal{C}$. If $\mathcal{C}$ has finite coproducts, $\mathcal{P}_\Sigma(\mathcal{C})\subset \mathcal{P}(\mathcal{C})$ denotes the sifted cocompletion (presheaves turning finite coproducts into products), with localization $L_\Sigma$ and finite-coproduct-preserving Yoneda embedding $h\colon \mathcal{C}\to \mathcal{P}_\Sigma(\mathcal{C})$ [\cite{lurie2009higher}, Proposition~5.5.8.10]. For a presentable $\infty$-category $\mathcal{C}$, we write $\mathcal{C}_\bullet:=\mathcal{C}_{pt/}$ for pointed objects, $\mathcal{C}_{\leq n}$ (resp.\ $\mathcal{C}_{\geq n}$) for $n$-truncated (resp.\ $n$-connective) objects with truncation $\tau_{\leq n}^{\mathcal{C}}$ (resp.\ connective cover $\tau_{>n}^{\mathcal{C}}$) [\cite{lurie2009higher}, Proposition~5.5.6.5], $\mathrm{Disc}(\mathcal{C}):=\mathrm{N}(\mathcal{C}_{\leq 0})$, and $\mathcal{EM}_n(\mathcal{C}):=\mathcal{C}_{\leq n}\cap\mathcal{C}_{\geq n}$. We write $(\mathrm{C})\mathrm{Mon}(\mathcal{C})$ and $\mathrm{Grp}(\mathcal{C}):=\mathrm{Mon}(\mathcal{C})^{\mathrm{gp}}$ for (commutative) monoid and grouplike monoid objects in $\mathcal{C}$, and $\mathrm{B}$ for the bar construction (with $\mathrm{B}_\tau:=L_\tau\mathrm{B}$ for the $\tau$-local bar construction). Also, for a discrete cartesian category $\mathcal{C}$ $n\text{-}\mathrm{Grp}(\mathcal{C})$ collecctive denotes $\mathcal{C}_\bullet$ ($n=0$), $\mathrm{Grp}(\mathcal{C})$ ($n=1$), and $\mathrm{Ab}(\mathcal{C})$ ($n\geq 2$). 

Throughout, $S$ denotes a scheme, $k$ a field (not assumed perfect unless specified), all scheme morphisms are separated, and we replace $S$ by $k$ when $S=\operatorname{Spec} k$. Let $\mathrm{Sm}_S$ (resp.\ $\mathrm{SM}_S$) denote the category of smooth, finite-type (resp.\ all smooth) separated $S$-schemes. Since $\mathcal{P}_{nis}(\mathrm{SM}_S)\simeq \mathcal{P}_{nis}(\mathrm{Sm}_S)$ for qcqs $S$ [\cite{hoyois2015quadratic}, Proposition~C.5(3)], we work over $\mathrm{Sm}_S$ and call $\mathcal{P}_{nis}(S):=\mathcal{P}_{nis}(\mathrm{Sm}_S)$ the Nisnevich $\infty$-topos over $S$. Similarly, we write $\mathcal{P}(S):=\mathcal{P}(\mathrm{Sm}_S)$ and $\mathcal{P}_\Sigma(S):=\mathcal{P}_\Sigma(\mathrm{Sm}_S)$, with localization $L_{nis}$ onto $\mathcal{P}_{nis}(S)$, and denote sectionwise (resp.\ Nisnevich) homotopy groups and truncations by $\pi_i$, $\tau_{\leq n}$, $\tau_{>n}$ (resp.\ $\pi_i^{nis}$, $\tau_{\leq n}^{nis}$, $\tau_{>n}^{nis}$). Finally, the $\infty$-category of $S$-motivic spaces is denoted $\mathcal{H}^{\mathbb{A}^1}(S):=L_{\mathbb{A}^1}\mathcal{P}(S)\mcap \mathcal{P}_{nis}(S)$ (with localization $L_{mot}$, and $\pi_i^{\mathbb{A}^1}:=\pi_i^{nis}L_{mot}$).

We use $\mathrm{PSh}_S$, $\mathrm{PGrp}_S$, $\mathrm{PAb}_S$ (resp.\ $\mathrm{Sh}_S$, $\mathrm{Grp}_S$, $\mathrm{Ab}_S$) for presheaves (resp.\ Nisnevich sheaves) on $\mathrm{Sm}_S$ valued in sets, groups, and abelian groups; $\mathrm{Sh}^{\mathbb{A}^1}_S$, $\mathrm{Grp}_S^{\mathbb{A}^1}$, $\mathrm{Ab}_S^{\mathbb{A}^1}$ for ordinary, strongly, and strictly $\mathbb{A}^1$-invariant Nisnevich sheaves; and $\mathrm{Shv}^b_S$, $\mathrm{Grp}^b_S$, $\mathrm{Ab}^b_S$ for birational sheaves (\Cref{birational presheaves of sets}). Finally, $\mathbb{Z}_\Sigma$ denotes the free $\Sigma$-presheaf of abelian groups, $\mathbb{H}_0^{\mathbb{A}^1}$ the $0$-th $\mathbb{A}^1$-homology sheaf, $(-)_{-1}:=\Omega_{\mathbb{G}}$ the $\mathbb{G}_m$-contraction (here $\mathbb{G}:=(\mathbb{G}_m,1)$; similarly $\mathbb{P}:=(\mathbb{P}^1,\infty)$), $\mathcal{F}_k$ the category of finitely generated separable field extensions of $k$, and $\mathrm{Sch}_k$, $\mathrm{Prop}_k$, $\mathrm{SmProp}_k$ the categories of all, proper, and smooth proper $k$-schemes.
\section{Constructing the Birational Motivic Homotopy Category}

Let $S$ be a qcqs scheme. When the base scheme $S$ is fixed, we refer to $S$-spaces in $\mathcal{P}(S):=\mathcal{P}(\mathrm{Sm}_S)$ simply as `spaces' and implicitly identify any smooth scheme $X$ with its representable presheaf $h_S(X)$.
We say that an open immersion (see [\cite[\href{https://stacks.math.columbia.edu/tag/01HE}{Definition 01HE}]{stacks-project}]) $j:U\hookrightarrow X$ is a dense open immersion if the underlying topological space $j(U)$ is dense in $X$. We say that a presheaf $\esc{X}\in \mathcal{P}(S)$ is dense local if it is local with respect to the class $$B_0(S):=\{j:U\hookrightarrow X\in \mathrm{Sm}_S\mid j \text{ is a dense open immersion} \}$$ of dense open immersions in $\mathrm{Sm}_S$. We denote the full subcategory of $\mathcal{P}(S)$ consisting of dense local spaces by $L_{dense}\mathcal{P}(S)$. Moreover, since $B_0(S)$ is (essentially) small, it follows from [\cite{lurie2009higher}, Proposition 5.5.4.15] that $L_{dense}\mathcal{P}(S)$ is an accessible localization of $\mathcal{P}(\mathrm{Sm}_S)$. Let us denote the associated saturated class and the corresponding localization functor by $Dense(S)$ and $L_{dense}$, respectively.  
 
It is well known that the Yoneda embedding $h: \mathrm{Sm}_S\to \mathcal{P}(\mathrm{Sm}_S)$ does not, in general, preserve coproducts. Consequently, the $\infty$-category $\mathcal{P}(S):=\mathcal{P}(\mathrm{Sm}_S)$ contains many (geometrically) unidentified homotopy types. As a result, the dense local homotopy category $L_{dense}\mathcal{P}(S)$ is too naive to support meaningful birational geometry. To avoid introducing new homotopy types for existing coproducts, one should replace $\mathcal{P}(\mathrm{Sm}_S)$ with the category $\mathcal{P}_{\Sigma}(\mathrm{Sm}_S)$ [\cite{lurie2009higher}, Definition 5.5.8.8]. This is the full subcategory of $\mathcal{P}(\mathrm{Sm}_S)$ consisting of presheaves that send finite coproducts in $\mathrm{Sm}_S$ to products in $Spc$.

\begin{rem}
Since $\mathrm{Sm}_S$ is an extensive ($\infty$-)category (i.e., has disjoint coproducts), the category $\mathcal{P}_{\Sigma}(S):=\mathcal{P}_{\Sigma}(\mathrm{Sm}_S)$ is equivalent to the $\infty$-category $\mathcal{P}_{\sqcup}(\mathrm{Sm}_S)$ of sheaves of spaces on $\mathrm{Sm}_S$ with the disjoint cover topology (i.e., $\{U_\alpha\to X\}$ is a cover iff $\sqcup U_\alpha\to X$ is an isomorphism) [\cite{bachmann2017norms}, Lemma 2.4]. Thus, the localization functor $$L_\Sigma: \mathcal{P}(\mathrm{Sm}_S)\to \mathcal{P}_{\Sigma}(\mathrm{Sm}_S)$$ is left exact and accessible, making $\mathcal{P}_{\Sigma}(\mathrm{Sm}_S)$ an $\infty$-topos. Clearly, the Nisnevich topology is finer than the $\sqcup$-topology and hence $\mathcal{P}_{nis}(S)\subset \mathcal{P}_\Sigma(S)$.
\end{rem}

With this understanding, we are now ready to define the birational homotopy category:

\begin{defn}\label[defn]{bir local}
A presheaf $\esc{X}\in \mathcal{P}(S)$ is said to be birational local if it belongs to $\mathcal{P}_{\Sigma}(S)$ and is dense local. The full subcategory of $\mathcal{P}(S)$ consisting of birational local spaces is called the \textbf{Birational Homotopy category}, denoted $\mathcal{H}^b(S)$. Since $B_0(S)$ is an essentially small set, $\mathcal{H}^b(S) \subset \mathcal{P}(S)$ is an accessible localization. The associated saturated class is denoted $bir(S)$, and a morphism in it is called a \textit{Birational equivalence}.
\end{defn}
 
On the other hand, starting with the $\infty$-category $\mathcal{H}^{\mathbb{A}^1}(S)$ of $S$-motivic spaces [\cite{morel19991}], we define the \textbf{Birational Motivic homotopy category} as the localization of $\mathcal{H}^{\mathbb{A}^1}(S)$ at $B_0(S)$:
$$\mathcal{H}^{b\mathbb{A}^1}(S):=L_{dense}\mathcal{H}^{\mathbb{A}^1}(S)\equiv L^0_{bir}\mathcal{H}^{\mathbb{A}^1}(S)$$
(see [\cite{bachmann2019voevodsky}]), and denote the associated localization functor simply by $L_{bir}$. By construction, there is a canonical natural transformation $L_{mot}\to L_{bir}$ that is an $L_{bir}$-equivalence.
\begin{rem}
The dense localization of the motivic homotopy category is a strictly non-trivial localization. For example, it identifies the $\mathbb{A}^1$-local scheme $\mathbb{G}_m$ with the $\mathbb{A}^1$-contractible scheme $\mathbb{A}^1$ by the dense open immersion $\mathbb{G}_m\hookrightarrow \mathbb{A}^1$.
\end{rem}
Since $\mathcal{H}^{\mathbb{A}^1}(S)$, $\mathcal{P}_{nis}(S)$, and $\mathcal{P}_\Sigma(S)$ are presentable and $B_0(S)$ is essentially small, each member of the chain of inclusions $$\mathcal{H}^{b\mathbb{A}^1}(S)\subset L_{dense}\mathcal{P}_{nis}(S)\subset \mathcal{H}^b(S)$$ is a presentable $\infty$-category [\cite{lurie2009higher}, Proposition 5.5.4.15]. The aim of the next two subsections is to show that each of these inclusions is an equality.

\subsection{Nisnevich locality of Birational local spaces}

In this subsection, we aim to show that $\mathcal{H}^b(S)\subset\mathcal{P}_{nis}(S)$. That is, a birational local space, in the sense of \Cref{bir local}, is automatically Nisnevich local, so $\mathcal{H}^b(S):=L_{dense}\mathcal{P}_\Sigma(S)= L_{dense}\mathcal{P}_{nis}(S)$.
Before we do that, let us note something preliminary:
\begin{prop}\label[prop]{red et invariance}
Let $r:S^{red}\to S$ be the reduction map. Then $r^*: Et_{S}\to Et_{S^{red}}$ is an equivalence of categories.
\end{prop}
\begin{proof}
    Since $r$ is a homeomorphism, this follows from [\cite{grothendieck1966elements}, Th\'eor\`eme (18.1.2)].
\end{proof}
\begin{lem}\label[lem]{dense prod}
    Let $X/S$ be smooth. Then, for a dense open immersion $V\hookrightarrow Y$ in $\mathrm{Sm}_S$, the pullback morphism $X\times _SV\to X\times _SY$ is also a dense open immersion in $\mathrm{Sm}_S$.
\end{lem}
\begin{proof}
    Openness is standard. For denseness, let $W\subset X\times_S Y$ be an open set. We have to show that $W\mcap (X\times _SV)\neq \emptyset$. Since smooth maps are (universally) open, the image $f(W)$ of $W$ under the smooth map $f:X\times_S Y\to Y$ is open in $Y$ and hence intersects the dense open subset $V$ of $Y$, say at $y\in V\mcap f(W)$. A preimage $t\in W$ of $y$ must therefore be in $X\times_S V$.
\end{proof}
Here is the main theorem of this subsection:
\begin{thm}\label{bir local is nis local}
    For any qcqs scheme $S$, one has $\mathcal{H}^b(S)\subset\mathcal{P}_{nis}(S)$. Therefore, $L_{dense}\mathcal{P}_{nis}(S)= \mathcal{H}^b(S)$.
\end{thm}
\begin{proof}
      Since our base scheme is qcqs (i.e., coherent), by [\cite{hoyois2015quadratic}, Proposition C.5 (3), (4)], a presheaf on $\mathrm{Sm}_S$ is a Nisnevich sheaf if and only if it satisfies Nisnevich excision. By definition, $\esc{X}\in \mathcal{P}(\mathrm{Sm}_S)$ satisfies Nisnevich excision if it is contractible over the empty scheme and takes Nisnevich cd squares to pullback squares of spaces. So let $\esc{X}$ be a birational local space, i.e., $\esc{X}\in L_{dense}L_{\Sigma}\mathcal{P}(\mathrm{Sm}_S)$. Since $\esc{X}\in \mathcal{P}_\Sigma(\mathrm{Sm}_S)$, we have $\esc{X}(\emptyset)\simeq pt$ for free. For the excision part, consider a Nisnevich cd square:
    \[
    \xymatrix{
    V\ar@{^(->}[r]\ar[d]_{q}&Y\ar[d]^{p}\\
    U\ar@{^(->}[r]_j& X
    }
    \]
    where $j$ is an open subscheme and $p$ is \'etale. Consider the open subscheme $W=X\setminus \bar{U}$. Clearly, $U\sqcup W\xhookrightarrow{}X$ is a dense open immersion. Since $p_{| Y\setminus V}: Y\setminus V \to X\setminus U$ is an isomorphism of reduced schemes and $W\subset X\setminus U$, the pullback $$p_{|p^{-1}W}^{red}:p^{-1}W^{red} \to W^{red}$$ is an isomorphism. Since both $W$ and $p^{-1}W$ are \'etale over $X$, by \Cref{red et invariance} we deduce that the original map $p_{|p^{-1}W}^{}:p^{-1}W^{} \to W^{}$ is also an isomorphism. On the other hand, since $U\sqcup W\xhookrightarrow{}X$ is dense, by \Cref{dense prod} so is $V\sqcup p^{-1}W\hookrightarrow Y$. Now, let us consider the following commutative diagram:
    \[
    \xymatrix{
    \esc{X}(U)\times \esc{X}(W)\ar[d]^{q^*\times p^*_{|p^{-1}W}}& \esc{X}(U\sqcup W)\ar[l]^{\simeq}&\esc{X}(X)\ar[l]^-{\simeq}\ar[r] \ar[d]&\esc{X}(U)\ar[d]^{q^*}\\
    \esc{X}(V)\times \esc{X}(p^{-1}W)& \esc{X}(V\sqcup p^{-1}W)\ar[l]^-{\simeq}&\esc{X}(Y)\ar[l]^-{\simeq}\ar[r] &\esc{X}(V)}
    \]
The horizontal arrows on the left are equivalences because $\esc{X}\in L_{\Sigma}\mathcal{P}(\mathrm{Sm}_S)$. The horizontal arrows in the middle are equivalences because $\esc{X}$ inverts dense open immersions. Since $p_{|p^{-1}W}$ is an isomorphism, the entire outer diagram (with the left and middle horizontal arrows replaced by their inverses) is tautologically a pullback square. Thus, the square on the extreme right (being equivalent to the whole outer diagram) is a pullback square itself.
\end{proof}
\begin{rem}\label[rem]{correction}
Alternatively, one can show directly that Nisnevich-local equivalences over qcqs schemes are birational in our sense. This is established in [\cite{sfbat}, Theorem 2.3.5]. The argument there rectifies a slight inaccuracy in the formulation of the analogous result over fields, proven in [\cite{choudhury2022characterisation}, Lemma 3.6] (see also [\cite[Theorem 3.9]{choudhury2022characterisation}] and [\cite[Theorem 2.3]{cisinski2023homotopy}]). Because the framework in [\cite{choudhury2022characterisation}] and [\cite{cisinski2023homotopy}] relies on the naive dense local homotopy category (i.e., their birational category $\mathbf{H}_b(k)$ or $\mathscr{H}_b(k)$ is merely $L_{dense}\mathcal{P}(k)$), the \textit{birational weak equivalences} they consider are, in fact, just the dense local equivalences in our sense. However, for a nontrivial decomposition $A\sqcup B=X$ (which is a Nisnevich cover) of smooth schemes, the \v{C}ech nerve of the presheaf map $$h_SA\sqcup h_SB\to h_SX$$ need neither be a sectionwise equivalence nor a dense local equivalence (for example, the empty presheaf is clearly dense-local, but it is not a Nisnevich sheaf). In [\cite{sfbat}, Theorem 2.3.4] we address this discrepancy by providing a modified proof while simultaneously working over an arbitrary qcqs base scheme.
\end{rem} 

\begin{rem}\label[rem]{bir is cdh}
 Using similar techniques, one can show that birational local spaces are also (smooth) cdp local. Therefore, birational local spaces are scdh local. 
\end{rem}
We conclude this subsection with a discussion of the implications of the above theorem. First, we present a generalization of a well-known definition (see [\cite{asok2011smooth}, Definition 6.1.1]): 
\begin{defn}\label[defn]{birational presheaves of sets}
A presheaf of sets $F\in \mathrm{PSh}_S$ is called \emph{dense-local} if it is $B_0(S)$-local in $\mathrm{PSh}_S$, and \emph{birational} if it moreover lies in $\mathrm{PSh}_\Sigma(S)$. Equivalently, birational presheaves of sets are precisely the discrete objects of $\mathcal{H}^b(S)$.
\end{defn}

\begin{rem}\label[rem]{birational presheavesa re Nisnevich sheaves}
By \Cref{bir local is nis local}, over any qcqs scheme $S$, every birational presheaf is automatically a Nisnevich sheaf. We denote by $\mathrm{Shv}^b_S$, $\mathrm{Grp}^b_S$, and $\mathrm{Ab}^b_S$ the categories of birational sheaves on $\mathrm{Sm}_S$ with values in sets, groups, and abelian groups, respectively.
\end{rem}

\begin{prop}\label[prop]{pi is birational}
Let $S$ be a qcqs scheme. For a presheaf of spaces $\esc{X}\in \mathcal{P}(S)$, the following conditions are equivalent:
\begin{enumerate}
    \item $\esc{X}$ is birational-local.
    \item $\pi_i\esc{X}$ is a birational sheaf for every $i\geq 0$ (and every choice of basepoint).
    \item $\tau_{\leq i}\esc{X}$ is birational-local for every $i\geq 0$.
\end{enumerate}
\end{prop}
\begin{proof}
Since dense-locality is tested sectionwise and $\pi_i$ preserves finite products, $\esc{X}$ lies in $L_{dense}\mathcal{P}(S)$ (resp.\ $\mathcal{P}_\Sigma(S)$) if and only if $\pi_i\esc{X}$ is dense-local (resp.\ lies in $\mathrm{PSh}_\Sigma(S)$) for all $i\geq 0$. This gives (1)$\iff$(2). Since $\pi_j(\tau_{\leq i}\esc{X})\cong \pi_j\esc{X}$ for $j\leq i$ and vanishes for $j>i$, (2)$\iff$(3) follows immediately from (1)$\iff$(2).
\end{proof}

By \Cref{birational presheavesa re Nisnevich sheaves}, when $\esc{X}\in \mathcal{H}^b(S)$, the presheaves $\pi_i\esc{X}$ are automatically Nisnevich sheaves. In particular, the birational analogue of Morel's theorem (that $\mathbb{A}^1$-homotopy group sheaves of motivic spaces are $\mathbb{A}^1$-invariant) is immediate and holds even for $\pi_0$:

\begin{cor}\label[cor]{nis sheaves of bir is bir}
Let $\esc{X}\in \mathcal{H}^b(S)$ be a birational-local space over a qcqs scheme $S$. Then for all $i\geq 0$, $\pi_i^{nis}\esc{X}=\pi_i\esc{X}$, and the Nisnevich sheaf $\pi_i^{nis}\esc{X}$ is birational.
\end{cor}
\begin{proof}
By \Cref{pi is birational}, each $\pi_i\esc{X}$ is birational and hence already a Nisnevich sheaf (\Cref{bir local is nis local}), so $\pi_i\esc{X}\xrightarrow{\sim}\pi_i^{nis}\esc{X}$.
\end{proof}

\begin{rem}
The analogous statement for $\mathcal{H}^{\mathbb{A}^1}(S)$ is false: $\pi_0^{nis}$ of an $S$-motivic space need not be $\mathbb{A}^1$-invariant. This was a long-standing conjecture of Morel [\cite{morel2012a1}, Conjecture~1.12], recently disproved by Ayoub [\cite{ayoubcounterexamples}].
\end{rem}

Denote by $\tau_{\leq n}^{bir}$ and $\pi_i^{bir}:=\tau_{\leq 0}^{bir}\Omega^i$ the internal truncation functors [\cite{lurie2009higher}, Proposition~5.5.6.5] and internal homotopy group functors [\cite{inhot}, \S2.1] of the presentable $\infty$-category $\mathcal{H}^{b}(S)$. Because $\tau_{\leq n}$ preserves $\mathcal{H}^{b}(S)$ (\Cref{pi is birational}), internal truncations in $\mathcal{H}^b(S)$ coincide with sectionwise and Nisnevich truncations:

\begin{thm}\label{birational truncation}
Let $S$ be a qcqs scheme.
\begin{enumerate}
    \item There are canonical identifications of truncation and homotopy group functors
    \[
    \tau^{bir}_{\leq n} = \tau_{\leq n}|_{\mathcal{H}^{b}(S)} = \tau^{\sqcup}_{\leq n}|_{\mathcal{H}^{b}(S)} = \tau^{nis}_{\leq n}|_{\mathcal{H}^{b}(S)} \quad\text{and}\quad \pi_i^{bir} = \pi_i|_{\mathcal{H}^{b}(S)} = \pi_i^{nis}|_{\mathcal{H}^{b}(S)}.
    \]
    for all $n\geq -2$ and $i\geq 0$.
    \item In particular, $\mathrm{Disc}(\mathcal{H}^b(S)) = \mathrm{Shv}^b_S$, the induced localization on discrete objects $L_{\leq 0}$ coincides with $\pi_0^b:=\pi_0L_{bir}$, and $L_{bir}$ is path-injective in the sense of \textup{[\cite{inhot}, Definition~2.2.7]}. 
    \item The birational homotopy category $\mathcal{H}^b(S)$ is Postnikov complete.
\end{enumerate}
\end{thm}
\begin{proof}
Since $\tau_{\leq i} $ preserves products and hence objects of $\mathcal{P}_\Sigma(S)$ (and any sectionwise functor preserves dense locality), for any $\esc{X}\in\mathcal{H}^{b}(S)$, $\tau_{\leq n}\esc{X}$ is birational-local for all $n\geq -2$ (note that this is a strengthening of the ((1)$\implies$(3)) part of \Cref{pi is birational}). Thus \Cref{bir local is nis local} yields, $$\tau_{\leq n}\esc{X}\simeq \tau_{\leq n}^{\sqcup}\esc{X}\simeq \tau_{\leq n}^{nis}\esc{X}\in \mathcal{H}^b(S)_{\leq n} = \mathcal{H}^b(S)\mcap \mathcal{P}(S)_{\leq n}.$$ Now, for any $\esc{Y}\in\mathcal{H}^{b}(S)_{\leq n}$, we have
\begin{flalign*}
\mathrm{Map}_{\mathcal{H}^{b}(S)_{\leq n}}(\tau^{bir}_{\leq n}\esc{X}, \esc{Y}) \simeq \mathrm{Map}_{\mathcal{H}^{b}(S)}(\esc{X}, \esc{Y}) = \mathrm{Map}_{\mathcal{P}(S)}(\esc{X}, \esc{Y}) &\simeq \mathrm{Map}_{\mathcal{P}(S)_{\leq n}}(\tau_{\leq n}\esc{X}, \esc{Y})\\& = \mathrm{Map}_{\mathcal{H}^{b}(S)_{\leq n}}(\tau_{\leq n}\esc{X}, \esc{Y}),  
\end{flalign*}
which proves (1) and (2) by Yoneda's lemma. For (3), since $\mathcal{H}^{b}(S)\subset \mathcal{P}(S)$ is a reflective localization, it is closed under small limits in $\mathcal{P}(S)$, and by (1) it is closed under $\tau_{\leq n}$; thus Postnikov completeness of $\mathcal{P}(S)$ passes directly to $\mathcal{H}^{b}(S)$.
\end{proof}

\subsection{\texorpdfstring{$\mathbb{A}^1$}{A1}-invariance and rationality of Birational local spaces}
The purpose of this subsection is to show that birational local spaces are $\mathbb{A}^1$-local (and, more generally, local with respect to any rational scheme $A\in \mathrm{Sm}_S$, i.e., any scheme $A/S$ admitting a span of dense open immersions $A\hookleftarrow U \hookrightarrow \mathbb{A}^n_S$ over $S$), and hence $\mathcal{H}^{b\mathbb{A}^1}(S)=\mathcal{H}^b(S)$. This is well known over fields, and appears in [\cite{MR3404383}, Appendix~A], where the authors first deduce $\mathbb{P}^1$-locality and restrict along $\mathbb{A}^1_S\hookrightarrow \mathbb{P}^1_S$ to deduce the $\mathbb{A}^1$-local case. The method in \textit{loc.\ cit.} passes through normalizing the birational isomorphism $\mathbb{P}^1\times \mathbb{P}^1\dashrightarrow \mathbb{P}^2$ via a blow-up construction (see the author's [\cite{sfbat}, Lemma~2.3.8] for a detailed explanation); however, we shall present here an entirely different proof that proceeds via a simple slope-and-ray technique: namely, we connect the two rational points $0,1$ of the affine line through the slope function (well-defined up to rational modification) on the affine plane. This does not use any blow-up construction and works simultaneously over all base schemes.

\begin{lem}\label[lem]{bir local is rational local}
    Let $S$ be a scheme and $A\in \mathrm{Sm}_S$ a rational scheme over $S$ (for instance, $A=\mathbb{A}^1_S$ or $A=\mathbb{P}^1_S$). Then a dense local presheaf of sets $F$ over $S$ is $A$-local (meaning $F(X)\to F(A\times_S X)$ is a bijection for all $X\in \mathrm{Sm}_S$). 
\end{lem} 
\begin{proof}
Let $F$ be a dense local presheaf of sets over $S$, and let $X\in \mathrm{Sm}_S$. We first show that $F$ is $\mathbb{A}^1$-local, i.e., that the projection $p:\mathbb{A}^1_X\to X$ induces a bijection $p^*:F(X)\to F(\mathbb{A}^1_X)$. Since $p\circ i_0=\mathrm{id}_X$, it suffices to show that $i_0^*=i_1^*:F(\mathbb{A}^1_X)\to F(X)$ for every $X\in \mathrm{Sm}_S$; indeed, applying this equality to the multiplication morphism $m:\mathbb{A}^1_X\times_X \mathbb{A}^1_X\to \mathbb{A}^1_X$, $(s,t)\mapsto st$, then yields $\mathrm{id}_{F(\mathbb{A}^1_X)}=({i_1}_{{\mathbb{A}^1_X}})^*m^*=({i_0}_{{\mathbb{A}^1_X}})^*m^*=(m\circ {i_0}_{{\mathbb{A}^1_X}})^*=(i_0p)^*=p^*i_0^*$.

To prove $i_0^*=i_1^*$, consider the dense open immersions $j:\mathbb{G}_{m,X}\hookrightarrow \mathbb{A}^1_X$ and $k=\mathrm{id}_{\mathbb{A}^1_X}\times j:\mathbb{A}^1_X\times_X \mathbb{G}_{m,X}\hookrightarrow \mathbb{A}^2_X$ (see \Cref{dense prod}), together with the slope morphism $\Phi:\mathbb{A}^1_X\times_X \mathbb{G}_{m,X}\to \mathbb{A}^1_X$ given by $(x,y)\mapsto xy^{-1}$. For each $c\in \{0,1\}$, let $L_c:\mathbb{A}^1_X\to \mathbb{A}^2_X$ be the line $y\mapsto (cy,y)$ and let $\tilde{L}_c:\mathbb{G}_{m,X}\to \mathbb{A}^1_X\times_X \mathbb{G}_{m,X}$ be its restriction, fitting into the commutative squares
\[\xymatrix{ X \ar[d]_{i_c} & \mathbb{G}_{m,X} \ar[l]_-{p\circ j} \ar@{^(->}[r]^-{j} \ar[d]^{\tilde{L}_c} & \mathbb{A}^1_X \ar[d]^{L_c} \\ \mathbb{A}^1_X & \mathbb{A}^1_X\times_X \mathbb{G}_{m,X} \ar[l]^-{\Phi} \ar@{^(->}[r]^-{k} & \mathbb{A}^2_X }\]
Given $\alpha\in F(\mathbb{A}^1_X)$, since $k$ is a dense open immersion, there is a unique $\widetilde{\alpha}\in F(\mathbb{A}^2_X)$ such that $k^*(\widetilde{\alpha})=\Phi^*(\alpha)$. The commutative squares above yield
$$j^*\big(L_c^*(\widetilde{\alpha})\big)=\tilde{L}_c^*k^*(\widetilde{\alpha})=\tilde{L}_c^*\Phi^*(\alpha)=j^*\big(p^*i_c^*(\alpha)\big) \quad \text{in } F(\mathbb{G}_{m,X}).$$
Since $j^*$ is a bijection, we obtain $L_c^*(\widetilde{\alpha})=p^*i_c^*(\alpha)$ in $F(\mathbb{A}^1_X)$ for $c\in \{0,1\}$. Because $L_0$ and $L_1$ intersect at the origin ($L_0\circ i_0=(0,0)=L_1\circ i_0:X\to \mathbb{A}^2_X$) and $p\circ i_0=\mathrm{id}_X$, applying $i_0^*$ gives
$$i_0^*(\alpha)=i_0^*p^*i_0^*(\alpha)=i_0^*L_0^*(\widetilde{\alpha})=i_0^*L_1^*(\widetilde{\alpha})=i_0^*p^*i_1^*(\alpha)=i_1^*(\alpha).$$
Thus $F$ is $\mathbb{A}^1$-local, and hence $\mathbb{A}^n$-local for all $n\geq 1$. Finally, for any rational scheme $A/S$, the span of dense open immersions $A\hookleftarrow U\hookrightarrow \mathbb{A}^n_S$ remains a span of dense open immersions after base change by $X\in \mathrm{Sm}_S$ (\Cref{dense prod}), inducing $F(A\times_S X)\xrightarrow{\sim} F(U\times_S X)\xleftarrow{\sim} F(\mathbb{A}^n_S\times_S X)\xleftarrow{\sim} F(X)$. 
\end{proof}

\begin{thm}\label{Motivic equivalences are Birational equivalences}
Let $S$ be a qcqs scheme. Then:
\begin{enumerate}
    \item Every birational presheaf of sets on $\mathrm{Sm}_S$ is an $\mathbb{A}^1$-invariant Nisnevich sheaf.
    \item For every rational scheme $A/S$ (in particular, for $A=\mathbb{A}^1_S$ and $A=\mathbb{P}^1_S$), every birational local space $\esc{E}\in \mathcal{H}^b(S)$ is $A$-local. Thus $\mathcal{H}^b(S)\subset \mathcal{H}^{\mathbb{A}^1}(S)$, or, equivalently, motivic equivalences are birational equivalences.
    \item $\mathcal{H}^{b\mathbb{A}^1}(S)= \mathcal{H}^b(S)$.
\end{enumerate}
\end{thm}
\begin{proof}
Statement (1) combines \Cref{bir local is rational local} and \Cref{bir local is nis local}. Since $A$-locality of a space $\esc{E}\in \mathcal{P}(\mathrm{Sm}_S)$ is tested on the presheaves $\pi_i(\esc{E})$, \Cref{bir local is rational local} and \Cref{bir local is nis local} immediately yield (2). Statement (3) then follows from (2):
\begin{align*}\mathcal{H}^{b\mathbb{A}^1}(S)=\mathcal{H}^{\mathbb{A}^1}(S)\mcap L_{dense}\mathcal{P}(S)&=\mathcal{H}^{\mathbb{A}^1}(S)\mcap\mathcal{P}_{\Sigma}(S)\mcap L_{dense}\mathcal{P}(S)\\&=\mathcal{H}^{\mathbb{A}^1}(S)\mcap \mathcal{H}^b(S)=\mathcal{H}^b(S).
\end{align*}
\end{proof}
\begin{rem}\label[rem]{bir contr but not rational}
    The birational localization is strictly stronger than localizing at rational schemes: when $S=Spec k$, there is a stably rational smooth $k$-variety that is not rational [\cite{asok2011smooth}, Example 2.3.4], yet stably rational varieties are clearly birationally contractible. We must, however, note that the example in loc. cit. is that of a threefold. For proper schemes of dimension $\leq 2$, however, no such counterexample exists over an algebraically closed field $k$ of characteristic zero, i.e., all birationally connected smooth proper varieties of dim $\leq 2$ are necessarily rational. Indeed, by \Cref{proper schemes are a1 connected iff bir connected}, any birationally connected proper surface is $\mathbb{A}^1$-connected, and in characteristic zero every $\mathbb{A}^1$-connected scheme of dimension $\leq 2$ is rational [\cite{asok2011smooth}, Corollary~2.4.7].
\end{rem}
\begin{rem}
     While $\mathbb{A}^1_S\to S$ is a dense equivalence over $S$, a dense open immersion (in fact, any open immersion) $j:U\hookrightarrow S$ into $S$ is \'etale, so it can never be a motivic equivalence unless it is an isomorphism [\cite{A1contractibledubouloz}, Proposition 3.1].
\end{rem}

\subsection{Logarithmic Birational motives}\label{2.3}
The logarithmic version of motivic homotopy theory has been extensively studied by Binda, Park, and \O stv\ae r in a series of papers starting with [\cite{binda2023logarithmic}], with the goal of representing more cohomology theories that are not necessarily $\mathbb{A}^1$-invariant (such as topological Hochschild homology).

In the previous subsection, we saw that $\mathcal{H}^b(S)$ is a full subcategory of $\mathcal{H}^{\mathbb{A}^1}(S)$, as well as of $\mathcal{H}^{\mathbb{P}^1}(S)$ and $\mathcal{H}^{\mathbb{P}^\bullet}(S)$ (\Cref{Motivic equivalences are Birational equivalences}; see [\cite{binda2023logarithmic}, Definition~4.0.3] for notation). This suggests a natural connection between the birational homotopy category and the logarithmic motivic homotopy category $\mathrm{log}\mbox{-}\mathcal{H}^\Box(S)$ over a qcqs scheme $S$.

Consider the forgetful functor $\rho\colon \mathrm{SmlSm}_S\to \mathrm{Sm}_S$ defined by $\rho (\bar{X}, D):= \bar{X}$, which admits a fully faithful right adjoint $\lambda\colon \mathrm{Sm}_S\to \mathrm{SmlSm}_S$ defined by $\lambda (X):=(X,\emptyset)$, and let $\omega\colon \mathrm{SmlSm}_S\to \mathrm{Sm}_S$, $\omega(\bar{X},D):=\bar{X}\setminus D$, denote the right adjoint of $\lambda$ (see [\cite{binda2023logarithmic}, \S4]). Equipped with the strict logarithmic Nisnevich topology, $\rho$ is neither a morphism of sites nor does it preserve the interval object, so it does not directly admit motivic left derived functors (though $\lambda$ does). Nevertheless, after birational localization we obtain:

\begin{thm}\label{roh preserves bir}
The induced functor $\rho_\sharp\colon \mathcal{P}(\mathrm{SmlSm}_S)\to \mathcal{P}(S)$ sends log motivic equivalences to birational equivalences. Hence, it induces an adjunction
\[
L_{bir}\rho_\sharp\colon \mathrm{log}\mbox{-}\mathcal{H}^{\Box}(S)\;\leftrightarrows\; \mathcal{H}^{b}(S) : \rho ^*\simeq \lambda_{\sharp}
\]
whose right adjoint $\rho^*\colon\mathcal{H}^b(S)\to\mathrm{log}\mbox{-}\mathcal{H}^{\Box}(S)$ is fully faithful.
\end{thm}
\begin{proof}
Local equivalences in $\mathcal{P}(\mathrm{SmlSm}_S)$ given by strict Nisnevich covers are sent by $\rho_\sharp$ to Nisnevich covers of the underlying total schemes, which are birational equivalences by \Cref{bir local is nis local}. Next, the underlying morphism of total schemes of an admissible blow-up [\cite{binda2023logarithmic}, Definition~2.5.2, Construction~2.5.9] is an ordinary blow-up (at a smooth center) of schemes, which is a birational isomorphism and hence a birational equivalence. Finally, $\rho (\bar{\Box}\times X)=\mathbb{P}^1_X$, so the projection maps $\bar{\Box}\times X\to X$ are sent to $\mathbb{P}^1_X\to X$, which are birational equivalences by \Cref{Motivic equivalences are Birational equivalences}. Since $\lambda$ is fully faithful at the level of schemes, the right derived functor $\lambda_\sharp\simeq \rho ^*$ is fully faithful.
\end{proof}

To describe the essential image of $\rho^*$, we call a morphism of log schemes $\underline{U}\to \underline{X}$ in $\mathrm{SmlSm}_S$ a \emph{dense open immersion} if the underlying morphism of total schemes $\bar{U}\to \bar{X}$ is a dense open immersion of schemes and the log structure on $\underline{U}$ is the pullback of that on $\underline{X}$. Denote by $\mathrm{log}\mbox{-}\mathcal{H}^b(S)$ the localization of $\mathrm{log}\mbox{-}\mathcal{H}^\Box(S)$ with respect to the class of dense open immersions of log schemes.

\begin{thm}\label[cor]{Equivalence of ordinary and logarithmic birational motivic homotopy categories}
Both $\rho_\sharp\colon \mathcal{P}(\mathrm{SmlSm}_S)\to \mathcal{P}(S)$ and $\lambda_\sharp\colon \mathcal{P}(S)\to \mathcal{P}(\mathrm{SmlSm}_S)$ preserve birational equivalences. Consequently, $\rho^*$ factors through $\mathrm{log}\mbox{-}\mathcal{H}^b(S)$, and the induced adjunction
\[
\rho^* \simeq \lambda_{\sharp}\colon \mathcal{H}^b(S)\;\leftrightarrows\; \mathrm{log}\mbox{-}\mathcal{H}^b(S) : \lambda ^*\simeq \omega _\sharp
\]
is an equivalence of $\infty$-categories. In particular, a log motivic space $\esc{X}\in \mathrm{log}\mbox{-}\mathcal{H}^{\Box}(S)$ lies in the essential image of $\rho ^*$ if and only if it is dense-local in $\mathcal{P}(\mathrm{SmlSm}_S)$.
\end{thm}
\begin{proof}
That $\rho_\sharp$ preserves birational equivalences is immediate from the definition of dense open immersions of log schemes and \Cref{roh preserves bir}. Similarly, since birational equivalences in $\mathcal{P}(S)$ are generated by Nisnevich squares and dense open immersions (\Cref{bir local is nis local}), and $\lambda$ preserves both strict Nisnevich squares [\cite{binda2023logarithmic}, Construction~2.5.9] and dense open immersions by definition, $\lambda_\sharp$ preserves birational equivalences. Thus $L_{log\text{-}bir}\lambda_\sharp\simeq \lambda_\sharp$ on $\mathcal{H}^b(S)$, and $\rho^*\simeq \lambda_\sharp$ factors through $\mathrm{log}\mbox{-}\mathcal{H}^b(S)$.

Since $\lambda_\sharp$ is fully faithful, it remains to show that the counit $\lambda_\sharp \omega_\sharp \to 1$ is an equivalence in $\mathrm{log}\mbox{-}\mathcal{H}^b(S)$. Because both functors preserve colimits at the birational level, it suffices to check this on generators $X=(\bar{X},\partial X)\in \mathrm{SmlSm}_S$, where $\lambda\omega(X)=(\bar{X}\setminus \partial X,\emptyset)\to (\bar{X},\partial X)$ is a dense open immersion of log schemes by definition.
\end{proof}

We caution the reader that there is a weaker notion of log-dense maps: a dense open immersion $(U,D_U)\to (X,D_X)$ of log schemes is called \emph{log-dense} if $X\setminus D_X\subset U$. Denoting by $\mathrm{log}\mbox{-}\mathcal{H}^{log\mbox{-}b}(S)$ the localization of $\mathrm{log}\mbox{-}\mathcal{H}^\Box(S)$ at the log-dense open immersions, we have an inclusion $\mathrm{log}\mbox{-}\mathcal{H}^{b}(S)\subset \mathrm{log}\mbox{-}\mathcal{H}^{log\mbox{-}b}(S)$, which is the localization at the set of interior immersions
\[
int(S):=\{ X^* := (X \setminus D, \emptyset) \longrightarrow (X, D)\}_{(X,D)\in \mathrm{SmlSm}_S}.
\]

\begin{prop}
The inclusion $\mathrm{log}\mbox{-}\mathcal{H}^{b}(S)\subset \mathrm{log}\mbox{-}\mathcal{H}^{log\mbox{-}b}(S)$ is the localization with respect to $int(S)$, i.e., $\mathrm{log}\mbox{-}\mathcal{H}^{b}(S)=L_{int} \mathrm{log}\mbox{-}\mathcal{H}^{log\mbox{-}b}(S)$, fitting into a commutative diagram in $\mathrm{Cat}_\infty$:
\[
\xymatrix{
\mathcal{H}^b(S)\ar[r]^-{\sim}\ar@{^(->}[d]&L_{int}\mathrm{log}\mbox{-}\mathcal{H}^{log\mbox{-}b}(S)\ar@{^(->}[r]\ar@{^(->}[d]&\mathrm{log}\mbox{-}\mathcal{H}^{log\mbox{-}b}(S)\ar@{^(->}[d]\\
\mathcal{H}^{\mathbb{A}^1}(S)\ar[r]^-{\sim}&L_{int}\mathrm{log}\mbox{-}\mathcal{H}^{\Box}(S)\ar@{^(->}[r]&\mathrm{log}\mbox{-}\mathcal{H}^{\Box}(S).
}
\]
\end{prop}
\begin{proof}
For any dense open immersion $\underline{U}\to \underline{X}$ in $\mathrm{SmlSm}_S$, the commutative square
\[
\xymatrix{
U^*\ar@{^(->}[d]&X^*\ar@{^(->}[d]\ar[l]\\
\underline{U}\ar[r]&\underline{X}
}
\]
has vertical morphisms in $int(S)$ and top horizontal morphism $U^*\hookrightarrow X^*$ log-dense. By 2-out-of-3, $\underline{U}\to \underline{X}$ is inverted in $L_{int} \mathrm{log}\mbox{-}\mathcal{H}^{log\mbox{-}b}(S)$. Similarly, $\omega_\sharp$ induces the bottom equivalence $L_{int}\mathrm{log}\mbox{-}\mathcal{H}^\Box(S)\simeq\mathcal{H}^{\mathbb{A}^1}(S)$.
\end{proof}

\section{Properties of \texorpdfstring{$L_{bir}$}{Lbir} and Internal Homotopy Theory of \texorpdfstring{$\mathcal{H}^b(S)$}{Hb(S)}}
Throughout this section, let $S$ be a qcqs scheme. Using the identification $\mathcal{H}^{b\mathbb{A}^1}(S)=\mathcal{H}^b(S):=L_{dense}\mathcal{P}_\Sigma(S)$ (\Cref{Motivic equivalences are Birational equivalences}), we study the structural and internal topos-theoretic properties of $\mathcal{H}^b(S)$ and the localization functor $L_{bir}:\mathcal{P}(S)\to \mathcal{P}(S)$. General topos-theoretic properties under a localization have been studied in an earlier work of the author [\cite{inhot}, \S2]; we shall refer the reader to the same as and when certain relevant definitions and results are needed.

\subsection{Colimits and topos-theoretic structure of birational motives}

\begin{thm}\label{bir is closed under sifted colimits}
Let $S$ be a qcqs scheme. The fully faithful inclusions $\mathcal{H}^b(S)\subset\mathcal{H}^{\mathbb{A}^1}(S)$, $\mathcal{H}^b(S)\subset\mathcal{P}_{nis}(S)$, $\mathcal{H}^b(S)\subset\mathcal{P}_\Sigma(S)$, and $\mathcal{H}^b(S)\subset\mathcal{P}(S)$ are closed under sifted colimits.
\end{thm}
\begin{proof}
We first treat the inclusion $\mathcal{H}^b(S)\subset\mathcal{P}(S)$. Since $L_{dense}$ is the Bousfield localization at a set of morphisms in $\mathrm{Sm}_S$ (namely, the dense open immersions), a presheaf is $L_{dense}$-local if and only if it takes dense open immersions to equivalences on sections. Because colimits in $\mathcal{P}(S)$ are computed objectwise, the full subcategory $L_{dense}\mathcal{P}(S)\subset \mathcal{P}(S)$ is closed under all colimits. On the other hand, since sifted colimits in $\mathrm{Spc}$ commute with finite products [\cite{lurie2009higher}, Proposition~5.5.8.11], the full subcategory $\mathcal{P}_\Sigma(S)\subset \mathcal{P}(S)$ of finite-product-preserving presheaves is closed under sifted colimits [\cite{lurie2009higher}, Proposition~5.5.8.10(4)]. As $\mathcal{H}^b(S)=\mathcal{P}_\Sigma(S)\mcap L_{dense}\mathcal{P}(S)$, it follows that $\mathcal{H}^b(S)\subset \mathcal{P}(S)$ is closed under sifted colimits.

The remaining claims now follow formally. Let $D\colon K\to \mathcal{H}^b(S)$ be a diagram indexed by a sifted simplicial set $K$, and let $P := \colim_K (i\circ D)$ denote its colimit in $\mathcal{P}(S)$, where $i\colon \mathcal{H}^b(S)\hookrightarrow \mathcal{P}(S)$ is the inclusion. By the first paragraph, $P$ lies in $\mathcal{H}^b(S)$. In particular, $P$ is already $\Sigma$-local, Nisnevich-local (\Cref{bir local is nis local}), and motivic-local (\Cref{Motivic equivalences are Birational equivalences}), so $L_\Sigma P \simeq L_{nis}P \simeq L_{mot}P \simeq P$. Thus $P$ also computes the colimit of $D$ in $\mathcal{P}_\Sigma(S)$, $\mathcal{P}_{nis}(S)$, and $\mathcal{H}^{\mathbb{A}^1}(S)$.
\end{proof}

\begin{cor}\label[cor]{l bir preserve sifted colimit}
Let $S$ be a qcqs scheme. The localization endofunctors $L_{bir}\colon\mathcal{P}(S)\to \mathcal{P}(S)$ and $L_{bir}\colon\mathcal{P}_{nis}(S)\to \mathcal{P}_{nis}(S)$ preserve sifted colimits. Consequently, by the Barr-Beck-Lurie Monadicity Theorem \textup{[\cite{lurie2017higher}, Theorem~4.7.3.5]}, the adjunctions $\mathcal{P}(S)\leftrightarrows \mathcal{H}^b(S)$ and $\mathcal{P}_{nis}(S)\leftrightarrows \mathcal{H}^b(S)$ are finitary-monadic.
\end{cor}
\begin{proof}
The localization functor $L_{bir}\colon \mathcal{P}(S)\to \mathcal{H}^b(S)$ is a left adjoint and therefore preserves all colimits [\cite{lurie2009higher}, Proposition~5.2.3.5]. By \Cref{bir is closed under sifted colimits}, the fully faithful right adjoint $i\colon \mathcal{H}^b(S)\hookrightarrow \mathcal{P}(S)$ preserves sifted colimits (and is conservative). Thus the composite $i\circ L_{bir}\colon \mathcal{P}(S)\to \mathcal{P}(S)$ preserves sifted colimits. The identical argument applies to $\mathcal{P}_{nis}(S)$.
\end{proof}

\begin{rem}\label[rem]{sifted colimits in hmot}
The analogue of \Cref{bir is closed under sifted colimits} fails for the $\mathbb{A}^1$-motivic localization $\mathcal{H}^{\mathbb{A}^1}(S)\subset \mathcal{P}(S)$. While $L_{\mathbb{A}^1}\mathcal{P}(S)\subset \mathcal{P}(S)$ is closed under all colimits, $\mathcal{P}_{nis}(S)\subset\mathcal{P}(S)$ is only closed under filtered colimits and not under general sifted colimits: Nisnevich excision is defined by sending Nisnevich distinguished squares to pullback squares, and sifted colimits do not commute with pullbacks in $\mathrm{Spc}$. On the other hand, applying $L_{nis}$ to a sifted colimit of motivic spaces generally destroys $\mathbb{A}^1$-invariance. For instance, over a perfect field $k$, let $\mathbb{Z}(\mathbb{G}_m)$ denote the reduced free abelian presheaf on $\mathrm{Sm}_k$ generated by $(\mathbb{G}_m, 1)$. Its classifying presheaf $\mathrm{B}\mathbb{Z}(\mathbb{G}_m)$ is a geometric realization of $\mathbb{A}^1$-invariant presheaves and is thus $\mathbb{A}^1$-invariant, but its Nisnevich sheafification $L_{nis}\mathrm{B}\mathbb{Z}(\mathbb{G}_m)$ is not $\mathbb{A}^1$-invariant because the Nisnevich sheaf $L_{nis}\mathbb{Z}(\mathbb{G}_m)$ is not strongly $\mathbb{A}^1$-invariant [\cite{choudhury2014connectivity}, Lemma~5.6]. Since filtered colimits commute with finite limits in $\mathrm{Spc}$, $\mathcal{H}^{\mathbb{A}^1}(S)\subset \mathcal{P}(S)$ is closed under filtered colimits, but not under sifted colimits.
\end{rem}

For $X\in \mathrm{Sm}_S$, write $h_S^b(X) := L_{bir}h_S(X) \in \mathcal{H}^b(S)$ for the birational localization of the representable presheaf $h_S(X)$.

\begin{lem}\label[lem]{generators of Hb}
Let $S$ be a qcqs scheme. For every $X\in \mathrm{Sm}_S$, the object $h^b_S(X)$ is compact and projective in $\mathcal{H}^b(S)$.
\end{lem}
\begin{proof}
Since colimits in $\mathcal{P}(S)$ are computed pointwise, every representable $h_S(X)$ is a compact projective object. Since $i\colon \mathcal{H}^b(S)\hookrightarrow \mathcal{P}(S)$ preserves sifted colimits (\Cref{bir is closed under sifted colimits}) the corepresented functor $\mathrm{Map}_{\mathcal{H}^b(S)}(h_S^b(X), -)\simeq \mathrm{Map}_{\mathcal{P}(S)}(h_S(X), i(-))$ preserves sifted colimits (and hence both filtered colimits and geometric realizations).
\end{proof}

Except for the case (c), the next result is formal.

\begin{lem}\label[lem]{generation of H^b}
Let $S$ be a qcqs scheme. Then $\mathcal{H}^b(S)$ is a compactly generated, projectively generated presentable $\infty$-category, generated under sifted colimits by the set
\[
\{h_S^b(X)\mid X\in \mathrm{Sm}_S \text{ satisfies property } P\},
\]
where $P$ may be taken to be any of the following:
\begin{enumerate}[label=(\alph*)]
    \item ($S$ qcqs) smooth over $S$;
    \item ($S$ separated) smooth and affine over $S$, or smooth with affine source;
    \item ($S=\operatorname{Spec} k$ with $\operatorname{char}(k)=0$) smooth and projective over $k$.
\end{enumerate}
Moreover, any sifted colimit of such generators in $\mathcal{H}^b(S)$ is computed in $\mathcal{P}(S)$.
\end{lem}
\begin{proof}
By [\cite{lurie2009higher}, Proposition 5.5.8.10] $\mathcal{P}_{\Sigma}(S)$ is generated under sifted colimits by $\{h_S(X)\mid X\in \mathrm{Sm}_S\}$. Since $L_{bir}\colon \mathcal{P}_\Sigma(S)\to \mathcal{H}^b(S)$ is a localization, its essential image generates $\mathcal{H}^b(S)$ under sifted colimits, proving (a). 

When $S$ is separated, any $X\in \mathrm{Sm}_S$ is separated too and hence admits an open cover by finitely many affine open subschemes with each intersection affine. Since Zariski hypercovers (and hence \v{C}ech nerves of finite Zariski open covers) map to colimit diagrams in $\mathcal{P}_{nis}(S)$ and hence in $\mathcal{H}^b(S)$ (\Cref{bir local is nis local}), every $h_S^b(X)$ is a geometric realization of $L_{bir}$ of affine schemes smooth over $S$. This gives (b).

When $S=\operatorname{Spec} k$ with $\operatorname{char}(k)=0$, \Cref{char zero smooth to smoothproper} below implies that for every smooth affine $k$-scheme $X$, there is a dense open immersion $X\hookrightarrow \widetilde{X}$ into a smooth projective $k$-scheme $\widetilde{X}$, whence $h_k^b(X)\xrightarrow{\sim} h_k^b(\widetilde{X})$. Thus (c) follows from (b). Finally, the last assertion is a restatement of \Cref{bir is closed under sifted colimits}.
\end{proof}

\begin{lem}\label[lem]{char zero smooth to smoothproper}
Let $k$ be a field of characteristic zero. Every smooth affine (resp.\ smooth separated) $k$-scheme $X$ admits a dense open immersion $X\hookrightarrow \widetilde{X}$ into a smooth projective (resp.\ smooth proper) $k$-scheme $\widetilde{X}$.
\end{lem}
\begin{proof}
If $X$ is smooth and affine, it is quasi-projective, hence admits a dense open immersion $X\hookrightarrow \overline{X}$ into a projective $k$-scheme $\overline{X}$ (in the general smooth separated case, Nagata compactification [\cite[\href{https://stacks.math.columbia.edu/tag/0F41}{Tag 0F41}]{stacks-project}] yields a dense open immersion into a proper $k$-scheme $\overline{X}$). By Hironaka's resolution of singularities [\cite{annalaresolution}, Corollary~4.13], there exists a projective birational morphism $p\colon \widetilde{X} \to \overline{X}$ such that $\widetilde{X}$ is smooth over $k$ and $p$ is an isomorphism over the smooth locus of $\overline{X}$. Since $X$ is smooth, its image in $\overline{X}$ lies in the smooth locus; thus $p^{-1}(X)\to X$ is an isomorphism, providing the required dense open immersion $X\hookrightarrow \widetilde{X}$ into a smooth projective (resp.\ proper) $k$-scheme.
\end{proof}
\begin{rem}\label[rem]{why not all bir iso}
Because $\mathcal{H}^b(S)$ is constructed from $\mathrm{Sm}_S$ rather than $\mathrm{Sch}_S$ by inverting only dense open immersions between smooth $S$-schemes, birational morphisms (or even dense open immersions) involving non-smooth $S$-schemes (viewed as restrictions of representables to $\mathrm{Sm}_S$) need not automatically be birational equivalences, as the following counterexample shows. So Nagata compactification alone cannot guarantee that $\mathcal{H}^b(S)$ is generated by proper schemes.
\end{rem}
\begin{coex}\label[coex]{singular dense is not inverted}
Let $k$ be a field admitting resolution of singularities, and let $X=\operatorname{Spec}(k[x,y]/(xy))$ be the union of the two coordinate axes in $\mathbb{A}^2_k$, pointed at the origin $0=\operatorname{Spec}k$. From the cdh square expressing $X$ as the gluing of two copies of $\mathbb{A}^1_k$ along $0$, we have $L_{cdh, mot}X_+\simeq \operatorname{Spec}k_+$ in $\mathcal{H}^{\mathbb{A}^1}_{cdh}(\mathrm{Sch}_k)$, so [\cite{MR2593671}, Theorem~4.2] implies that $\Sigma X_+\to \Sigma \operatorname{Spec}k_+\simeq S^1$ is a motivic equivalence in $\mathcal{P}(\mathrm{Sm}_k)_\bullet$. On the other hand, the complement of the origin $U:=X\setminus \{0\}\cong \mathbb{G}_m\sqcup \mathbb{G}_m\hookrightarrow X$ is a dense open immersion into the singular scheme $X$, and $U\to \operatorname{Spec}k\sqcup \operatorname{Spec}k$ is a birational equivalence. If $U\hookrightarrow X$ were inverted in $\mathcal{H}^b(k)$, we would obtain an equivalence
\[
S^1\vee S^1\simeq \Sigma (\operatorname{Spec}k\sqcup \operatorname{Spec}k)_+\simeq \Sigma U_+\simeq \Sigma X_+\simeq S^1 \quad\text{in }\mathcal{H}^b(k)_\bullet.
\]
However, for any nontrivial abelian variety $V$ over $k$, $\mathrm{B}V$ lies in $\mathcal{H}^b(k)$ by \Cref{abelian variety} and \Cref{B_bir is B}, and mapping the above equivalence into $\mathrm{B}V$ would yield
\[
V\cong \mathrm{Map}_{\mathcal{H}^b(k)_\bullet}(S^1,\mathrm{B}V)\simeq \mathrm{Map}_{\mathcal{H}^b(k)_\bullet}(S^1\vee S^1,\mathrm{B}V)\simeq V\times V,
\]
a contradiction.
\end{coex}

We next establish explicit sequential formulas for $L_{bir}$ and show that the localization simplifies dramatically over geometrically unibranch bases.

\begin{lem}\label[lem]{unibranch lemma}
Let $S$ be a qcqs scheme and $f\colon X\to S$ a smooth, qcqs morphism.
\begin{enumerate}
    \item If $S$ is geometrically unibranch, then $X$ is geometrically unibranch; in particular, the irreducible components of $X$ are pairwise disjoint.
    \item If $S$ has finitely many generic points, then so does $X$.
    \item Consequently, if $S$ is geometrically unibranch with finitely many generic points, then $X$ decomposes as a finite disjoint union of irreducible (and hence connected) smooth $S$-schemes.
\end{enumerate}
\end{lem}
\begin{proof}
(1) Since being geometrically unibranch is local on the source and target, we reduce to showing that if $A\to B$ is an essentially smooth local homomorphism of local rings with $A$ geometrically unibranch, then $B$ is geometrically unibranch. This follows by combining [\cite[\href{https://stacks.math.columbia.edu/tag/0DQ1}{Tag 0DQ1(1)}]{stacks-project}] and [\cite[\href{https://stacks.math.columbia.edu/tag/0C37}{Tag 0C37(3)}]{stacks-project}]. For any unibranch scheme $X$, distinct irreducible components cannot intersect: a point $x\in X$ lying in two distinct irreducible components would have at least two minimal primes in $\mathcal{O}_{X,x}$.

(2) Since $f$ is smooth, it is flat, so generalizations lift along $f$ [\cite[\href{https://stacks.math.columbia.edu/tag/03HV}{Tag 03HV}]{stacks-project}]; hence $f$ maps generic points of $X$ to generic points of $S$. For each generic point $\eta\in S$, the fiber $X_\eta$ is a quasi-compact smooth $\kappa(\eta)$-scheme, hence Noetherian, and therefore has only finitely many irreducible components [\cite[\href{https://stacks.math.columbia.edu/tag/0052}{Tag 0052}]{stacks-project}].

(3) By (1) and (2), $X$ has finitely many irreducible components, and they are pairwise disjoint. Since the irreducible components are closed and finite in number, each is also open in $X$, yielding the claimed finite clopen decomposition.
\end{proof}

\begin{lem}\label[lem]{form for lbir}
Let $S$ be a qcqs scheme. As endofunctors of $\mathcal{P}(S)$, there are canonical equivalences
\[
L_{bir}\simeq \dcolim_n{(L_{\Sigma}L_{dense})^n}\simeq \dcolim_n{(L_{nis}L_{dense})^n}\simeq \dcolim_n{(L_{mot}L_{dense})^n}.
\]
Moreover, if $S$ is geometrically unibranch with finitely many generic points, then a single step suffices:
\[
L_{bir}\simeq L_{\Sigma}L_{dense}\simeq L_{nis}L_{dense}\simeq L_{mot}L_{dense}.
\]
\end{lem}
\begin{proof}
Let $\esc{X}\in \mathcal{P}(S)$ and let $L'\in \{L_\Sigma, L_{nis}, L_{mot}\}$. Using the localization units $\mathrm{id}\to L_{dense}$ and $\mathrm{id}\to L'$, we form the sequential diagram
\begin{equation}\label{1}
L_{dense}\esc{X}\to L'L_{dense}\esc{X}\to L_{dense}L'L_{dense}\esc{X}\to (L'L_{dense})^2\esc{X}\to\cdots
\end{equation}
The subsequences $\{(L'L_{dense})^n\esc{X}\}_{n\ge 1}$ and $\{L_{dense}(L'L_{dense})^n\esc{X}\}_{n\ge 0}$ are both cofinal in~\eqref{1}. Since $\mathcal{P}_\Sigma(S)\subset \mathcal{P}(S)$ and $L_{dense}\mathcal{P}(S)\subset \mathcal{P}(S)$ are both closed under filtered colimits, the colimit $P := \dcolim_n (L'L_{dense})^n\esc{X}$ simultaneously lies in $\mathcal{P}_\Sigma(S)$ (by the first cofinal subsequence, as $L'\mathcal{P}(S)\subset \mathcal{P}_\Sigma(S)$) and in $L_{dense}\mathcal{P}(S)$ (by the second). Hence $P\in \mathcal{P}_\Sigma(S)\mcap L_{dense}\mathcal{P}(S) = \mathcal{H}^b(S)$. Moreover, since both $L_{dense}$-equivalences and $L'$-equivalences are birational equivalences (\Cref{Motivic equivalences are Birational equivalences}) and birational equivalences are stable under colimits, the canonical map $\esc{X}\to P$ is a birational equivalence, identifying $P\simeq L_{bir}\esc{X}$.

Suppose now that $S$ is geometrically unibranch with finitely many generic points. By \Cref{unibranch lemma}(3), every $X\in \mathrm{Sm}_S$ decomposes into a finite disjoint union of irreducible (and hence connected) components $X = \coprod_{i=1}^m X_i$, and every dense open immersion $U\hookrightarrow X$ decomposes as $\coprod_{i=1}^m U_i \hookrightarrow \coprod_{i=1}^m X_i$ where each $U_i\hookrightarrow X_i$ is a nonempty open subscheme of an irreducible scheme, hence itself irreducible and connected. Since $(L_\Sigma\mathcal{F})(Y)\simeq \prod_{j} \mathcal{F}(Y_j)$ for any finite decomposition of $Y\in \mathrm{Sm}_S$ into connected components $Y=\coprod_j Y_j$, it follows that $L_\Sigma$ preserves $L_{dense}$-local presheaves. Thus $L_\Sigma L_{dense}\esc{X}$ already lies in $\mathcal{H}^b(S)$ and is a model for $L_{bir}$. Since $\mathcal{H}^b(S)\subset \mathcal{H}^{\mathbb{A}^1}(S)\subset \mathcal{P}_{nis}(S)\subset \mathcal{P}_\Sigma(S)$, we conclude that $L_\Sigma L_{dense}\esc{X}\simeq L_{nis}L_{dense}\esc{X}\simeq L_{mot}L_{dense}\esc{X}\simeq L_{bir}\esc{X}$.
\end{proof}

Finally, we show that when $S$ is geometrically unibranch with finitely many generic points, $L_{bir}$ preserves finite coproducts in $\mathcal{P}_\Sigma(S)$ and $\mathcal{P}_{nis}(S)$, and hence all colimits, endowing $\mathcal{H}^b(S)$ with the structure of an $\infty$-topos. First we set up the abstract categorical result.

\begin{lem}\label[lem]{coreflexive sub topos}
Let $\mathcal{C}$ be an $\infty$-topos and $L\colon \mathcal{C}\leftrightarrows L\mathcal{C} : i$ an accessible localization such that $i:L\mathcal{C} \subset \mathcal{C}$ is closed under all colimits. Then $L\mathcal{C}$ is an $\infty$-topos, and $i$ is the inverse image functor and $L$ the exceptional pushforward of an essential geometric morphism $\mathcal{C}\to L\mathcal{C}$.
\end{lem}
\begin{proof}
As an accessible localization of a presentable $\infty$-category, $L\mathcal{C}$ is presentable. Because $i\colon L\mathcal{C}\hookrightarrow \mathcal{C}$ is fully faithful and preserves all limits and all colimits, Giraud's axioms, namely universality of colimits, disjointness of finite coproducts, and effectiveness of groupoid objects [\cite{lurie2009higher}, Theorem~6.1.0.6(3)(ii)-(iv)], pass directly from $\mathcal{C}$ to $L\mathcal{C}$. Moreover, since $i$ preserves all colimits and limits, the Adjoint Functor Theorem [\cite{lurie2009higher}, Corollary~5.5.2.9] furnishes a right adjoint $R$ to $i$; the triple $L \dashv i \dashv R$ then defines an essential geometric morphism $\mathcal{C}\to L\mathcal{C}$ with inverse image $i$ and exceptional pushforward $L$ [\cite{lurie2009higher}, Definition~6.3.1.1].
\end{proof}

\begin{prop}\label[prop]{closed under coproducts}
Let $S$ be a qcqs, geometrically unibranch scheme with finitely many generic points. Then the inclusions $\mathcal{H}^b(S)\subset \mathcal{P}_\Sigma(S)$ and $\mathcal{H}^b(S)\subset \mathcal{P}_{nis}(S)$ are closed under all colimits.
\end{prop}
\begin{proof}
Since any colimit can be decomposed into sifted colimits and finite coproducts [\cite{lurie2009higher}, Lemma~5.5.8.13], \Cref{bir is closed under sifted colimits} reduces the claim to showing that $\mathcal{H}^b(S)$ is closed under finite coproducts in $\mathcal{P}_\Sigma(S)$ and $\mathcal{P}_{nis}(S)$. Note that $\mathcal{P}_\Sigma(S) = \mathrm{Shv}_{\sqcup}(\mathrm{Sm}_S)$ is the $\infty$-topos of sheaves for the finite disjoint union topology. By \Cref{unibranch lemma}(3), every $X\in \mathrm{Sm}_S$ is a finite disjoint union of connected smooth $S$-schemes, and any dense open immersion $U\hookrightarrow X$ decomposes into a finite disjoint union of dense open immersions between connected smooth $S$-schemes. Because connected schemes $Y\in \mathrm{Sm}_S$ are points of $\sqcup$-topology, $L_{dense}$-locality of $\sqcup$-sheaves can be reduced to checking the same on connected schemes only. Hence any finite coproduct in $\mathcal{P}_\Sigma(S)$ of $L_{dense}$-local objects is $L_{dense}$-local. In other words, $\mathcal{H}^b(S)$ is closed under coproducts in $\mathcal{P}_\Sigma(S)$ Since $\mathcal{P}_{nis}(S)\subset \mathcal{P}_\Sigma(S)$ is also closed under finite coproducts, the claim for $\mathcal{P}_{nis}(S)$ follows.
\end{proof}

\begin{thm}\label{Hbir a topos}
Let $S$ be a qcqs, geometrically unibranch scheme with finitely many generic points (for example, the spectrum of a field or any normal Noetherian scheme). Then $\mathcal{H}^{b}(S)$ is an $\infty$-topos. Moreover, both $L_{bir}:\mathcal{P}_{\Sigma}(S)\to \mathcal{H}^{b}(S)$ and $L_{bir}:\mathcal{P}_{nis}(S)\to \mathcal{H}^{b}(S)$ define essential geometric morphisms of $\infty$-topoi.
\end{thm}
\begin{proof}
Combine \Cref{coreflexive sub topos} and \Cref{closed under coproducts}.
\end{proof}
\begin{thm}\label{homotopy dim 0}
    The birational $\infty$-topos $\mathcal{H}^b(S)$ of a geometrically unibranch scheme $S$ with finitely many generic points has homotopy dimension $0$.
\end{thm}
\begin{proof}
    Since $\tau_{\leq -1}^{bir}\simeq {\tau_{\leq -1}^{\mathcal{P}(S)}}_{|\mathcal{H}^b(S)}$ (\Cref{birational truncation}) this follows from the fact that $\mathcal{P}(S)$ has homotopy dimension $0$ [\cite{lurie2009higher}, Example 7.2.1.3].
\end{proof}
\begin{prop}\label[prop]{sub obj in bir}
Let $S$ be a qcqs, geometrically unibranch scheme with finitely many generic points. Then the $\infty$-topos $\mathcal{H}^b(S)$ admits an explicit subobject classifier $\Omega^b_S$ given by
\[
X \longmapsto \Omega^{\sqcup}_S(X):= \{\text{clopen subschemes } W\subset X\}.
\]
\end{prop}
\begin{proof}
Since the inclusion $\mathcal{H}^b(S)\hookrightarrow \mathcal{P}_\Sigma(S)=\mathrm{Shv}_\sqcup(\mathrm{Sm}_S)$ preserves all limits and colimits (\Cref{closed under coproducts}), it preserves monomorphisms and pullback squares; hence it suffices to check that the subobject classifier $\Omega^\sqcup_S$ of $\mathrm{Shv}_\sqcup(\mathrm{Sm}_S)$ is birational local. By \Cref{unibranch lemma}(3), any $X\in \mathrm{Sm}_S$ decomposes as a finite disjoint union of its irreducible components, each of which is clopen. Thus $\Omega^{\sqcup}_S(X)$ is simply the power set of the set of irreducible components of $X$, which is manifestly invariant under restriction to dense open subschemes.
\end{proof}

\subsection{Preservation of finite limits and the bar construction}\label{3.2}

\begin{prop}\label[prop]{closed under limits}
Let $S$ be a qcqs scheme. The fully faithful inclusions $\mathcal{H}^b(S)\subset\mathcal{H}^{\mathbb{A}^1}(S)$, $\mathcal{H}^b(S)\subset\mathcal{P}_{nis}(S)$, $\mathcal{H}^b(S)\subset\mathcal{P}_\Sigma(S)$, and $\mathcal{H}^b(S)\subset\mathcal{P}(S)$ are closed under all small limits.
\end{prop}
\begin{proof}
Each of these inclusions is right adjoint to the corresponding localization functor and therefore preserves all small limits [\cite{lurie2009higher}, Proposition~5.2.3.5].
\end{proof}

Our first main result in this subsection is that the birational localization functor $L_{bir}$ is cartesian (i.e., preserves finite products).

\begin{thm}\label{Lbir is cartesian}
Let $S$ be a qcqs scheme. Every composite of localization functors in the sequence
\[
\mathcal{P}(S)\xrightarrow{L_\Sigma} \mathcal{P}_\Sigma(S)\xrightarrow{L_{nis}}\mathcal{P}_{nis}(S)\xrightarrow{L_{mot}} \mathcal{H}^{\mathbb{A}^1}(S)\xrightarrow{L_{bir}}\mathcal{H}^b(S)
\]
preserves finite products.
\end{thm}
\begin{proof}
Since $L_{\Sigma}$ and $L_{nis}$ are sheafification functors for Grothendieck topologies on $\mathrm{Sm}_S$, they are left exact, and $L_{mot}$ preserves finite products by [\cite{hoyois2017six}, Proposition~3.15]. Moreover, since each local subcategory in the chain is closed under limits in $\mathcal{P}(S)$ (\Cref{closed under limits}), products in each subcategory are computed in $\mathcal{P}(S)$; hence it suffices to show that $L_{bir}\colon \mathcal{P}(S)\to \mathcal{H}^b(S)$ preserves finite products. Because the terminal object $*\in \mathcal{P}(S)$ is already birational-local, $L_{bir}(*)\simeq *$, so it remains to check that $L_{bir}$ preserves binary products.

Since filtered colimits in $\mathcal{P}(S)$ commute with finite products and $L_\Sigma$ preserves finite products, the sequential formula $L_{bir}\simeq \dcolim_n (L_\Sigma L_{dense})^n$ (\Cref{form for lbir}) reduces the problem to showing that $L_{dense}\colon \mathcal{P}(S)\to L_{dense}\mathcal{P}(S)$ preserves binary products. It suffices to show that for every $\esc{X}\in \mathcal{P}(S)$, the functor $\esc{X}\times -\colon \mathcal{P}(S)\to \mathcal{P}(S)$ preserves $L_{dense}$-equivalences. Since $\mathcal{P}(S)$ is generated under small colimits by representables $\{h_S(X)\mid X\in \mathrm{Sm}_S\}$ and the product in $\mathcal{P}(S)$ preserves colimits in each variable separately, we may reduce to checking that $h_S(X)\times j$ is a dense equivalence for every $X\in \mathrm{Sm}_S$ and every dense open immersion $j$ in $\mathrm{Sm}_S$. This holds by \Cref{dense prod}.
\end{proof}

Combining the preservation of finite products (\Cref{Lbir is cartesian}) with the preservation of sifted colimits (\Cref{l bir preserve sifted colimit}) yields a strong compatibility of $L_{bir}$ with algebraic structures and the bar construction.

\begin{lem}\label[lem]{bir bar commute}
Let $S$ be a qcqs scheme. The localization functor $L_{bir}$ satisfies the following properties:
\begin{enumerate}
    \item $L_{bir}$ preserves (grouplike) (commutative) monoid objects and thus induces functors
    \[
    L_{bir} \colon (\mathrm{C})\mathrm{Mon}(\mathcal{P}(S))^{(\mathrm{gp})} \longrightarrow (\mathrm{C})\mathrm{Mon}(\mathcal{H}^b(S))^{(\mathrm{gp})}.
    \]
    Likewise, $L_{bir}$ preserves monoid actions.
    \item $L_{bir}$ commutes with the two-sided bar construction $\mathrm{B}(-,-,-)$ and its Nisnevich sheafification $\mathrm{B}_{nis}(-,-,-):=L_{nis}\mathrm{B}(-,-,-)$, i.e., there are canonical equivalences
    \[
    L_{bir}\mathrm{B}_{nis}(-,-,-)\simeq L_{bir}\mathrm{B}(-,-,-)\simeq \mathrm{B}(L_{bir}-,L_{bir}-,L_{bir}-).
    \]
    In particular, if $\esc{M}\in \mathrm{Mon}(\mathcal{H}^b(S))$ acts on $\esc{X},\esc{Y}\in \mathcal{H}^b(S)$ from the right and left respectively, then $\mathrm{B}_{nis}(\esc{X},\esc{M},\esc{Y})\simeq \mathrm{B}(\esc{X},\esc{M},\esc{Y})$ already lies in $\mathcal{H}^b(S)$.
\end{enumerate}
\end{lem}
\begin{proof}
(1) Since monoid, commutative monoid, and module structures in a cartesian monoidal $\infty$-category, as well as the grouplike condition, are defined via finite-product-preserving Segal and shearing conditions, they are preserved by any finite-product-preserving functor (\Cref{Lbir is cartesian}).

(2) Let $\esc{M}\in \mathrm{Mon}(\mathcal{P}(S))$ act on $\esc{X}\in \mathcal{P}(S)$ from the right and on $\esc{Y}\in \mathcal{P}(S)$ from the left. Recall that the two-sided bar construction is the geometric realization of the simplicial bar object:
\[
\mathrm{B}(\esc{X},\esc{M},\esc{Y}) := \big|\esc{X}\times \esc{M}^{\times \bullet}\times \esc{Y}\big|.
\]
Since $L_{bir}\colon \mathcal{P}(S)\to \mathcal{P}(S)$ preserves geometric realizations (\Cref{l bir preserve sifted colimit}) and finite products (\Cref{Lbir is cartesian}), we obtain canonical equivalences
\begin{align*}
L_{bir}\mathrm{B}(\esc{X}, \esc{M},\esc{Y})
= L_{bir}\big| \esc{X}\times \esc{M}^{\times \bullet}\times \esc{Y}\big| 
\simeq \big| L_{bir}(\esc{X}\times \esc{M}^{\times \bullet}\times \esc{Y})\big| &\simeq \big| L_{bir}\esc{X}\times (L_{bir}\esc{M})^{\times \bullet}\times L_{bir}\esc{Y}\big| \\
&= \mathrm{B}(L_{bir}\esc{X}, L_{bir}\esc{M},L_{bir}\esc{Y}).
\end{align*}
Since $\mathcal{H}^b(S)\subset \mathcal{P}(S)$ is closed under geometric realizations (\Cref{bir is closed under sifted colimits}) and finite products (\Cref{closed under limits}), $\mathrm{B}(L_{bir}\esc{X}, L_{bir}\esc{M},L_{bir}\esc{Y})$ lies in $\mathcal{H}^b(S)$ and is therefore Nisnevich-local (\Cref{bir local is nis local}). Combined with the absorption equivalence $L_{bir}L_{nis}\simeq L_{bir}$, this yields
\[
L_{bir}\mathrm{B}_{nis}(\esc{X},\esc{M},\esc{Y})\simeq L_{bir}\mathrm{B}(\esc{X},\esc{M},\esc{Y})\simeq \mathrm{B}(L_{bir}\esc{X},L_{bir}\esc{M},L_{bir}\esc{Y})\simeq \mathrm{B}_{nis}(L_{bir}\esc{X}, L_{bir}\esc{M},L_{bir}\esc{Y}).\qedhere
\]
\end{proof}

\begin{cor}\label[cor]{B_bir is B}
The functor $L_{bir}$ commutes with the classifying space functor $\mathrm{B}_{nis}:=L_{nis}\mathrm{B}(*,-,*)$. Consequently, the internal bar construction $\mathrm{B}_{bir}\colon \mathrm{Mon}(\mathcal{H}^b(S))\to\mathcal{H}^b(S)$ coincides with the restrictions of both $\mathrm{B}$ and $\mathrm{B}_{nis}$ to $\mathrm{Mon}(\mathcal{H}^b(S))$. In particular, for any birational-local monoid space $\esc{F}$ (assumed commutative if $n\geq 2$), $\mathrm{B}_{nis}^n \esc{F}\simeq \mathrm{B}^n\esc{F}$ is again birational-local. It follows that a Nisnevich-local presheaf of grouplike monoids $\esc{G}$ is birational-local if and only if $\mathrm{B}_{nis} \esc{G}$ is birational-local.
\end{cor}

\begin{rem}\label[rem]{bar mot does not commute}
The analogue of \Cref{B_bir is B} fails for $L_{mot}$ and $\mathrm{B}_{nis}$ because $\mathcal{H}^{\mathbb{A}^1}(S)\subset \mathcal{P}_{nis}(S)$ is not closed under geometric realizations (\Cref{sifted colimits in hmot}). Notably, the obstruction lies entirely with Nisnevich sheafification rather than $\mathbb{A}^1$-localization. Indeed, $L_{\mathbb{A}^1}\colon \mathcal{P}(S)\to \mathcal{P}(S)$ preserves finite products [\cite{hoyois2017six}, Corollary~3.5] as well as all small colimits (since $L_{\mathbb{A}^1}\mathcal{P}(S)\subset \mathcal{P}(S)$ is closed under all colimits), so the argument of \Cref{bir bar commute}(2) gives $L_{\mathbb{A}^1}\mathrm{B}\simeq \mathrm{B}L_{\mathbb{A}^1}$. Because filtered colimits also commute with geometric realizations and finite products, in the sequential formula $L_{mot}\simeq \dcolim_{n}(L_{\mathbb{A}^1} L_{nis})^n$, the failure of $L_{mot}$ to commute with $\mathrm{B}$ stems solely from the failure of $L_{nis}$ (and indeed of sheafification for any topology finer than the $\sqcup$-topology) to preserve geometric realizations in $\mathcal{P}(S)$.

The grouplike condition  $\esc{G}\simeq\Omega \mathrm{B}_{nis} \esc{G}$ together with the fact that $\Omega$ preserves birational-local objects (\Cref{closed under limits}) yields the ``if'' direction of the last statement.
\end{rem}

Let $\mathrm{Shv}^b_S$, $\mathrm{Grp}^b_S$, and $\mathrm{Ab}^b_S$ denote the categories of discrete birational sheaves of sets, groups, and abelian groups on $\mathrm{Sm}_S$, respectively.

\begin{cor}\label[cor]{bir is strong a1}
Every birational-local (commutative) monoid space is (strictly) strongly $\mathbb{A}^1$-invariant in the sense of \textup{[\cite{elmanto2021motivic}, Definition~3.1.6]}; that is,
\[
\mathrm{Mon}(\mathcal{H}^b(S))\subset \mathrm{Mon}_{mot}(\mathcal{H}^{\mathbb{A}^1}(S)) \quad\text{and}\quad \mathrm{CMon}(\mathcal{H}^b(S))\subset \mathrm{CMon}_{mot}(\mathcal{H}^{\mathbb{A}^1}(S)).
\]
In particular, on discrete sheaves of groups and abelian groups,
\[
\mathrm{Grp}^{b}_S\subset \mathrm{Grp}^{\mathbb{A}^1}_S \quad\text{and}\quad \mathrm{Ab}^b_S\subset \mathrm{Ab}^{\mathbb{A}^1}_S,
\]
where $\mathrm{Grp}^{\mathbb{A}^1}_S$ (resp.\ $\mathrm{Ab}^{\mathbb{A}^1}_S$) denotes the category of strongly (resp.\ strictly) $\mathbb{A}^1$-invariant Nisnevich sheaves of groups (resp.\ abelian groups).
\end{cor}
\begin{proof}
Let $\esc{M}\in \mathrm{Mon}(\mathcal{H}^b(S))$ (resp.\ $\mathrm{CMon}(\mathcal{H}^b(S))$). By \Cref{B_bir is B}, $\mathrm{B}_{nis}\esc{M}$ (resp.\ $\mathrm{B}^n_{nis}\esc{M}$ for all $n\ge 1$) lies in $\mathcal{H}^b(S)$, and is therefore motivic-local by \Cref{Motivic equivalences are Birational equivalences}. By [\cite{elmanto2021motivic}, Definition~3.1.6], this means precisely that $\esc{M}$ is strongly (resp.\ strictly) $\mathbb{A}^1$-invariant.
\end{proof}

\begin{rem}
\Cref{bir is strong a1} strengthens [\cite{MR4637972}, Remark~2.3] by an entirely different method that applies over any qcqs base scheme $S$.
\end{rem}

In [\cite{inhot}, \S3.2], we observed that the ordinary motivic homotopy category $\mathcal{H}^{\mathbb{A}^1}(S)$ fails to be an $\infty$-topos precisely because not every discrete motivic group sheaf is strongly $\mathbb{A}^1$-invariant, and hence not every such sheaf admits an internal delooping. By contrast, \Cref{B_bir is B} shows that in the birational setting there is no distinction between birational-local group sheaves and ``strongly'' or ``strictly'' birational-local group sheaves. (In fact, in [\cite{inhot}, Lemma 3.2.28], we observe a stronger converse to \Cref{bir is strong a1}: free motivic groups are strongly $\mathbb{A}^1$-invariant only if they are birational.) Combining \Cref{B_bir is B} with the path injectivity of $L_{bir}$ (\Cref{birational truncation}(2)), we see that the delooping obstruction vanishes completely in $\mathcal{H}^b(S)$ [\cite{inhot}, Theorem 2.3.7.]. Recall from [\cite{inhot}, Definition~2.2.5] that a pointed object $\esc{X}\in \mathcal{H}^b(S)_\bullet$ is birationally $n$-connective ($\esc{X}\in \mathcal{H}^b(S)_{\geq n}$) if $\tau_{\leq n-1}^{bir}\esc{X}\simeq *$, and birationally $n$-Eilenberg-MacLane ($\esc{X}\in \mathcal{EM}_n(\mathcal{H}^b(S))$) if it lies in $\mathcal{H}^b(S)_{\leq n}\mcap \mathcal{H}^b(S)_{\geq n}$.

\begin{thm}\label{bir em}\label[prop]{bir omega bar}
Let $S$ be a qcqs scheme. Exactly as in an $\infty$-topos \textup{[\cite{lurie2009higher}, Lemma~7.2.2.11(1) and Proposition~7.2.2.12]}:
\begin{enumerate}
\item The loop functor
\[
\Omega\colon \mathcal{H}^{b}(S)_{\geq 1} \xrightarrow{} \mathrm{Grp}(\mathcal{H}^{b}(S)_\bullet)\simeq \mathrm{Grp}(\mathcal{H}^{b}(S))
\]
is an equivalence of $\infty$-categories, with inverse given by the bar construction $\mathrm{B}_{bir}\simeq \mathrm{B}|_{\mathrm{Grp}(\mathcal{H}^b(S))}$.
\item There are canonical equivalences of categories
\begin{align*}
\mathrm{Shv}_{*,S}^{b}\simeq \mathrm{Disc}(\mathcal{H}^{b}(S))_\bullet &\underset{\mathrm{B}^0_{bir}}{\overset{\pi_0^{bir}}{\;\leftrightarrows\;}} \mathcal{EM}_0(\mathcal{H}^{b}(S)), \\
\mathrm{Grp}_S^{b}\simeq \mathrm{Grp}(\mathrm{Disc}(\mathcal{H}^{b}(S))) &\underset{\mathrm{B}_{bir}}{\overset{\pi_1^{bir}}{\;\leftrightarrows\;}} \mathcal{EM}_1(\mathcal{H}^{b}(S)), \\
\mathrm{Ab}_S^{b}\simeq \mathrm{Ab}(\mathrm{Disc}(\mathcal{H}^{b}(S))) &\underset{\mathrm{B}^n_{bir}}{\overset{\pi_n^{bir}}{\;\leftrightarrows\;}} \mathcal{EM}_n(\mathcal{H}^{b}(S)) \quad (n\geq 2).
\end{align*}
\end{enumerate}
\end{thm}
\begin{proof}
First of all, using path injectivity of $L_{bir}$ (\Cref{birational truncation}(2)), we know that birational local spaces are connected inside $\mathcal{H}^b(S)$ iff they are connected in $\mathcal{P}(S)$, i.e. $\mathcal{H}^b(S)_{\geq 1}=\mathcal{P}(S)_{\geq 1}\mcap \mathcal{H}^b(S)$ (see also [\cite{inhot}, Proposition 2.2.8.]). On the other hand, since $\mathrm{B}_{bir}\simeq \mathrm{B}|_{\mathrm{Mon}(\mathcal{H}^{b}(S))}$ (\Cref{B_bir is B}) the classical equivalence $\mathrm{B}\dashv \Omega$ between $\mathrm{Grp}(\mathcal{P}(S))$ and $\mathcal{P}(S)_{\geq 1}$ restricts directly to $\mathcal{H}^b(S)$. Similarly, for (2), note that $\mathcal{EM}_n(\mathcal{H}^b(S))=\mathcal{EM}_n(\mathcal{P}(S))\mcap \mathcal{H}^b(S)$ and restrict the Eilenberg-MacLane equivalences $\mathrm{B}^n\colon n\text{-}\mathrm{Grp}(\mathrm{PSh}_S)\simeq \mathcal{EM}_n(\mathcal{P}(S)) : \pi_n$ to $\mathcal{H}^b(S)$ via the identification $\mathrm{B}_{bir}\simeq \mathrm{B}|_{\mathrm{Mon}(\mathcal{H}^{b}(S))}$ of \Cref{B_bir is B}.
\end{proof}

\begin{rem}\label[rem]{flasque remark}
Geometrically, the coincidence $\mathrm{B}^nA\simeq \mathrm{B}_{nis}^nA$ in \Cref{B_bir is B} and \Cref{bir em} reflects the fact that birational sheaves are flasque over any qcqs base scheme $S$: for any open immersion $U\hookrightarrow X$ in $\mathrm{Sm}_S$, the inclusion $U\sqcup (X\setminus\overline{U})\hookrightarrow X$ is a dense open immersion, so for any $F\in \mathrm{Shv}_S^b$ the isomorphism $F(X)\xrightarrow{\sim} F(U\sqcup (X\setminus \overline{U}))\cong F(U)\times F(X\setminus \overline{U})$ followed by projection onto $F(U)$ shows that restriction $F(X)\to F(U)$ is surjective. Since flasque Nisnevich sheaves of abelian groups have vanishing higher Nisnevich cohomology (by [\cite{riou2002theorie}, Lemmas~1.39 and 1.40] and Noetherian approximation), this corresponds to the vanishing of internal higher cohomology in $\mathcal{H}^b(S)$.
\end{rem}

\begin{cor}\label[cor]{bir is exact in presheaves}
The inclusions $\mathrm{Ab}^b_S\subset \mathrm{Ab}_S$ and $\mathrm{Ab}^b_S\subset \mathrm{PAb}_S$ are closed under all small limits and colimits; in particular, they are exact. \textup{(}Compare with the ordinary motivic case \textup{[\cite{morel2012a1}, Corollary 6.24]} which requires the stable $\mathbb{A}^1$-connectivity theorem.\textup{)}
\end{cor}
\begin{proof}
Closure under limits is immediate since both inclusions are right adjoints. For colimits, since $\mathrm{Ab}^b_S\subset \mathrm{Ab}_S\subset \mathrm{PAb}_S$, it suffices to show that $\mathrm{Ab}^b_S\subset \mathrm{PAb}_S$ is closed under all small colimits. Because colimits in $\mathrm{PAb}_S$ are computed sectionwise, $L_{dense}\mathrm{PAb}_S\subset \mathrm{PAb}_S$ is closed under all colimits. Moreover, $\mathrm{Ab}^b_S = L_{dense}\mathrm{PAb}_S\mcap L_\Sigma\mathrm{PAb}_S$. Since $\mathrm{PAb}_S$ is semiadditive, finite coproducts in $\mathrm{PAb}_S$ coincide with finite products; as colimits commute with finite direct sums in $\mathrm{Ab}$, the subcategory $L_\Sigma\mathrm{PAb}_S\subset \mathrm{PAb}_S$ of finite-product-preserving presheaves is also closed under all colimits.
\end{proof}
\begin{thm}\label{coh dim of bir}
Over a qcqs scheme $S$, every smooth $X/S$ has cohomological dimension $0$ in $\mathcal{H}^b(S)$; that is, $\mathbf{H}^{n}_{\mathcal{H}^{b}(S)}(h_S^b(X);A)=0$ for all $A\in \mathrm{Ab}_S^b$ and all $n>0$.
\end{thm}
\begin{proof}
For any $n>0$, $A\in \mathrm{Ab}_S^b$, and $X\in \mathrm{Sm}_S$, \Cref{B_bir is B} and \Cref{bir em}(2) give
\[
\mathbf{H}^{n}_{\mathcal{H}^{b}(S)}(h_S^b(X); A) := \pi_0\mathrm{Map}_{\mathcal{H}^{b}(S)}\big(h_S^b(X),\mathrm{B}^n_{bir}A\big) \cong \pi_0\mathrm{B}^n(A(X)) = 0.
\]
Since $\mathcal{H}^{b}(S)$ is generated under small colimits by $\{h_S^b(X)\mid X\in \mathrm{Sm}_S\}$ (\Cref{generation of H^b}), the vanishing of $\mathbf{H}^{n}_{\mathcal{H}^{b}(S)}(\esc{E};A)$ for arbitrary $\esc{E}\in \mathcal{H}^b(S)$ follows by writing $\esc{E}$ as a colimit of representables. (Of course, when $S$ is qcqs geometrically unibranch with finitely many generic points so that $\mathcal{H}^b(S)$ is of homotopy dimension $0$ (\Cref{homotopy dim 0}), this is an immediate consequence of the fact that for $\infty$-toposes of finite homotopy dimension, cohomological dimension is bounded above by its homotopy dimension [\cite{lurie2009higher}, Corollary 7.2.2.30]).
\end{proof}

We now turn to the interaction of $L_{bir}$ with fiber sequences and pullbacks.

\begin{lem}\label[lem]{Lbir preserve connected fiber sequence}
The functor $L_{bir}$ preserves fiber sequences $\Gamma \to \esc{X}\to \esc{Y}$ in $\mathcal{P}_{nis}(S)_\bullet$ (resp.\ $\mathcal{P}(S)_\bullet$) whenever the base $\esc{Y}$ is Nisnevich-locally (resp.\ sectionwise) connected.
\end{lem}
\begin{proof}
We give the proof for $\mathcal{P}_{nis}(S)_\bullet$; the case of $\mathcal{P}(S)_\bullet$ is identical. Since $\esc{Y}$ is a $1$-connective pointed object of the $\infty$-topos $\mathcal{P}_{nis}(S)$, we have $\esc{Y}\simeq \mathrm{B}_{nis}\Omega\esc{Y}$, and the fiber sequence $\Gamma\to \esc{X}\to \esc{Y}$ is canonically classified by an action of the grouplike monoid $\Omega\esc{Y}$ on $\Gamma$, yielding an equivalence of sequences [\cite{asok2017simplicial}, Theorem~2.3.3, Steps~1-2]
\[
\xymatrix{
\Gamma \ar[r]\ar[d]_{\rotatebox{90}{$\sim$}} & \esc{X}\ar[r]\ar[d]_{\rotatebox{90}{$\sim$}} & \esc{Y}\ar[d]^{\rotatebox{90}{$\sim$}} \\
\Gamma \ar[r] & \mathrm{B}_{nis}(\Gamma, \Omega \esc{Y}, *) \ar[r] & \mathrm{B}_{nis}(*, \Omega \esc{Y}, *).
}
\]
Applying $L_{bir}$ and using \Cref{bir bar commute}(2) produces the commutative diagram
\[
\xymatrix{
L_{bir}\Gamma \ar[r]\ar[d]_{\rotatebox{90}{$\sim$}} & L_{bir}\esc{X}\ar[r]\ar[d]_{\rotatebox{90}{$\sim$}} & L_{bir}\esc{Y}\ar[d]^{\rotatebox{90}{$\sim$}} \\
L_{bir}\Gamma \ar[r] & \mathrm{B}_{nis}(L_{bir}\Gamma, L_{bir}\Omega \esc{Y}, *) \ar[r] & \mathrm{B}_{nis}(*, L_{bir}\Omega \esc{Y}, *).
}
\]
By \Cref{bir bar commute}(1), $L_{bir}\Omega\esc{Y}$ is a grouplike monoid in $\mathcal{H}^b(S)\subset \mathcal{P}_{nis}(S)$ acting on $L_{bir}\Gamma$. The bottom row is therefore the principal fiber sequence associated to this action in $\mathcal{P}_{nis}(S)$, so the top row $L_{bir}\Gamma \to L_{bir}\esc{X}\to L_{bir}\esc{Y}$ is a fiber sequence as well.
\end{proof}

\begin{thm}\label{bir and Omega commute on connected}
The localization functor $L_{bir}$ commutes with the loop space functor $\Omega$ on $\tau^{nis}_{\geq 1}\mathcal{P}_{nis}(\mathrm{Sm}_S)_\bullet$ and on $\tau_{\geq 1}\mathcal{P}(\mathrm{Sm}_S)_\bullet$. Consequently, $L_{bir}$ commutes with the Nisnevich group completion $\esc{M}^{\mathrm{gp}}\simeq \Omega\mathrm{B}_{nis}\esc{M}$ of commutative monoids in $\mathcal{P}_{nis}(S)$.
\end{thm}
\begin{proof}
For any $\esc{E}\in \tau^{nis}_{\geq 1}\mathcal{P}_{nis}(\mathrm{Sm}_S)_\bullet$, applying \Cref{Lbir preserve connected fiber sequence} to the path-loop fiber sequence $\Omega\esc{E}\to *\to \esc{E}$ yields $L_{bir}\Omega\esc{E}\simeq \Omega L_{bir}\esc{E}$ (and similarly on $\tau_{\geq 1}\mathcal{P}(\mathrm{Sm}_S)_\bullet$). For any $\esc{M}\in \mathrm{CMon}(\mathcal{P}(S))$, $\mathrm{B}_{nis}\esc{M}$ is $1$-connective, so combining this with \Cref{B_bir is B} gives
\[
L_{bir}(\esc{M}^{\mathrm{gp}}) \simeq L_{bir}\Omega\mathrm{B}_{nis}\esc{M} \simeq \Omega L_{bir}\mathrm{B}_{nis}\esc{M} \simeq \Omega\mathrm{B}_{nis}(L_{bir}\esc{M}) \simeq (L_{bir}\esc{M})^{\mathrm{gp}}.\qedhere
\]
\end{proof}

Till now, we have only seen that the birational localization performs excellently in places where $L_{mot}$ suffers. But there is a drawback: because dense open immersions are not stable under arbitrary pullbacks, unlike $\mathbb{A}^1$-projections, while the motivic localization functor $L_{mot}$ is locally cartesian [\cite{hoyois2017six}, Proposition~3.15], $L_{bir}$ is not so in general:

\begin{coex}[$L_{bir}$ is not locally cartesian]\label[coex]{Lbir not locally cart}
Let $k$ be a field and $X$ an abelian $k$-variety of positive dimension, so that $h_k(X)$ is birational-local by \Cref{abelian variety}. Let $0\colon \operatorname{Spec} k\to X$ be the identity section and $j\colon X\setminus \{0\}\hookrightarrow X$ the open complement. Then $j$ is a dense open immersion, hence a birational equivalence, whereas its pullback along $0\colon \operatorname{Spec} k\to X$ is the initial inclusion $\emptyset \hookrightarrow \operatorname{Spec} k$, which is not a birational equivalence.
\end{coex}

Nevertheless, using birational connectivity (\Cref{0 bir connectivity}), we establish that $L_{bir}$ is locally cartesian along morphisms from connected pointed spaces:

\begin{lem}\label[lem]{conn locally cart}
The functor $L_{bir}$ is locally cartesian on Nisnevich-locally connected pointed spaces: given a Nisnevich-locally connected pointed space $\esc{X}\in \mathcal{P}_{nis}(S)_\bullet$, a morphism $\esc{Y}\to \esc{W}$ in $\mathcal{H}^b(S)$, and a morphism $\esc{X}\to \esc{W}$, the canonical map
\[
L_{bir}(\esc{X}\times_{\esc{W}}\esc{Y}) \longrightarrow L_{bir}\esc{X}\times_{\esc{W}} \esc{Y}
\]
is an equivalence.
\end{lem}
\begin{proof}
Factoring $\esc{X}\to \esc{W}$ through the localization unit $\esc{X}\to L_{bir}\esc{X}$, we may replace $\esc{W}$ by $L_{bir}\esc{X}$ and assume that $\esc{X}\to \esc{W}$ is the localization map, so that $L_{bir}\esc{X}\xrightarrow{} \esc{W}$ is an equivalence and $\esc{W}$ is pointed via $\esc{X}$. Set $\esc{Z}:=\esc{X}\times_{\esc{W}} \esc{Y}$; we must show that the projection $L_{bir}\esc{Z}\to \esc{Y}$ is an equivalence.

Let $\Gamma := \esc{Z}\times_{\esc{X}} *$ and $\Lambda := \esc{Y}\times_{\esc{W}} *$ denote the fibers over the basepoints of $\esc{X}$ and $\esc{W}$. By the pasting law for pullbacks, the induced map $\phi\colon \Gamma \to \Lambda$ is an equivalence in $\mathcal{P}_{nis}(S)$, and since $\Lambda\in \mathcal{H}^b(S)$ (\Cref{closed under limits}), we have $\Gamma\simeq L_{bir}\Gamma\xrightarrow{\sim} \Lambda$. Consider the commutative diagram of horizontal fiber sequences in $\mathcal{P}_{nis}(S)$ before and after applying $L_{bir}$:
\[
\xymatrix{
\Gamma \ar[r]\ar[d]_{\phi}^{\rotatebox{90}{$\sim$}} & \esc{Z}\ar[r]\ar[d] & \esc{X}\ar[d] && L_{bir}\Gamma \ar[r]\ar[d]_{L_{bir}\phi}^{\rotatebox{90}{$\sim$}} & L_{bir}\esc{Z}\ar[r]\ar[d] & L_{bir}\esc{X}\ar[d]^{\rotatebox{90}{$\sim$}} \\
\Lambda \ar[r] & \esc{Y}\ar[r] & \esc{W} && \Lambda \ar[r] & \esc{Y}\ar[r] & \esc{W}.
}
\]
Because $\esc{X}$ is Nisnevich-locally connected, the top row on the right is a fiber sequence by \Cref{Lbir preserve connected fiber sequence}. Since $L_{bir}\esc{X}\xrightarrow{\sim}\esc{W}$ is an equivalence, we obtain canonical equivalences
\[
L_{bir}\Gamma \simeq L_{bir}\esc{Z}\times_{L_{bir}\esc{X}} * \xrightarrow{\sim} L_{bir}\esc{Z}\times_{\esc{W}} * \simeq L_{bir}\esc{Z}\times_{\esc{Y}} (\esc{Y}\times_{\esc{W}} *) = L_{bir}\esc{Z}\times_{\esc{Y}} \Lambda,
\]
showing that the left-hand square in the right-hand diagram is a pullback square in $\mathcal{P}_{nis}(S)$. Moreover, by \Cref{0 bir connectivity} (which we shall prove later, but simply reusing \Cref{bir bar commute}), $\esc{W}\simeq L_{bir}\esc{X}$ is Nisnevich-locally connected, so the basepoint inclusion $*\to \esc{W}$ is an effective epimorphism in the $\infty$-topos $\mathcal{P}_{nis}(S)$. Since effective epimorphisms are stable under pullback, $\Lambda = \esc{Y}\times_{\esc{W}} *\to \esc{Y}$ is also an effective epimorphism. Because pullback along an effective epimorphism in an $\infty$-topos is conservative [\cite{lurie2009higher}, Lemma~6.2.3.16] and the base change $L_{bir}\phi\colon L_{bir}\Gamma\xrightarrow{\sim}\Lambda$ of $L_{bir}\esc{Z}\to \esc{Y}$ along $\Lambda\to \esc{Y}$ is an equivalence, we conclude that $L_{bir}\esc{Z}\to \esc{Y}$ is an equivalence.
\end{proof}

\subsection{Birationality vs Null motivic spaces}
In \Cref{form for lbir}, we have constructed a standard formula for birational localization, namely, over geometrically unibranch qcqs schemes with finitely many generic points $L_{bir}\simeq L_{\Sigma}L_{dense}$. However, without an explicit model for dense localization, the referred-to model is insufficient for meaningful computations. It is known stably that, over perfect fields, birational local spectra are equivalent to the zero slice category of Voevodsky $s_0\mathcal{SH}^{S^1}(k)$ [\cite{pelaez2013birational}], and hence having a stable model of $L_{bir}$, at least over perfect fields, is easy [\cite{csf}]. The unstable analog of the zero slice category is the $L^{2,1}$-localization of the motivic homotopy category. If we can establish a way to identify $L_{bir}$ and $L^{2,1}$ unstably, we can model $L_{bir}$ via $L^{2,1}$. For that, we must first set up $L^{2,1}$-localization and construct an explicit model for $L^{2,1}$.

{\small{\textbf{{Motivic null spaces.}}}} We now recall the notion of motivic null spaces from [\cite{asok2023p}]. A motivic space is said to be $S^{p,q}$-null if it is local with respect to the motivic sphere $S^{p,q}$. In the special case of $S^{2,1}$ (and $S^{1,1}$), the definition can be restated as follows:
\begin{defn}
A space $\esc{X}\in \mathcal{P}(S)$ is said to be motivically $S^{2,1}$ (resp. $S^{1,1}$) null if it is $\mathbb{A}^1$-local, Nisnevich-local, and $\mathbb{P}^1$-local (resp. $\mathbb{G}_m$-local). We denote the corresponding localization functors by $$L^{2,1}:\mathcal{P}(S)\to L^{2,1}\mathcal{P}(S)\text{ (and }L^{1,1}:\mathcal{P}(S)\to L^{1,1}\mathcal{P}(S),\text{ respectively)},$$ and write $\mathcal{H}^{\mathbb{P}^1}(S):=L_{\mathbb{P}^1}\mathcal{P}_{nis}(S)$.
\end{defn}
It is clear that $\mathcal{H}^b(S)=\mathcal{H}^{b\mathbb{A}^1}(S)\subset L^{2,1}\mathcal{H}^{\mathbb{A}^1}(S)$. The main goal of this section is to study some aspects of the reverse inclusion, which is completely true stably, over all perfect base fields ([\cite{csf}]). We first note that $L^{2,1}$ can repair non local cartesianness of $L_{bir}$ (\Cref{Lbir not locally cart}), thus strengthening \Cref{conn locally cart}:
\begin{cor}\label[cor]{locally cart on L21 conn}
The functor $L_{bir}$ is locally cartesian on pointed $L^{2,1}$-connected spaces: for any pointed $L^{2,1}$-connected (e.g., $\mathbb{A}^1$-connected) space $\esc{X}$, any morphism $\esc{Y}\to \esc{W}$ in $\mathcal{H}^b(S)$, and any map $\esc{X}\to \esc{W}$, the canonical morphism $L_{bir}(\esc{X}\times_{\esc{W}}\esc{Y})\to L_{bir}\esc{X}\times_{\esc{W}}\esc{Y}$ is an equivalence.
\end{cor}
\begin{proof}
Since $L^{2,1}\colon \mathcal{P}_{nis}(S)\to \mathcal{P}_{nis}(S)$ is locally cartesian (since both the $\mathbb{A}^1$ and $\mathbb{P}^1$ projections are stable under base change) and $\esc{Y},\esc{W}\in \mathcal{H}^b(S)\subset \mathcal{H}^{2,1}(S)$, we have $L^{2,1}(\esc{X}\times_{\esc{W}}\esc{Y})\simeq L^{2,1}\esc{X}\times_{\esc{W}}\esc{Y}$. Using the absorption equivalence $L_{bir}L^{2,1}\simeq L_{bir}$ and the fact that $L^{2,1}\esc{X}$ is Nisnevich-locally connected, the claim follows by applying \Cref{conn locally cart} to $L^{2,1}\esc{X}$.
\end{proof}
\underline{\textit{Model for $L^{2,1}$ and $L^{2,1}$-connectivity.}}
In [\cite{asok2023p}] the authors deduce a connectivity result for $L^{p,q}$ in general; however, their formula for $L^{p,q}$, denoted $(-)_{\infty}^{p,q}$ [Construction 3.1.17 of \textit{loc.\ cit.}], is obtained only via the small object argument. We emphasize that in the special case of $S^{2,1}$ (and $S^{1,1}$) localization, it is possible to construct an explicit model (over an arbitrary base), and, in fact, one may immediately deduce a connectivity result from it. We recall the relevant results from [\cite{p1algtop}] (which apply equally to $\mathbb{G}_m$ localization):
\begin{thm}[\textup{[\cite{p1algtop}, Theorem 2.1.1, Lemma 2.1.4]}]\label{model for P1 localization}
Let $Y/S$ be a smooth scheme and $\esc{X}\in \mathcal{P}(S)$.
\begin{enumerate}
    \item  Let $\delta^\bullet_{\mathbb{P}^1}\times Y$ be the sub-diagram of the slice category $({\mathrm{Sm}_S})_{/Y}$ defined by $$\delta^n_{\mathbb{P}^1}\times Y:=(\mathbb{P}^1_S)^{\times n}\times _S Y,$$ with structure maps given by projection to $Y$. The $\mathbb{P}^1$-localization functor on $\mathcal{P}(S)$ is given by the following colimit formula: $$Y\mapsto  L_{\mathbb{P}^1}\esc{X}(Y) =\colim_{(\delta^\bullet_{\mathbb{P}^1}\times Y)^{op}} \esc{X}({{\delta_{\mathbb{P}^1}^n}\times Y}).$$
    \item The canonical map $\pi_0^{}\esc{X}\to \pi_0^{}L_{\mathbb{P}^1}\esc{X}$ is a surjective morphism of presheaves of sets on $\mathrm{Sm}_S$.
\end{enumerate}
\end{thm}
This leads to the following explicit model of $L^{2,1}$:
\begin{thm}\label[thm]{double localization}
 A model for $S^{2,1}$ motivic localization is given by: $$L^{2,1}:=L_{nis}^{\mathbb{A}^1,\mathbb{P}^1}\simeq \colim _{n\to \infty }(L_{nis}\circ  L_{\mathbb{A}^1}\circ L_{{\mathbb{P}^1}})^n\simeq \colim _{n\to \infty }(L^{\mathbb{A}^1}_{nis}\circ  L_{nis}^{{\mathbb{P}^1}})^n\simeq \colim _{n\to \infty }(L^{\mathbb{P}^1}_{nis}\circ  L_{nis}^{{\mathbb{A}^1}})^n$$
where $L_{\mathbb{P}^1}\esc{X}(Y) =\colim_{(\delta^\bullet_{\mathbb{P}^1}\times Y)^{op}} \esc{X}({{\delta_{\mathbb{P}^1_S}^m}\times_S Y})$ and $L_{\mathbb{A}^1}\esc{Y}(Y) =|\esc{Y}({{\Delta_{\mathbb{A}^1}^\bullet}\times Y})|$. 
\end{thm}
\begin{proof}
    The method is similar to \Cref{form for lbir}. The space $\colim _{n\to \infty }(L^{\mathbb{A}^1}_{nis}\circ  L_{nis}^{{\mathbb{P}^1}})^n\esc{X}$, associated with an $S$-space $\esc{X}$, is the colimit of the even-numbered subdiagram of the following diagram:
        $$ L^{\mathbb{P}^1}_{nis}\esc{X}\to L^{\mathbb{A}^1}_{nis}L^{\mathbb{P}^1}_{nis}\esc{X}\to  L^{\mathbb{P}^1}_{nis}L^{\mathbb{A}^1}_{nis}L^{\mathbb{P}^1}_{nis}\esc{X}\to  \cdots$$ It follows that the colimit of the whole diagram is both $\mathbb{A}^1$-local and Nisnevich-local. Similarly, by considering the odd-numbered (cofinal) subdiagram, it follows that the colimit is $\mathbb{P}^1$-local as well. The same arguments apply verbatim to the other colimit $\colim _{n\to \infty }(L^{\mathbb{P}^1}_{nis}\circ  L_{nis}^{{\mathbb{A}^1}})^n\esc{X}$.
        
        For the first colimit, one considers the $3\mathbb{N}$, $3\mathbb{N}+1$, and $3\mathbb{N}+2$ subdiagrams of the diagram: $$ L_{\mathbb{P}^1}\esc{X}\to L_{\mathbb{A}^1}L_{\mathbb{P}^1}\esc{X}\to L_{nis}L_{\mathbb{A}^1}L_{\mathbb{P}^1}\esc{X}\to  L_{\mathbb{P}^1}L_{nis}L_{\mathbb{A}^1}L_{\mathbb{P}^1}\esc{X}\to  L_{\mathbb{A}^1}L_{\mathbb{P}^1}L_{nis}L_{\mathbb{A}^1}L_{\mathbb{P}^1}\esc{X}\to \cdots$$
        and proceeds similarly.
\end{proof}
\begin{cor}[see \textup{[\cite{asok2023p}, Corollary 3.1.26]} for a version over a field]\label[cor]{null connectivity}
Let $S$ be qcqs and $\esc{X}\in \mathcal{P}(S)$. Then there is a canonical surjection $\pi_0^{nis}\esc{X}\to \pi_0^{nis}L^{2,1}\esc{X}$. This implies that $L^{2,1}$ preserves Nisnevich connectivity.
\end{cor}
\begin{proof}
    From \Cref{model for P1 localization} (2), it follows that $\pi_0^{nis}\esc{X}\to \pi_0^{nis}L_{\mathbb{P}^1}\esc{X}$ is surjective (see [\cite{p1algtop}, Proposition 2.1.5] for complete details). By [\cite{morel19991}, Corollary 2.3.22], $\pi_0^{nis}\esc{X}\to \pi_0^{nis}L_{\mathbb{A}^1}\esc{X}$ is also surjective. 
    
    The claim then follows immediately from the formulas for $L^{2,1}$ given in \Cref{double localization}. For example, $L^{2,1}\esc{X}$ is given by the colimit of the following diagram:
    $$ \esc{X}\to L^{\mathbb{P}^1}_{nis}\esc{X}\to L^{\mathbb{A}^1}_{nis}L^{\mathbb{P}^1}_{nis}\esc{X}\to  L^{\mathbb{P}^1}_{nis}L^{\mathbb{A}^1}_{nis}L^{\mathbb{P}^1}_{nis}\esc{X}\to  \cdots$$ Since $\pi_0^{nis}$ preserves colimits, we see that $\pi_0^{nis}L^{2,1}\esc{X}$ is the colimit of the following diagram 
     $$ \pi_0^{nis}\esc{X}\to \pi_0^{nis}L^{\mathbb{P}^1}_{nis}\esc{X}\to \pi_0^{nis}L^{\mathbb{A}^1}_{nis}L^{\mathbb{P}^1}_{nis}\esc{X}\to  \pi_0^{nis}L^{\mathbb{P}^1}_{nis}L^{\mathbb{A}^1}_{nis}L^{\mathbb{P}^1}_{nis}\esc{X}\to  \cdots$$ Now, inductively applying the observations from the previous paragraph, we know that each map in the last diagram is an epimorphism of Nisnevich sheaves of sets. We conclude that the colimit $\pi_0^{nis}\esc{X}\to \pi_0^{nis}L^{2,1}\esc{X}$ is also an epimorphism.
\end{proof}

{\small{\textbf{{Birationality and contraction in motivic group theory.}}}} To establish the connection between birational local spaces and $S^{2,1}$ (or $S^{1,1}$) null motivic spaces, we will need to establish this for motivic groups first.

Let $G$ be a presheaf of groups on a qcqs scheme $S$. The $\mathbb{G}_m$-contraction of $G$ is defined as $$G_{-1}:=\Omega_{\mathbb{G}_m^{\wedge1}}G=ker \big(G(\mathbb{G}_{m,S}\times_S-)\xrightarrow{ev_1} G\big),$$ where the map $G(\mathbb{G}_{m,S}\times_S-)\xrightarrow{ev_1} G$ is the restriction along the unit section $1: S\to \mathbb{G}_{m,S}$. 
\begin{lem}\label[lem]{dense has trivial contraction}
    If $G$ is a dense local presheaf of groups over $S$, then $G_{-1}=0$.
\end{lem}
\begin{proof}
 By \Cref{bir local is rational local}, $G$ is $\mathbb{G}_m$-local. That is, $G(\mathbb{G}_{m,S}\times_S-)\to G$ is an isomorphism. Since $G_{-1}$ is the kernel of this map, it is trivial.
\end{proof}
\begin{rem}\label[rem]{GENERIC SMOOTHNESS}
In the rest of this subsection, we will assume that the base is a perfect field. The reason is mostly generic smoothness. In particular, given a dense open immersion $U\hookrightarrow X$ one can find a filtration of $X$ by open subschemes $U=U_n\subset U_{n-1}\subset \cdots\subset U_1\subset U_0=X$ such that for all $i$, $U_{i} \setminus U_{i-1} =:C_i$ is smooth and has trivial normal bundle in $U _i$ (see [Remark 5.2, \cite{levine2010slices}]). In fact, if $cod_{X}(X\setminus U)\ge d$, then it is clear that $cod_{U_i}(C_i)\geq d$ for all $i$. Thus, given a presheaf, any problem related to the injectivity or surjectivity of its restriction maps along open immersions can be reduced to checking the same on open immersions with a smooth complement having a trivial normal bundle. 
\end{rem}

\begin{lem}\label[lem]{gm contraction 0}
     Suppose $G$ is a presheaf of (not necessarily abelian) groups over a perfect field $k$. Then $G$ is birational if and only if it is a strongly $\mathbb{A}^1$-invariant Nisnevich sheaf with $G_{-1}=1$.
\end{lem} 
\begin{proof}
The only if part follows from \Cref{bir is strong a1} and \Cref{dense has trivial contraction}.

Conversely, let $G$ be a strongly $\mathbb{A}^1$-invariant Nisnevich sheaf of groups with $G_{-1}=1$. Since such sheaves are weakly unramified [\cite{bachmann2024strongly}, Corollary 2.4], i.e., are injective on dense open immersions, to show that $G$ is birational, it suffices to show that the restriction map $G(X)\to G(U)$ is surjective for a dense open immersion $U\hookrightarrow X$. Moreover, by \Cref{GENERIC SMOOTHNESS}, we may assume that the complement $Z=X\setminus U$ is smooth and has a trivial normal bundle in $X$. In this situation, purity gives a cofiber sequence $$Th_Z(N_ZX)\to \Sigma_+ U\to \Sigma_+ X,$$ where $N_ZX$ is the normal bundle of $Z$ in $X$. Mapping this sequence to $\mathrm{B}^{nis}G$ (which is given to be $\mathbb{A}^1$-local) in $\mathcal{H}^{\mathbb{A}^1}(k)_\bullet$, we are left to show that $$[Th_Z(N_ZX), \mathrm{B}^{nis}G]_{\mathcal{H}^{\mathbb{A}^1}{(k) }}$$ is trivial. Since $U\hookrightarrow X$ is dense, the normal bundle has positive rank, say $n\geq 1$. Because $Z$ has trivial normal bundle in $X$, i.e., there is an isomorphism $N_ZX\simeq Z\times \mathbb{A}^n_k$, we have $$Th_Z(N_ZX)\simeq Z_+\wedge T^n\simeq ^{mot} \Sigma^nZ_+\wedge \mathbb{G}^{\wedge n}.$$ So we see that indeed, 
\begin{flalign*}
  [Th_Z(N_ZX), \mathrm{B}^{nis}G]_{\mathcal{H}^{\mathbb{A}^1}(k)_\bullet }&\simeq[\Sigma^nZ_+\wedge \mathbb{G}^{\wedge n}, \mathrm{B}^{nis}G]_{\mathcal{H}^{\mathbb{A}^1}{(k)_\bullet }}\simeq [\Sigma^{n-1}Z_+\wedge \mathbb{G}^{\wedge n}, G]_{\mathcal{H}^{\mathbb{A}^1}{(k) _\bullet}}\\&\simeq [\Sigma^{n-1}Z_+\wedge \mathbb{G}^{\wedge n-1}, \Omega_\mathbb{G}G]_{\mathcal{H}^{\mathbb{A}^1}{(k) }_\bullet}\simeq pt \text{ (since }\Omega_\mathbb{G}G=G_{-1}=1). 
\end{flalign*}
\end{proof}
{\small{\textbf{Birationality and the Nisnevich Postnikov tower.}}} To extend the above group theoretic results to arbitrary spaces, we would need to climb the Postnikov ladder within the Nisnevich topos. But unfortunately, the analog of \Cref{pi is birational} does not work in $\mathcal{P}_{nis}$ as birationality of the Nisnevich path component $\pi_0^{nis} $ does not guarantee the birationality of the ordinary path component, for example: 
\begin{coex}\label[coex]{nis sheafification cannot detect birationality}
   There are non-birational presheaves of sets in $PSh_{\Sigma}(\mathrm{Sm}_S)$ (and even in $PSh_{\Sigma}^{\mathbb{A}^1}(\mathrm{Sm}_S)$) whose Nisnevich sheafification is birational. A general class of such examples is as follows: take a non-birational sheaf of groups $G$ in $Grp^{\mathbb{A}^1}_{k}$ (say, $\mathbb{G}_m$). Let $F:=H^1_{nis}(-;G)$ (for $\mathbb{G}_m$ this is the Picard group functor). Since $G$ is (strongly) $\mathbb{A}^1$-invariant, we have $F(\mathbb{A}^1_X)=F(X)$, so $F\in PSh_{\Sigma}^{\mathbb{A}^1}(\mathrm{Sm}_S)$. Moreover, its Nisnevich sheafification $L_{nis}F$ is trivial, hence birational. But clearly, $F$ itself is not birational: $$F(\mathbb{P}^1_k)=[ \Sigma\mathbb{G}_m, \mathrm{B}^{nis}G]_{\mathcal{H}^{\mathbb{A}^1}}=[\mathbb{G}_m, G]_{\mathcal{H}^{\mathbb{A}^1}}=G(\mathbb{G}_m)=X^*G\neq pt=F(\mathbb{A}^1_k).$$
\end{coex}
The good news is that, once the triviality $\pi_0^{nis}$ is assumed, $L_{bir} $ is closed under extensions over such spaces just as $L_{mot}$:
\begin{lem}\label[lem]{total space of bir seq}
    Suppose $\Gamma \to \esc{X}\to \esc{Y}$ is a fiber sequence in $\mathcal{P}_{nis}(S)$ (or $\mathcal{P}(S)$), with $\esc{Y}$ a pointed Nisnevich-locally (or sectionwise) connected space. If $\Gamma$ and $\esc{Y}$ are birational local, then so is $\esc{X}$.
\end{lem}
\begin{proof}
As in the proof of \Cref{Lbir preserve connected fiber sequence},
 we know that $\esc{X}\simeq \mathrm{B}_{nis}(\Gamma, \Omega \esc{Y}, pt)$, which is birational local by \Cref{bir bar commute} (2), since both $\Gamma$ and $\Omega\esc{Y}$ are birational local.
\end{proof}
\begin{rem}\label[rem]{morel's methods of fiber seq}
    In the Nisnevich topology, the standard method for proving such statements (see, for example, [\cite{morel2012a1}, Lemma~6.51]) is not applicable here. Indeed, when the method in \textit{loc.\ cit.} is adapted to the birational setup, it requires one to trace the following map of fiber sequences of Nisnevich-local spaces, induced by a nontrivial dense open immersion $U\hookrightarrow X$:
    \[
    \xymatrix{    \Gamma^X\ar[r]\ar[d]^{j^*_{\Gamma}}&\esc{X}^X\ar[r]\ar[d]^{j^*_{\esc{X}}}& \esc{Y}^X\ar[d]^{j^*_{\esc{Y}}}\\
     \Gamma^{U}\ar[r]&\esc{X}^{U}\ar[r]& \esc{Y}^{U}
    }
    \]
    But the problem is that even if $\esc{Y}$ is Nisnevich connected, there is no reason for either $\esc{Y}^X$ or $\esc{Y}^U$ to be so. So the morphism $\Gamma^{U}\to\esc{X}^{U}$ need not be an effective epimorphism.
\end{rem}

\begin{prop}\label[prop]{sigma pi is birational}
Suppose $S$ is a qcqs scheme. For a space $\esc{X}\in \mathcal{P}_{nis}(S)$, consider the following conditions: 
    \begin{enumerate}
        \item $\esc{X}$ is birational local.
        \item $\tau_{\leq i}^{nis}\esc{X}$ is birational local for all $i\geq 0$.
        \item $\pi_i^{nis}\esc{X}$ is a birational presheaf for all $i\geq 0$.
    \end{enumerate} 
Then:
\begin{enumerate}[label=(\alph*)]
    \item $1\implies 2\implies 3$. 
    \item If $\esc{X}\in \tau^{nis}_{\geq1}\mathcal{P}_{nis}(\mathrm{Sm}_S)_{\bullet}$ then $3\implies 2$. 
    \item If $\mathcal{P}_{nis}(S)$ is Postnikov complete (e.g., when $S$ has finite Krull dimension) then $2\implies 1$.
\end{enumerate}
In particular, when $S$ is a qcqs scheme of finite Krull dimension and $\esc{X}$ is pointed and Nisnevich connected, the three conditions 1, 2, and 3 are equivalent.
\end{prop}
\begin{proof}
(a). Assume (1) holds. Then, by \Cref{pi is birational}, we know that $\tau_{\leq i}\esc{X}$ is birational local. Since birational local spaces are automatically Nisnevich local [\Cref{bir local is nis local}] it follows that $\tau_{\leq i}\esc{X}\simeq L_{nis}\tau_{\leq i}\esc{X}=:\tau^{nis}_{\leq i}\esc{X}$ is birational local, so (2) holds. Performing the same line of arguments with $\pi_i\esc{X}$ instead of $\tau_{\leq i}\esc{X}$ we see that (1)$\implies $(3). Applying (1)$\implies $(3) to $\esc{Y}_i:=\tau_{\leq i}^{nis}\esc{X}$ yields (2)$\implies$ (3).

(b). Assume (3) holds. Consider the first slice of the Postnikov tower in $\mathcal{P}_{nis}(S)$:
$$\mathrm{B}_{nis}(\pi^{nis}_1(\esc{X}))\to \tau^{nis}_{\leq1}\esc{X}\to\tau^{nis}_{\leq0}\esc{X}=pt.$$
Since $\pi^{nis}_1\esc{X}$ is birational, by \Cref{B_bir is B} we know that $\mathrm{K}^{nis}(\pi^{nis}_1(\esc{X}),1)\simeq \tau^{nis}_{\leq 1}\esc{X}$ is birational local. Inductively, assume that $\tau^{nis}_{\leq n}\esc{X}$ is birational local. In the fiber sequence
$$\mathrm{B}_{nis}(\pi^{nis}_{n+1}(\esc{X}))\to \tau^{nis}_{\leq n+1}\esc{X}\to\tau^{nis}_{\leq n}\esc{X},$$
the base $\tau^{nis}_{\leq n}\esc{X}$ is Nisnevich-locally connected and the fiber $\mathrm{B}_{nis}(\pi^{nis}_{n+1}(\esc{X}))$ is birational local by \Cref{B_bir is B}. By \Cref{total space of bir seq}, $\tau^{nis}_{\leq n+1}\esc{X}$ is birational local, proving (3) $\implies $ (2).    

(c). Assume all truncations $\tau^{nis}_{\leq n}\esc{X}$ are birational local. By Postnikov completeness of $\mathcal{P}_{nis}(S)$ (which holds when $S$ is qcqs of finite Krull dimension by [\cite{MR4296353}, Theorem 3.18] and [\cite{lurie2009higher}, Proposition 7.2.1.10]), we have $\esc{X}\simeq \plim_{n}\tau^{nis}_{\leq n}\esc{X}\in \mathcal{H}^b(S)$ by \Cref{closed under limits}.
\end{proof}

Combining the results of the last two subsubsections, we are now ready to show that the reverse inclusion of $\mathcal{H}^b(S)=\mathcal{H}^{b\mathbb{A}^1}(S)\subset L^{2,1}\mathcal{H}^{\mathbb{A}^1}(S)$ does hold over $L^{2,1}\mathcal{H}^{\mathbb{A}^1}(S)_{\geq 1}$ over perfect fields:
\begin{thm}\label{detecting n birationality by gm contraction}
Let $k$ be a perfect field. The following conditions on a Nisnevich-locally connected pointed motivic space $\esc{X}$ over $k$ are equivalent:
    \begin{enumerate}
        \item $\esc{X}$ is birational local.
        \item $\tau_{\leq i}^{nis}\esc{X}$ is birational local for all $i\geq 0$.
        \item $\pi_i^{nis}\esc{X}$ is birational for all $i\geq 0$.
        \item $(\pi_i^{nis}\esc{X})_{-1}=0$ for all $i\geq  0$.
        \item $\esc{X}$ is $\mathbb{G}_m$ local (equivalently, $S^{1,1}$ null).
        \item $\esc{X}$ is $\mathbb{P}^1$ local (equivalently, $S^{2,1}$ null).
    \end{enumerate}
In other words, $\tau^{nis}_{\geq 1}\mathcal{H}^b(k)= \tau^{nis}_{\geq 1}\big(\mathcal{H}^{\mathbb{A}^1}(k)\mcap\mathcal{H}^{\mathbb{P}^1}(k)\big)$.
\end{thm}
\begin{proof}
    The equivalence of the first three conditions is \Cref{sigma pi is birational}. (3)$\iff$(4) follows from \Cref{gm contraction 0} since $\pi_i^{nis}\esc{X}$ are strongly $\mathbb{A}^1$-invariant [\cite{morel2012a1}, Corollary 6.2]. (6)$\iff$(4)$\iff$ (5) follows from [\cite{asok2023p}, Corollary 3.1.21]. Indeed, for a motivic space, $\mathbb{G}_m$-locality is equivalent to $S^{1,1}$-nullity, which, by [\cite{asok2023p}, Corollary 3.1.21], is equivalent to the vanishing $(\pi_i^{nis}\esc{X})_{-1}=0$ for $i\geq 0$ (the case $i=0$ is trivially true due to the connectivity of $\esc{X}$). Similarly, $\mathbb{P}^1$-locality is equivalent to $S^{2,1}$-nullity, which, by the conditions in \textit{loc.\ cit.}, is equivalent to the vanishing $(\pi_i^{nis}\esc{X})_{-1}=0$ for $i\geq 1$ (again, the case $i=0$ is trivially true by the connectivity of $\esc{X}$).
\end{proof}
\begin{rem}
    Note that the results referred to in the proof, namely [\cite{asok2023p}, Corollary 3.1.21], do not use any explicit model for nullification and follow immediately from Morel's result: $\Omega_{\mathbb{G}}\pi_i^{\mathbb{A}^1}\simeq \pi_i^{\mathbb{A}^1}\Omega_\mathbb{G}$ [\cite{morel2012a1}, Theorem 6.13]. The following counterexample shows that the connectivity assumption in \Cref{detecting n birationality by gm contraction} cannot be omitted.
\end{rem}
\begin{coex}[to the possibility that $\mathcal{H}^{\mathbb{A}^1}\mcap \mathcal{H}^{\mathbb{P}^1}\subset \mathcal{H}^b$]\label[coex]{coexamp}
\mbox{}

1. It is well known that the stack of finite \'etale schemes and isomorphisms, denoted $\mathrm{F\acute{e}t}:=\mathrm{F\acute{E}t}^{\simeq}$, is Nisnevich-local and $\mathbb{P}^1$-local (see, for example, [\cite{p1algtop}, Example 2.2.18]); in fact, $\mathrm{F\acute{E}t}$ has higher homotopies and is also a cdh sheaf since it satisfies descent for proper surjective morphisms [\cite{10.1007/BFb0058664}, Theorem 4.12]. When the base is the field of complex numbers, we claim that this stack (viewed as a presheaf of groupoids on $Sm_\mathbb{C}$) is $\mathbb{A}^1$-local as well. To see this, it suffices to show that for any smooth scheme $X/\mathbb{C}$, the canonical map $ \mathrm{F\acute{E}t}(X)\to \mathrm{F\acute{E}t}(X\times_\mathbb{C} \mathbb{A}^1_\mathbb{C})$, induced by the projection $X\times _\mathbb{C}\mathbb{A}^1_\mathbb{C}\to X$ is an equivalence. This follows from the Riemann existence theorem [\cite{milne1980etale}, Theorem 3.4] and the $\mathbb{C}$-invariance of finite coverings of topological spaces. However, $\mathrm{F\acute{e}t}$ is not birational local. For example, consider the dense open immersion $\mathbb{G}_{m,\mathbb{C}}\hookrightarrow \mathbb{A}^1_\mathbb{C}$. By [\cite[\href{https://stacks.math.columbia.edu/tag/0BQF}{Lemma 0BQF}]{stacks-project}], we know that $\mathrm{F\acute{E}t}({\mathbb{A}^1_\mathbb{C}})\to \mathrm{F\acute{E}t}({\mathbb{G}_{m,\mathbb{C}}})$ is fully faithful. Hence, if $\mathrm{F\acute{e}t}({\mathbb{A}^1_\mathbb{C}})\to \mathrm{F\acute{e}t}({\mathbb{G}_{m,\mathbb{C}}})$ were an equivalence of groupoids, then $\mathrm{F\acute{E}t}({\mathbb{A}^1_\mathbb{C}})\to \mathrm{F\acute{E}t}({\mathbb{G}_{m,\mathbb{C}}})$ would be so as well. But by the Riemann Existence Theorem again, this is impossible, since $\pi_1^{top}(\mathbb{G}_{m,\mathbb{C}}(\mathbb{C}))=\mathbb{Z}$ while $\pi_1^{top}(\mathbb{A}^1_{\mathbb{C}}(\mathbb{C}))=0$. The problem is that $\mathrm{F\acute{e}t}$ is not Nisnevich locally connected. Indeed, $\mathrm{F\acute{e}t}(\mathbb{C})=\mathrm{F\acute{E}t}(\mathbb{C})^{\simeq }=FinSet^{\simeq}$; so $\pi_0^{nis}\mathrm{F\acute{e}t}$ evaluated at the base point $Spec\mathbb{C}$ equals $\pi_0(FinSet^{\simeq})=\mathbb{Z}\neq *$.

2. There is an even simpler (discrete, non-cdh) counterexample. Namely, the scheme $\mathbb{G}_m$ is both $\mathbb{A}^1$-local and $\mathbb{P}^1$-local (see [\cite{p1algtop}, Example 2.2.13]). But $\mathbb{G}_m$ is not birational local; it is in fact birationally contractible. The problem is that $\mathbb{G}_m$ is disconnected. More generally, one can establish that every discrete motivic space is birational local [\cite{p1algtop}, Example 2.2.7].
\end{coex}   

The disconnected case does work for representables by proper varieties:
\begin{prop}\label[prop]{rationally rigid is bir rigid}
  Over a field $k$, a proper variety with no rational curves (equivalently, $L^{2,1}$-local varieties) is birational local.
 \end{prop}
 \begin{proof}
 Indeed, suppose $T$ is a proper variety over $k$ with no rational curves, and we are given a morphism $f: U\to T$ with $U\in \mathrm{Sm}_k$. Then by [\cite{debarre2001higher}, Corollary 1.44], $f$ extends along the dense open immersion $U\hookrightarrow X\in \mathrm{Sm}_k$. 
     
To elaborate, consider the graph of $f$, given by $\Gamma_f\subset U\times T\subset X\times T$. Let $\bar{\Gamma} _f\subset X\times T$ be the closure of $\Gamma_f$ in $X\times T$, and let $\phi: \bar{\Gamma}_f\to X$ be the corresponding projection map, which is proper since $T$ is so. Since $U\hookrightarrow X$ is dense and $U\cong \Gamma_f$, we get that $\phi $ is birational. Since $X$ is smooth, the irreducible components of the exceptional fiber are ruled (by Proposition 1.43 in \textit{loc.\ cit.}). Since $T$ contains no rational curve, each fiber of $\phi$ must map to a single point in $T$. Because $\phi$ is surjective by construction, we get an extension to all of $X$.   
 \end{proof}
\begin{rem}
For certain grouplike spaces (strongly $\mathbb{A}^1$-invariant and strongly $\mathbb{P}^1$-invariant spaces), one also does not need a connectivity assumption [\cite{p1algtop}, Theorem 2.3.9]. Similarly, stably, an $S^1$-spectrum is birational local iff it is $\mathbb{A}^1$-local and $\mathbb{P}^1$-local [\cite{pelaez2013birational}, Theorem 3.11.]. There is also a version for chain complexes with transfers known due to Ayoub [\cite{ayoub2020P1}, Proposition 3.9].
\end{rem}
{\small{\textbf{On finding a model for \texorpdfstring{$L_{bir}$}{}}.}}
We will now describe a model for $L_{bir}$ on $\mathbb{A}^1$-connected spaces.
\begin{thm}\label{model for L_bir on connected}
    Over a perfect field $k$, the restriction of $L_{bir}: \mathcal{P}(k)_\bullet\to \mathcal{H}^b(k)$ to $L^{2,1}$-connected objects (resp. $L^{1,1}$-connected objects) is equivalent to the restriction of $L^{2,1}$ (resp. $L^{1,1}$) to those objects. Consequently, on $L^{2,1}$-connected objects: $$L_{bir}\simeq \colim _{n\to \infty }(L_{nis}\circ  L_{\mathbb{A}^1}\circ L_{{\mathbb{P}^1}})^n\simeq \colim _{n\to \infty }(L^{\mathbb{A}^1}_{nis}\circ  L_{nis}^{{\mathbb{P}^1}})^n\simeq \colim _{n\to \infty }(L^{\mathbb{P}^1}_{nis}\circ  L_{nis}^{{\mathbb{A}^1}})^n$$
where $L_{\mathbb{P}^1}\esc{X}(Y) =\colim_{(\delta^\bullet_{\mathbb{P}^1}\times Y)^{op}} \esc{X}({{\delta_{\mathbb{P}^1_S}^m}\times_S Y})$ and $L_{\mathbb{A}^1}\esc{Y}(Y) =|\esc{Y}({{\Delta_{\mathbb{A}^1}^\bullet}\times Y})|$. 
    
\end{thm}
\begin{proof}
    Suppose $\esc{X}$ is a space with a chosen base point such that $L^{2,1}\esc{X}$ (or $L^{1,1}\esc{X}$) is Nisnevich connected. By \Cref{detecting n birationality by gm contraction}, $L^{2,1}\esc{X}$ (or $L^{1,1}\esc{X}$) is birational local. To conclude, it suffices to note that $L^{2,1}$ (or $L^{1,1}$) equivalences are birational equivalences. Indeed, then the localization map $\esc{X}\to L^{2,1}\esc{X}$ (or $\esc{X}\to L^{1,1}\esc{X}$) is a birational equivalence to a birational local space, and hence is a model of birational localization. But since $\mathbb{A}^1,\mathbb{P}^1,\mathbb{G}_m$ are rational it is obvious that $L^{2,1}$ and $L^{1,1}$ equivalences are birational equivalences (\Cref{Motivic equivalences are Birational equivalences}(2)). The formulas then follow from \Cref{double localization}.
\end{proof}
\begin{cor}\label[cor]{bir conn on perf}
In the previous theorem, it suffices to assume that $\esc{X}$ is $\mathbb{A}^1$-connected, $\mathbb{P}^1$-connected, or even Nisnevich-connected. In particular, over a perfect field, $L_{bir}$ preserves Nisnevich local connectivity.
\end{cor}
\begin{proof}
By \Cref{null connectivity}, Nisnevich connectivity implies $L^{2,1}$-connectivity (and gives $L_{bir}\esc{X}\simeq L^{2,1}\esc{X}$ Nisnevich connected). When $\esc{X}$ is $\mathbb{A}^1$-connected (resp. $\mathbb{P}^1$-connected), replace it with the Nisnevich-connected space $L_{mot}\esc{X}$ (resp. $L_{\mathbb{P}^1}^{nis}\esc{X}$).
\end{proof}

\subsection{Connectivity of \texorpdfstring{$L_{bir}$}{Lbir}}
While the model of $L_{bir}$, and hence its connectivity (\Cref{bir conn on perf}) in the previous subsection, applies to connected spaces only over perfect fields, the commutation of $L_{bir}$ with the bar construction (\Cref{bir bar commute}) yields a stronger connectivity result over any qcqs base scheme $S$. We start with a primary observation:
\begin{cor}
    A birational local space is birationally connected if and only if it is $\mathbb{A}^1$-connected if and only if it is Nisnevich-connected if and only if it is sectionwise connected.
\end{cor}
\begin{proof}
Clearly, if $\esc{X}$ is birational local, then $\pi_0^b\esc{X}\cong\pi_0\esc{X}\cong\pi_0^{nis}\esc{X}$ (see \Cref{nis sheaves of bir is bir} for the last isomorphism). By \Cref{Motivic equivalences are Birational equivalences}, it follows that $\esc{X}\simeq L_{mot}\esc{X}$. Applying $\pi_0^{nis}$ to this equivalence yields $\pi_0^{nis}\esc{X}\cong \pi^{\mathbb{A}^1}_0\esc{X}$. Thus $\pi_0^b\esc{X}\cong \pi_0\esc{X}\cong\pi_0^{nis}\esc{X}\cong \pi^{\mathbb{A}^1}_0\esc{X}$.
\end{proof}
We shall say that a space $\esc{X}$ is birationally connected if $L_{bir}\esc{X}$ is (sectionwise/Nisnevich locally/$\mathbb{A}^1$)-connected.

\begin{cor}\label[cor]{0 bir connectivity}
Over an arbitrary qcqs base $S$ and for all $n\geq -2$, $L_{bir}$ takes Nisnevich-locally (or sectionwise) $n$-connected pointed $S$-spaces to sectionwise $n$-connected $S$-spaces.
\end{cor}
\begin{proof}
The case $n=-2$ is vacuous, and the case $n=-1$ is immediate since any pointed $S$-space (in particular $L_{bir}\esc{X}$) admits a section over the terminal object $S\in Sm_S$, hence over every $U\in Sm_S$. For $n\geq 0$, from the loop-path fiber sequence, the existence of a base point for $\esc{X}\in\tau_{\geq 1}^{nis}\mathcal{P}_{nis}(S)$ yields an equivalence $\esc{X}\simeq\mathrm{B}_{nis} \Omega(\esc{X})$. By induction, if $\esc{X}$ is Nisnevich-locally $n$-connected, then $\esc{X}\simeq \mathrm{B}^{n+1}_{nis}\Omega^{n+1}\esc{X}$. Hence,
$$L_{bir}\esc{X}\simeq L_{bir}\mathrm{B}_{nis}^{n+1}\Omega^{n+1}(\esc{X})\simeq \mathrm{B}^{n+1}L_{bir}\Omega^{n+1}(\esc{X})$$
[by \Cref{bir bar commute}]. By construction, $\mathrm{B}^{n+1}$ takes values in (sectionwise) $n$-connected objects and we are done.
\end{proof}

\begin{cor}\label[cor]{bir preserve a1 n conn}
For all $n\geq -2$, $L_{bir}$ takes pointed $\mathbb{A}^1$-$n$-connected spaces to (sectionwise) $n$-connected spaces. In particular, pointed $\mathbb{A}^1$-connected spaces are birationally connected.
\end{cor}
\begin{proof}
    Assume $\esc{X}$ is pointed $\mathbb{A}^1$-$n$-connected, i.e., $\esc{Y}:=L_{mot}\esc{X}$ is a Nisnevich $n$-connected pointed space. By \Cref{0 bir connectivity}, $L_{bir}\esc{Y}$ is sectionwise $n$-connected. Since motivic equivalences are birational equivalences (see \Cref{Motivic equivalences are Birational equivalences}), we have $$L_{bir}\esc{X}\simeq  L_{bir}L_{mot}\esc{X}=L_{bir}\esc{Y}.$$ It thus follows that $L_{bir}\esc{X}$ is sectionwise $n$-connected.
\end{proof}

\begin{rem}\label[rem]{Lbir is not effective}  
     The surjectivity of $\pi_0^{nis}$ for the $\mathbb{A}^1$-motivic localization [\cite{morel19991}, Corollary 3.22] ensures that $\eta: 1\to L_{mot}$ is an objectwise effective epimorphism. However, for birational localization, the canonical map $$\pi_0^{nis}\esc{X}\to \pi_0^{nis}L_{bir}\esc{X}=\pi_0L_{bir}\esc{X}$$ is not surjective (see the counterexample below), unlike its $\mathbb{A}^1$-motivic counterpart.
\end{rem}  

\begin{coex}\label[coex]{Lbir not effective coex}
    In this counterexample, we show that $L_{bir}$ is not an effective localization (of the Nisnevich $\infty$-topos), i.e., $$\pi_0^{nis}\esc{X}\to \pi_0^{nis}L_{bir}\esc{X}$$ need not be an epimorphism of sheaves in general. Of course, over finite fields, it is easy to find counterexamples: 
    \begin{enumerate}
    \item Let $S=\operatorname{Spec}k$ for $k= \mathbb{F}_q$. Consider the generic point inclusion $X\hookrightarrow \mathbb{P}^1_{\mathbb{F}_q}$. Since the set of closed points $\mathbb{P}^1(\mathbb{F}_q)$ of $ \mathbb{P}^1_{\mathbb{F}_q}$ is finite, $X=\mathbb{P}^1_{\mathbb{F}_q}\setminus \mathbb{P}^1(\mathbb{F}_q)$ and is hence smooth. So $X$ is a dense open subscheme of $\mathbb{P}^1_{\mathbb{F}_q}$ and hence birationally contractible. Thus, the map $\pi_0^{nis}X\to \pi_0^{nis}L_{bir}X$ in question can be identified with the canonical projection $X\to \operatorname{Spec}k$, which is clearly not surjective, since, by construction, $X$ has no $k$-rational points. This example certainly fails for infinite fields.
    \item For an example that works in all scenarios, consider $U:=X\setminus 0\hookrightarrow X$ for an abelian variety $X/k$ of positive dimension as in \Cref{Lbir not locally cart}. Since $X$ is birationally rigid (\Cref{abelian variety}), $L_{bir}U\simeq X$, so the map $\pi_0^{nis}U=U\hookrightarrow X=\pi_0^{nis}X\simeq \pi_0^{nis}L_{bir}U$ is an injection rather than a surjection. 
    \end{enumerate}
\end{coex}

\begin{rem}\label[rem]{about pointing}
    When the base scheme $S$ is the spectrum of a henselian local ring (so that $S$ admits no nontrivial Nisnevich cover), $\pi_0^{nis}\esc{X}(S)=\pi_0\esc{X}(S)$, and hence we can remove the pointed condition from the connectivity statements above:
\end{rem}

\begin{cor}\label[cor]{bir preserves a1 conn}
    Let $S$ be the spectrum of a henselian local ring. Then, for all $n\geq -2$:
    \begin{enumerate}
        \item $L_{bir}$ takes Nisnevich-locally $n$-connected $S$-spaces to sectionwise $n$-connected $S$-spaces.
        \item $L_{bir}$ takes $L^{2,1}$-$n$-connected $S$-spaces to sectionwise $n$-connected $S$-spaces.
        \item $L^{2,1}$-connected spaces over $S$ are birationally connected.
    \end{enumerate}
\end{cor}
\begin{proof}
    The case $n=-2$ is vacuous. For $n\geq -1$, because $S$ is henselian local, evaluation at $S$ preserves Nisnevich epimorphisms and satisfies $\pi_0\esc{X}(S)\simeq \pi_0^{nis}\esc{X}(S)$ for any $\esc{X}\in \mathcal{P}(S)$. Hence, if $\esc{X}$ is Nisnevich-locally $n$-connected, then $\pi_0^{nis}\esc{X}\to *$ is a Nisnevich epimorphism, so $\pi_0\esc{X}(S)\simeq \pi_0^{nis}\esc{X}(S)\neq \emptyset$. Choosing a section $S\to \esc{X}$ makes $\esc{X}$ pointed, and (1) follows from \Cref{0 bir connectivity}. For (2), apply (1) to $\esc{Y}:=L^{2,1}\esc{X}$ using $L_{bir}\esc{X}\simeq L_{bir}\esc{Y}$ (\Cref{Motivic equivalences are Birational equivalences}). Claim (3) is the $n=0$ case of (2).
\end{proof}
\begin{prop}[Birational $(-1)$-connectivity and ample fields]\label[prop]{large-field-points}
    Let $k$ be a field. Then $L_{bir}$ preserves $(-1)$-connectivity. It detects $(-1)$-connectivity (i.e., a space $\esc{X}\in\mathcal{P}(k)$ satisfies $\esc{X}(k)\neq\emptyset$ iff $(L_{bir}\esc{X})(k)\neq\emptyset$) if and only if $k$ is ample (i.e., every smooth $k$-curve with a $k$-point has infinitely many $k$-points). In particular, over an ample field $k$, any birationally connected $k$-space admits a $k$-point (despite \Cref{Lbir not effective coex}).
\end{prop}
\begin{proof}
     Preservation is the $n=-1$ case of \Cref{bir preserves a1 conn}. If $k$ is not ample, take a smooth integral $k$-curve $C$ with $C(k)$ nonempty and finite; then $U:=C\setminus C(k)\hookrightarrow C$ is a dense open immersion in $Sm_k$, so $(L_{bir}U)(k)\simeq (L_{bir}C)(k)\neq\emptyset$ while $U(k)=\emptyset$.

    Conversely, assume $k$ is ample. We claim that the empty sky-scraper sheaf, $G_k:=skysc_k(\emptyset)$ is a $(-1)$-truncated birational space. Indeed, by definition $G_k(W)\simeq \mathrm{Map}_{\mathrm{Spc}}(W(k),\emptyset)$ is $*$ if $W(k)=\emptyset$ and $\emptyset$ otherwise. So $G_k\in \mathcal{P}_\Sigma(Sm_k)$ and is $(-1)$-truncated.

    For dense locality, let $U\hookrightarrow W$ be a dense open immersion in $Sm_k$ with $W$ irreducible and $x\in W(k)$. Assume $\dim W=d\ge 1$ and $x\in Z:=W\setminus U$. Pick an open neighborhood $W'\subseteq W$ of $x$ with an \'etale morphism $\varphi:W'\to \mathbb{A}^d_k$ sending $x\mapsto 0$. Since $k$ is infinite, choose nonzero $v\in (\mathbb{A}^d_k\setminus \overline{\varphi(W'\cap Z)})(k)$ and let $L\subset \mathbb{A}^d_k$ be the line through $0$ and $v$. The connected component $C_0$ of $W'\times_{\mathbb{A}^d_k} L$ containing $x$ is a smooth integral $k$-curve with $x\in C_0(k)$ and finite $C_0\cap Z$. Ampleness forces $(C_0\setminus Z)(k)\neq\emptyset$, hence $U(k)\neq\emptyset$. Thus $G_k(W)\xrightarrow{\sim} G_k(U)$, so $G_k\in \mathcal{H}^b(k)_{\le -1}$.

    Now, if $L_{bir}\esc{X}$ is $(-1)$-connected, i.e., $\tau_{\le -1}^{bir}L_{bir}\esc{X}\simeq *$ (equivalently, $\tau_{\le -1}L_{bir}\esc{X}\simeq *$), then 
    \[
        \mathrm{Map}_{\mathrm{Spc}}(\esc{X}(k),\emptyset)\simeq \mathrm{Map}_{\mathcal{P}(k)}(L_{bir}\esc{X},G_k)\simeq \mathrm{Map}_{\mathcal{P}(k)}(*,G_k)\simeq G_k(\operatorname{Spec}k)=\emptyset,
    \]
   (as $skysc_k$ is the right adjoint to the global section functor $\esc{F}\mapsto \esc{F}(\operatorname{Spec}k)$) whence $\esc{X}(k)\neq\emptyset$. 
\end{proof}
\subsection{Birational coverings}\label{3.5}

We write $\pi_i^{bir}$ for the internal homotopy groups of $\mathcal{H}^{b}(S)$ [\cite{inhot}, \S2.1], i.e., $\pi_i^{bir}:=\tau_{\leq 0}^{bir}\Omega^i$. On the other hand, as usual, we define the external homotopy groups as $\pi_i^b:=\pi_iL_{bir}$, which we call the birational homotopy group presheaves. By \Cref{pi is birational} and \Cref{bir local is nis local}, these presheaves on $\mathrm{Sm}_S$ are automatically Nisnevich sheaves, so $\pi_i^b=\pi_i^{nis}L_{bir}=\pi_i^{bir}L_{bir}$ (\Cref{birational truncation}(1)).

\medskip
\noindent{\small{\textbf{Birationality and connective covers:}}} Let $\esc{X}\in \mathcal{P}_{nis}(S)_\bullet$ and let
\[
\tau^{nis}_{>n}\esc{X}\longrightarrow \esc{X}\longrightarrow \tau^{nis}_{\leq n}\esc{X}
\]
denote the $n$-th Postnikov fiber sequence of $\esc{X}$ in $\mathcal{P}_{nis}(S)$.

\begin{lem}\label[lem]{n conn cover of bir}
If $\esc{X}\in \mathcal{H}^b(S)_\bullet$ is a birational-local space, then $\tau^{nis}_{>n}\esc{X}\simeq \tau_{>n}\esc{X}$ is also birational-local.
\end{lem}
\begin{proof}
Since $\esc{X}$ is birational-local, $\tau_{\leq n}\esc{X}$ is birational-local by \Cref{pi is birational}, and hence so is the fiber $\tau_{>n}\esc{X}$. By \Cref{bir local is nis local}, this yields that $\tau_{>n}\esc{X}$ is Nisnevich-local. Thus $\tau^{nis}_{>n}\esc{X}:=L_{nis}\tau_{>n}\esc{X}\simeq \tau_{>n}\esc{X}\in \mathcal{H}^b(S)$.
\end{proof}

Since $\tau^{nis}_{>n}L_{bir}\esc{X}$ is birational-local by \Cref{n conn cover of bir}, the natural transformation $\tau^{nis}_{>n}\to \tau_{>n}^{nis} L_{bir}$ descends to a canonical comparison map $L_{bir}\tau^{nis}_{>n}\to \tau^{nis}_{>n}L_{bir}$. Similarly, \Cref{sigma pi is birational} furnishes a canonical map $L_{bir}\tau^{nis}_{\leq n}\to \tau^{nis}_{\leq n}L_{bir}$.

\begin{prop}\label[prop]{Lbir vs connective covers}
Let $\esc{X}\in \mathcal{P}_{nis}(S)_\bullet$ be a pointed Nisnevich-locally $0$-connected $S$-space such that for some $n\geq 0$ the truncation $\tau^{nis}_{\leq n}\esc{X}$ is birational-local (which holds precisely when $\pi_i^{nis}\esc{X}$ is birational for all $i\leq n$; see \Cref{sigma pi is birational}). Then for all $i\leq n$, the canonical morphisms
\[
L_{bir}\tau^{nis}_{\leq i}\esc{X}\longrightarrow \tau^{nis}_{\leq i}L_{bir}\esc{X}\quad\text{and}\quad L_{bir}\tau^{nis}_{>i}\esc{X}\longrightarrow \tau^{nis}_{>i}L_{bir}\esc{X}
\]
are equivalences.
\end{prop}
\begin{proof}
Since $\tau^{nis}_{\leq i}\esc{X}$ is Nisnevich-locally connected, \Cref{Lbir preserve connected fiber sequence} yields a fiber sequence
\[
L_{bir}\tau^{nis}_{> i}\esc{X}\longrightarrow L_{bir}\esc{X}\longrightarrow L_{bir}\tau^{nis}_{\leq i}\esc{X}.
\]
For $i\leq n$, the hypothesis that $\tau^{nis}_{\leq n}\esc{X}$ is birational-local implies $L_{bir}\tau^{nis}_{\leq i}\esc{X}\simeq \tau^{nis}_{\leq i}\esc{X}$. Because $\tau^{nis}_{> i}\esc{X}$ is Nisnevich-locally $i$-connected, $L_{bir}\tau^{nis}_{> i}\esc{X}$ is sectionwise $i$-connected by \Cref{0 bir connectivity}, whereas $L_{bir}\tau^{nis}_{\leq i}\esc{X}\simeq \tau^{nis}_{\leq i}\esc{X}$ is $i$-truncated. The long exact sequence of homotopy groups therefore identifies $L_{bir}\tau^{nis}_{\leq i}\esc{X}\xrightarrow{\sim} \tau^{nis}_{\leq i}L_{bir}\esc{X}$. Comparing the fiber sequence above with the $i$-th Postnikov fiber sequence of $L_{bir}\esc{X}$,
\[
\xymatrix{
L_{bir}\tau^{nis}_{> i}\esc{X}\ar[r]\ar[d] & L_{bir}\esc{X}\ar[d]^{\rotatebox{90}{$\sim$}}\ar[r] & L_{bir}\tau^{nis}_{\leq i}\esc{X}\ar[d]^{\rotatebox{90}{$\sim$}} \\
\tau^{nis}_{> i}L_{bir}\esc{X}\ar[r] & L_{bir}\esc{X}\ar[r] & \tau^{nis}_{\leq i}L_{bir}\esc{X},
}
\]
shows that $L_{bir}\tau^{nis}_{>i}\esc{X}\xrightarrow{} \tau^{nis}_{>i}L_{bir}\esc{X}$ is also an equivalence.
\end{proof}

\textit{Birational Freudenthal suspension theorem:} (From the methods of [\cite{morel2012a1}, Theorem~6.61]).

\begin{thm}[Birational Freudenthal suspension theorem]\label{Birational Freudenthal suspension theorem}
Assume that $S$ is qcqs and that $\esc{E}$ is a pointed $S$-space that is birationally $n$-connected for some $n\geq 1$. Then:
\begin{enumerate}
\item The canonical adjoint unit map $\eta^b_{\esc{E}}\colon L_{bir}\esc{E}\to \Omega L_{bir} \Sigma \esc{E}$ has $2n$-connected fiber.
\item The adjoint unit map $\eta_{\esc{E}}\colon \esc{E}\to \Omega \Sigma \esc{E}$ has a $2n$-connected birational fiber. If $\esc{E}$ is connected, then, in fact, the ordinary fiber of $\eta_{\esc{E}}$ is birationally $2n$-connected.
\item The canonical map $L_{bir}\esc{E}\to \Omega^\infty \Sigma^\infty_{bir} \esc{E}$ has $2n$-connected fiber. Here $\Sigma_{bir}:=L_{bir} \Sigma$ is the internal suspension of $\mathcal{H}^b(S)$; equivalently,
\[
\Sigma_{bir}\esc{X}=\colim{}^{\mathcal{H}^b(S)}(pt\leftarrow L_{bir}\esc{X}\to pt).
\]
\end{enumerate}
\end{thm}
\begin{proof}
(1) Since $L_{bir}\esc{E}$ is sectionwise $n$-connected, the classical Freudenthal suspension theorem applied to $L_{bir}\esc{E}$ yields that the fiber of the unit $\eta_{L_{bir}\esc{E}}\colon L_{bir}\esc{E}\to \Omega \Sigma L_{bir}\esc{E}$ is $2n$-connected. That is, the morphism
\[
f\colon \pi_{2n+1}^b\esc{E}=\pi_{2n+1}L_{bir}\esc{E}\longrightarrow \pi_{2n+2}\Sigma L_{bir}\esc{E}
\]
is an epimorphism of presheaves, and
\[
\tau_{\leq 2n}L_{bir}\esc{E}\xrightarrow{\tau_{\leq 2n}\eta_{L_{bir}\esc{E}}}\tau_{\leq 2n}\Omega \Sigma L_{bir}\esc{E}\simeq \Omega \tau_{\leq 2n+1}\Sigma L_{bir}\esc{E}
\]
is an equivalence. Since the space $\tau_{\leq 2n+1}\Sigma L_{bir}\esc{E}$ is pointed and connected, this equivalence implies
\[
\tau_{\leq 2n+1}\Sigma L_{bir}\esc{E}\simeq \mathrm{B}\Omega \tau_{\leq 2n+1}\Sigma L_{bir}\esc{E}\simeq \mathrm{B}\tau_{\leq 2n}L_{bir}\esc{E}.
\]
Since $\mathrm{B}\tau_{\leq 2n}L_{bir}\esc{E}$ is birational-local by \Cref{pi is birational} and \Cref{bir bar commute}(2), so is $\tau_{\leq 2n+1}\Sigma L_{bir}\esc{E}$. By \Cref{Lbir vs connective covers}, we obtain
\[
\tau_{\leq 2n+1}\Sigma L_{bir}\esc{E}\simeq L_{bir}\tau_{\leq 2n+1}\Sigma L_{bir}\esc{E}\simeq \tau_{\leq 2n+1}L_{bir}\Sigma L_{bir}\esc{E}\simeq \tau_{\leq 2n+1}L_{bir}\Sigma\esc{E}.
\]
Applying $\Omega$ and composing with the initial equivalence $\tau_{\leq 2n}L_{bir}\esc{E}\simeq \Omega \tau_{\leq 2n+1}\Sigma L_{bir}\esc{E}$ yields
\[
\tau_{\leq 2n}L_{bir}\esc{E}\simeq \Omega \tau_{\leq 2n+1}L_{bir}\Sigma\esc{E}\simeq \tau_{\leq 2n}\Omega L_{bir}\Sigma\esc{E}.
\]

On the other hand, birational locality of $\tau_{\leq 2n+1}\Sigma L_{bir}\esc{E}$ implies, by \Cref{higher pii as universal birationalization}, that $\pi_{2n+2}\Sigma L_{bir}\esc{E}\to \pi_{2n+2}^b \Sigma L_{bir}\esc{E}\simeq \pi_{2n+2}^b\Sigma\esc{E}$ is the universal morphism to a birational abelian group sheaf. Hence, the composite
\[
\pi_{2n+1}^b\esc{E}\xrightarrow{\;f\;} \pi_{2n+2}\Sigma L_{bir}\esc{E}\longrightarrow \pi_{2n+2}^b \Sigma\esc{E}
\]
is an epimorphism in $\mathrm{Ab}^b_S$. Since $\mathrm{Ab}^b_S\subset \mathrm{PAb}_S$ is exact (\Cref{bir is exact in presheaves}), $\pi_{2n+1}L_{bir}\esc{E}\to \pi_{2n+1}\Omega L_{bir} \Sigma\esc{E}$ is an epimorphism in $\mathrm{PAb}_S$, showing that $\eta^b_{\esc{E}}\colon L_{bir}\esc{E}\to \Omega L_{bir}\Sigma\esc{E}$ has $2n$-connected fiber.

(2) Since $\Sigma \esc{E}$ is connected, \Cref{bir and Omega commute on connected} identifies the codomain of $\eta^b_{\esc{E}}$ with $L_{bir}\Omega \Sigma \esc{E}$, giving the first statement of (2). When $\esc{E}$ is connected, $\Omega \Sigma\esc{E}$ is also connected, so \Cref{Lbir preserve connected fiber sequence} implies $L_{bir}\mathrm{fib}(\eta_{\esc{E}})\simeq \mathrm{fib}(L_{bir}\eta_{\esc{E}})$, yielding the second statement.

(3) The stabilization map $L_{bir}\esc{E}\to \Omega^\infty \Sigma^\infty_{bir} \esc{E}$ is the colimit of the sequential diagram
\[
L_{bir}\esc{E}\xrightarrow{\eta^b_{\esc{E}}} \Omega\Sigma_{bir}\esc{E}\xrightarrow{\Omega\eta^b_{\Sigma_{bir}\esc{E}}} \Omega^2\Sigma^2_{bir}\esc{E}\longrightarrow \cdots
\]
By (1), each map in this diagram has $2n$-connected fiber, and hence so does the colimit.
\end{proof}

\noindent{\small\textbf{Internal birational $n$-coverings.}}
\Cref{n conn cover of bir} shows that the connective covering functors restrict to endofunctors of $\mathcal{H}^b(S)$. More generally, following [\cite{constantin2026higher}] and [\cite{inhot}, \S2.4, \S3.4], for $n\ge -2$ we define the $\infty$-category of \emph{internal birational $n$-coverings} of $\esc{X}\in \mathcal{H}^b(S)$ by
\[
\mathrm{Cov}_{\mathcal{H}^b}^n(\esc{X}):=\big(\mathcal{H}^b(S)_{/\esc{X}}\big)_{\leq n},
\]
with pointed counterpart $\mathrm{Cov}_{\mathcal{H}^b}^n(\esc{X})_\bullet := \big(\mathcal{H}^b(S)_{\bullet/\esc{X}}\big)_{\leq n}$ for $\esc{X}\in \mathcal{H}^b(S)_\bullet$ ([\cite{inhot}, Definition~2.4.1 and Lemma~2.4.2]). When $S$ is geometrically unibranch with finitely many generic points, $\mathcal{H}^b(S)$ is an $\infty$-topos (\Cref{Hbir a topos}) and [\cite{constantin2026higher}] applies directly; below we treat an arbitrary qcqs base $S$.

\begin{thm}\label{initial object of bir cov}
Let $\esc{X}\in \mathcal{H}^b(S)_\bullet$ be a pointed birational motivic space (not necessarily connected) and let $n\ge 0$.
\begin{enumerate}
    \item The $\infty$-category $\mathrm{Cov}_{\mathcal{H}^{b}}^n(\esc{X})_\bullet$ has an initial object $\widetilde{\esc{X}}_n^{bir}\simeq \tau_{>n+1}\esc{X}\to \esc{X}$ (written $\widetilde{\esc{X}}^{bir}$ when $n=0$).
    \item This initial object is uniquely characterized as the pointed internal birational $n$-covering of $\esc{X}$ with birationally $(n+1)$-connected source (i.e., $\tau_{\leq n+1}^{bir}\widetilde{\esc{X}}_n^{bir}\simeq *$), and is preserved by the inclusions
    \[
    \mathrm{Cov}_{\mathcal{H}^{b}}^n(\esc{X})_\bullet\subset \mathrm{Cov}_{\mathcal{P}_{nis}}^n(\esc{X})_\bullet\subset \mathrm{Cov}_{\mathcal{P}}^n(\esc{X})_\bullet.
    \]
    \item There are canonical equivalences
    \[
    \tau_{\leq n}^{bir}\Omega\esc{X}\simeq \Omega\tau_{\leq n+1}^{bir}\esc{X}
    \simeq \underline{\mathrm{Aut}}_{\mathcal{H}^b(S)_{/\esc{X}}}(\widetilde{\esc{X}}_n^{bir})
    \simeq \underline{\mathrm{Aut}}_{\mathcal{P}_{nis}(S)_{/\esc{X}}}(\widetilde{\esc{X}}_n^{bir}),
    \]
    which reduce to $\mathrm{B}_{bir}^n\pi_{n+1}^{bir}(\esc{X})\simeq \mathrm{B}^n\pi_{n+1}^{bir}(\esc{X})$ when $\esc{X}$ is birationally $n$-connected (and unconditionally to $\pi_1^{bir}(\esc{X})$ when $n=0$).
\end{enumerate}
\end{thm}
\begin{proof}
By \Cref{birational truncation}, $L_{bir}$ is path-injective and $m$-fundamental on every object for all $m\ge -2$; it is also cartesian (\Cref{Lbir is cartesian}). Thus [\cite{inhot}, Proposition~2.4.22 and Lemma~2.4.23] together with \Cref{bir em} yield (1)-(3).
\end{proof}

\begin{rem}
Unlike its $\mathbb{A}^1$-motivic counterpart [\cite{inhot}, Theorem~3.4.13], \Cref{initial object of bir cov} requires neither connectivity of $\esc{X}$ nor a perfect base field, reinforcing that $\mathcal{H}^b(S)$ behaves internally like an $\infty$-topos over an arbitrary qcqs scheme $S$.
\end{rem}

\begin{cor}[Compare \textup{[\cite{inhot}, Proposition~3.4.14]}]\label[cor]{simply conn cover cor}
Let $n\ge 0$ and let $\esc{X}\in \mathcal{P}_{nis}(S)_{\geq 1}$ be pointed, with universal Nisnevich $n$-covering $\widetilde{\esc{X}}_n := \tau_{>n+1}^{nis}\esc{X}\to \esc{X}$. If $\tau_{\leq n+1}^{nis}\esc{X}$ is birational-local (equivalently, $\pi_i^{nis}\esc{X}$ is birational for all $1\le i\le n+1$ by \Cref{sigma pi is birational}), then $L_{bir}\widetilde{\esc{X}}_n\to L_{bir}\esc{X}$ is the universal internal birational $n$-covering $\widetilde{(L_{bir}\esc{X})}_n^{bir}\to L_{bir}\esc{X}$, and
\[
\tau_{\leq n}^{nis}\Omega\esc{X}\simeq \tau_{\leq n}^{bir}\Omega L_{bir}\esc{X}\simeq \underline{\mathrm{Aut}}_{\mathcal{H}^b(S)_{/L_{bir}\esc{X}}}(L_{bir}\widetilde{\esc{X}}_n)
\]
(giving $\pi_1^{nis}\esc{X}\simeq \pi_1^b\esc{X}\simeq \underline{\mathrm{Aut}}_{\mathcal{H}^b(S)_{/L_{bir}\esc{X}}}(L_{bir}\widetilde{\esc{X}}_0)$ when $n=0$).
\end{cor}
\begin{proof}
Combine \Cref{Lbir vs connective covers} (for $n+1$) with \Cref{initial object of bir cov}.
\end{proof}

\medskip
\noindent{\small\textbf{External birational $n$-coverings, monodromy, and the $n$-th Galois correspondence.}}
We now pass from internal $n$-coverings of birational spaces to external birational $n$-coverings of arbitrary $S$-spaces $\esc{X}\in \mathcal{P}(S)$.

\begin{defn}[\textup{[\cite{inhot}, Definition~2.4.10]}]
Let $n\ge -2$. A morphism $f\colon\esc{Y}\to \esc{X}$ in $\mathcal{P}(S)$ is a \emph{birational fibration} ($f\in \mathcal{F}ib_{\esc{X}}^{b}$) if it is local in $\mathcal{P}(S)_{/\esc{X}}$ with respect to birational equivalences, and a \emph{birational $n$-covering} ($f\in \mathrm{Cov}_{b}^n(\esc{X})$) if it lies in $(\mathcal{F}ib_{\esc{X}}^{b})_{\leq n}$ (a birational covering $\mathrm{Cov}_b(\esc{X}):=\mathrm{Cov}_b^0(\esc{X})$ when $n=0$). For pointed $\esc{X}$, set $\mathrm{Cov}_b^n(\esc{X})_\bullet := \mathrm{Cov}_b^n(\esc{X})\times_{\mathcal{P}(S)_{/\esc{X}}}\mathcal{P}(S)_{\bullet/\esc{X}}$.
\end{defn}

\begin{prop}
For every $n\ge -2$, a morphism between birational motivic spaces is a birational $n$-covering if and only if it is an $\mathbb{A}^1$-$n$-covering, if and only if it is a Nisnevich $n$-covering. In general, on Nisnevich-local $S$-spaces:
\[
\textup{Birational-}n\textup{-covering}\subset \mathbb{A}^1\textup{-}n\textup{-covering}\subset \textup{Nisnevich-}n\textup{-covering}.
\]
\end{prop}
\begin{proof}
The first claim is [\cite{inhot}, Proposition~2.4.4 and Proposition~2.4.15], and the containment follows from \Cref{Motivic equivalences are Birational equivalences}.
\end{proof}

\begin{prop}\label[prop]{birational torsor}
Let $n\ge 0$, $1\le m\le n+1$, and $G\in m\text{-}\mathrm{Grp}_S^b$. Any principal $\mathrm{B}_{scdh}^{m-1}G$-bundle $\esc{Y}\to \esc{X}$ (i.e., a pullback of $*\to \mathrm{B}_{scdh}^m G$; for $m=1$, any scdh $G$-torsor) is a birational $n$-covering.
\end{prop}
\begin{proof}
By \Cref{B_bir is B}, \Cref{bir em}(2), and \Cref{bir is cdh}, $\mathrm{B}_{scdh}^m G\simeq \mathrm{B}^m G\in \mathcal{H}^b(S)$, so $m\text{-}\mathrm{Grp}^{L_{bir}} = m\text{-}\mathrm{Grp}_S^b$ and [\cite{inhot}, Corollary~2.4.16] applies.
\end{proof}

\begin{rem}
For every $n\ge 0$, the inclusion $\textup{Birational-}n\textup{-covering}\subset \mathbb{A}^1\textup{-}n\textup{-covering}$ is strict: over a perfect field $k$, principal $\mathrm{B}_{nis}^{m-1}G$-bundles for any non-birational $G\in m\text{-}\mathrm{Grp}^{L_{mot}}$ (such as $\mathbb{G}_m$ or $\underline{\mathrm{K}}_r^M$ for $r\ge 1$, e.g.\ $\mathbb{A}^{r+1}\setminus\{0\}\to \mathbb{P}^r$ for $n=0$ or $*\to \mathrm{B}_{nis}^{n+1}\mathbb{G}_m$) are $\mathbb{A}^1$-$n$-coverings [\cite{inhot}, Corollary~3.4.7] whose fibers are not birational-local.
\end{rem}
\begin{rem}

Similar to Morel's remark [\cite{morel2012a1}, Remark 7.3], one might seek a lifting characterization for birational coverings. By \Cref{Motivic equivalences are Birational equivalences}, two potential criteria arise: a "birational motivic" criterion (lifting along dense open immersions and the zero section of $\mathbb{A}^1$) and a "naive" criterion (lifting only along dense open immersions). The naive criterion fails: since finite morphisms are birational over the image of a generic section and Zariski's Main Theorem implies that every generic section extends (uniquely) to a global section. Yet, finite morphisms need not be $\mathbb{A}^1$-coverings or birational coverings. For instance, when $char(k)\neq 2$, the ramified map $(-)^2:\mathbb{A}^1_k\to\mathbb{A}^1_k$ admits no lift along the birational equivalence $1:k\to\mathbb{A}^1_k$ in the following diagram:
\[
\xymatrix{k\ar[r]^{-1}\ar[d]_{1}&\mathbb{A}^1_k\ar[d]^{(-)^2}\\
\mathbb{A}^1_k\ar[r]_{id}&\mathbb{A}^1_k}
\]
The birational motivic criterion also fails. The counterexample to Morel's remark given in [\cite{thorcounter}, Proposition 3.22], applies here. Specifically, there is a finite morphism (a closed immersion) $Spec \mathbb{C}\to T$ over $k=\mathbb{C}$ that lifts along the zero section of $\mathbb{A}^1$ ([\cite{thorcounter}, Theorem 3.24]). Being finite, it also lifts along all dense open immersions by the preceding argument. However, it is not an $\mathbb{A}^1$-covering ([\cite{thorcounter}, Proposition 3.22]), and thus not a birational covering.

\end{rem}

Because $L_{bir}$ is neither locally cartesian (\Cref{Lbir not locally cart}) nor effective (\Cref{Lbir not effective coex}) in general, bridging internal and external birational $n$-coverings requires restricting to pointed $L^{2,1}$-connected spaces, where both properties hold:

\begin{cor}[Compare \textup{[\cite{inhot}, Corollary~3.4.8 and Theorem~3.4.15]}]\label{Lbir covering is covering of Lbir}
Let $\esc{X}\in \mathcal{P}(S)_\bullet$ be a pointed $L^{2,1}$-connected space.
\begin{enumerate}
    \item For every $n\ge -2$, pullback along $\eta_{bir}\colon 1\to L_{bir}$ induces equivalences of $\infty$-categories
    \[
    \eta^*_{bir}\colon \mathrm{Cov}_{\mathcal{H}^{b}}^n(L_{bir}\esc{X})\;\leftrightarrows\; \mathrm{Cov}_{b}^n(\esc{X}) : L_{bir}
    \quad\text{and}\quad
    \eta^*_{bir}\colon \mathrm{Cov}_{\mathcal{H}^{b}}^n(L_{bir}\esc{X})_\bullet\;\leftrightarrows\; \mathrm{Cov}_{b}^n(\esc{X})_\bullet : L_{bir}.
    \]
    \item For every $n\ge 0$, $\mathrm{Cov}_b^n(\esc{X})_\bullet$ admits an initial object $\widetilde{\esc{X}}_n^b\to \esc{X}$ (the unique pointed birational $n$-covering with birationally $(n+1)$-connected source), satisfying
    \[
    \tau_{\leq n}^{bir}\Omega L_{bir}\esc{X}\simeq \Omega\tau_{\leq n+1}^{bir}L_{bir}\esc{X}\simeq \underline{\mathrm{Aut}}_{\mathcal{P}(S)_{/\esc{X}}}(\widetilde{\esc{X}}_n^b)
    \]
    (reducing to $\mathrm{B}_{bir}^n\pi_{n+1}^b(\esc{X})$ when $\esc{X}$ is birationally $n$-connected, and to $\pi_1^b(\esc{X})$ when $n=0$).
\end{enumerate}
\end{cor}
\begin{proof}
By \Cref{locally cart on L21 conn}, $L_{bir}$ is locally cartesian on $\esc{X}$. Moreover, since $L^{2,1}\esc{X}$ is pointed and Nisnevich-locally connected and $L_{bir}\simeq L_{bir}L^{2,1}$, \Cref{0 bir connectivity} implies that $L_{bir}\esc{X}$ is sectionwise connected. Hence $*\to L_{bir}\esc{X}$ is an effective epimorphism in $\mathcal{P}(S)$ [\cite{lurie2009higher}, Proposition~7.2.1.14], and by right cancellation [\cite{lurie2009higher}, Corollary~6.2.3.12] so is $\eta_{bir}\colon \esc{X}\to L_{bir}\esc{X}$. Now (1) follows from [\cite{inhot}, Proposition~2.4.17], and (2) follows from [\cite{inhot}, Proposition~2.4.22 and Lemma~2.4.23] (using \Cref{Lbir is cartesian} and \Cref{initial object of bir cov}).
\end{proof}

\begin{prop}[Monodromy classification; compare \textup{[\cite{inhot}, Proposition~3.4.11]}]\label[prop]{mono class}
Let $n\ge 0$ and let $(\esc{Y},y)\in \mathcal{P}(S)_\bullet$ be pointed and birationally connected. Then the fiber functor along $L_{bir}y$ yields an equivalence of $\infty$-categories
\[
\Gamma_y^{bir,n}\colon \mathrm{Cov}_{\mathcal{H}^b}^n(L_{bir}\esc{Y})\xrightarrow{\;\sim\;} \mathrm{Act}_{\tau_{\leq n}^{bir}\Omega(L_{bir}\esc{Y},y)}(\mathcal{H}_{\leq n}^b(S))
\]
(for $n=0$, $\mathrm{Cov}_{\mathcal{H}^b}(L_{bir}\esc{Y})\simeq \pi_1^b(\esc{Y},y)\text{-}\mathrm{Shv}_S^b$). If $\esc{Y}$ is moreover $L^{2,1}$-connected, this extends via \Cref{Lbir covering is covering of Lbir} to an equivalence $\Gamma_y^{b,n}\colon \mathrm{Cov}_b^n(\esc{Y})\xrightarrow{\sim} \mathrm{Act}_{\tau_{\leq n}^{bir}\Omega(L_{bir}\esc{Y},y)}(\mathcal{H}_{\leq n}^b(S))$.
\end{prop}
\begin{proof}
Since $L_{bir}$ is $(n+1)$-fundamental (\Cref{birational truncation}) and $\mathcal{H}^b(S)$ is closed under extensions over connected objects (\Cref{total space of bir seq}), the internal equivalence is [\cite{inhot}, Lemma~2.4.19, Lemma~2.4.20, and Proposition~2.4.21]; the external extension follows from \Cref{Lbir covering is covering of Lbir}.
\end{proof}

\begin{thm}[Birational $n$-Galois correspondence; compare \textup{[\cite{inhot}, Theorem~3.4.16]}]\label{galois cor}
Let $n\ge 0$ and let $\esc{X}\in \mathcal{P}(S)_\bullet$ be pointed and birationally $n$-connected.
\begin{enumerate}
    \item There is an equivalence of posets
    \[
    \mathrm{Sub}_{(n+1)\text{-}\mathrm{Grp}^{b}_S}(\pi_{n+1}^{b}\esc{X})\simeq \mathrm{Cov}_{\mathcal{H}^{b}(S)}^n(L_{bir}\esc{X})_\bullet^{>_{bir}n}
    \]
    between birational $(n+1)$-subgroup sheaves of $\pi_{n+1}^b\esc{X}$ and pointed internal birational $n$-coverings of $L_{bir}\esc{X}$ with birationally $n$-connected source.
    \item If $\esc{X}$ is moreover $L^{2,1}$-connected (for instance, $\mathbb{A}^1$-$n$-connected, or simply $L^{2,1}$-connected when $n=0$), pullback along $\eta_{bir}$ induces a further equivalence of posets
    \[
    \mathrm{Cov}_{\mathcal{H}^{b}(S)}^n(L_{bir}\esc{X})_\bullet^{>_{bir}n}\simeq \mathrm{Cov}_{b}^n(\esc{X})^{>_{b}n}_\bullet
    \]
    with pointed birational $n$-coverings of $\esc{X}$ having birationally $n$-connected source.
\end{enumerate}
Moreover, for every $n\ge 0$, the normal part of the correspondence holds unconditionally over any qcqs scheme $S$ (contrast [\cite{inhot}, Remark~3.4.17]).
\end{thm}
\begin{proof}
For (1), \Cref{birational truncation} shows that $L_{bir}$ is path-injective and $(n+1)$-fundamental on every object, \Cref{total space of bir seq} ensures that $\mathcal{H}^b(S)$ is extension-closed over the connected space $L_{bir}\esc{X}$, and \Cref{bir em}(2) identifies $(n+1)\text{-}\mathrm{Grp}^{L_{bir}}$ with $(n+1)\text{-}\mathrm{Grp}_S^b$; hence [\cite{inhot}, Theorem~2.4.24(1)] applies directly. For (2), when $\esc{X}$ is $L^{2,1}$-connected, $L_{bir}$ is locally cartesian and effective on $\esc{X}$ (as in the proof of \Cref{Lbir covering is covering of Lbir}), so [\cite{inhot}, Theorem~2.4.24(2)] yields the second equivalence.

For the normal part, given a normal subgroup $H\trianglelefteq G:=\pi_{n+1}^b\esc{X}$ in $(n+1)\text{-}\mathrm{Grp}_S^b$, the sectionwise quotient $G/H$ is dense-local and preserves finite products, hence already lies in $(n+1)\text{-}\mathrm{Grp}_S^b$. By \Cref{B_bir is B}, $ \mathrm{B}^{n+1}(G/H)\in \mathcal{H}^b(S)$, so $\mathrm{B}^{n+1}H\to \mathrm{B}^{n+1}G\to \mathrm{B}^{n+1}(G/H)$ is a fiber sequence in $\mathcal{H}^b(S)$. Thus $E_H\to L_{bir}\esc{X}$ (and its pullback to $\esc{X}$ in (2)) is the principal $\mathrm{B}^{n}(G/H)$-bundle classified by $L_{bir}\esc{X}\to \mathrm{B}^{n+1}G\to \mathrm{B}^{n+1}(G/H)$.
\end{proof}

\medskip
\noindent{\small\textbf{Birational geometric coverings and the Higher Van Kampen Theorem.}}
Dual to $n$-coverings, we recall from [\cite{inhot}, Definitions~2.4.25 and~2.4.36] that a morphism $f\colon \esc{Y}\to \esc{X}$ in $\mathcal{P}(S)$ is an \emph{external birational effective epimorphism} if $L_{bir}\check{C}_\bullet(f)\to L_{bir}\esc{X}$ is a colimit diagram in $\mathcal{H}^b(S)$, and an \emph{internal birational effective epimorphism} if $L_{bir}f$ is an effective epimorphism in $\mathcal{H}^b(S)$ (i.e., $\check{C}_\bullet^{\mathcal{H}^b}(L_{bir}f)\to L_{bir}\esc{X}$ is a colimit diagram in $\mathcal{H}^b(S)$). A family $\{\esc{U}_i\to \esc{X}\}_{i\in I}$ in $\mathcal{P}(S)$ is called an \emph{(external/internal) birational geometric covering} if $\coprod_{i\in I}\esc{U}_i\to \esc{X}$ is an (external/internal) birational effective epimorphism. Since $L_{bir}\simeq L_{bir}L_{mot}$ and birational spaces are scdh-local (\Cref{bir is cdh}), we have
\begin{flalign*}
   \textup{Nisnevich-geometric-cover}\subset \mathbb{A}^1\textup{-geometric-cover}\subset &\textup{External birational-geometric-cover}\\&\supset \textup{scdh covering}. 
\end{flalign*}

In contrast to $\mathcal{H}^{\mathbb{A}^1}(S)$, the closure of $\mathcal{H}^b(S)\subset \mathcal{P}(S)$ under sifted colimits forces effective epimorphisms in $\mathcal{H}^b(S)$ to be completely rigid and detected sectionwise on $\pi_0$:

\begin{prop}[Compare \textup{[\cite{inhot}, Propositions~3.4.19 and~3.4.23, and Lemma~3.4.25]}]\label[prop]{bir eff epi}
Let $S$ be a qcqs scheme, and let $f\colon \esc{X}\to \esc{Y}$ be a morphism in $\mathcal{H}^b(S)$. The following conditions are equivalent:
\begin{enumerate}
    \item $f$ is an effective epimorphism in $\mathcal{H}^b(S)$ (in the sense of [\cite{inhot}, Definition~2.4.25]).
    \item $f$ is an effective epimorphism in $\mathcal{P}(S)$ (equivalently, in $\mathcal{P}_\Sigma(S)$, $\mathcal{P}_{nis}(S)$, or $\mathcal{H}^{\mathbb{A}^1}(S)$).
    \item $\pi_0(f)=\pi_0^{bir}(f)\colon \pi_0\esc{X}\to \pi_0\esc{Y}$ is a sectionwise surjection (i.e., an epimorphism in $\mathrm{PSh}_S$).
    \item $\pi_0^{bir}(f)\colon \pi_0^{bir}\esc{X}\to \pi_0^{bir}\esc{Y}$ is an effective epimorphism in $\mathrm{Shv}^b_S$.
\end{enumerate}
Consequently, for an arbitrary morphism $g\colon \esc{U}\to \esc{V}$ in $\mathcal{P}(S)$:
\begin{enumerate}[label=(\alph*)]
    \item $g$ is an internal birational effective epimorphism if and only if $L_{bir}g$ is an effective epimorphism in $\mathcal{P}(S)$, if and only if $\pi_0^b(g)\colon \pi_0^b\esc{U}\to \pi_0^b\esc{V}$ is sectionwise surjective.
    \item Every external birational effective epimorphism is an internal birational effective epimorphism. In particular, $L_{bir}$ preserves effective epimorphisms (both to $\mathcal{H}^b(S)$ and as an endofunctor of $\mathcal{P}(S)$ and $\mathcal{P}_{nis}(S)$).
\end{enumerate}
\end{prop}
\begin{proof}
Because each inclusion of $\mathcal{H}^b(S)$ into $\mathcal{H}^{\mathbb{A}^1}(S)$, $\mathcal{P}_{nis}(S)$, $\mathcal{P}_\Sigma(S)$, and $\mathcal{P}(S)$ is closed under both small limits (\Cref{closed under limits}) and sifted colimits (\Cref{bir is closed under sifted colimits}), and $\Delta^{\mathrm{op}}$ is sifted [\cite{lurie2009higher}, Lemma~5.5.8.4], the \v{C}ech nerve $\check{C}_\bullet(f)$ and its geometric realization $|\check{C}_\bullet(f)|$ coincide across all five categories. Full faithfulness therefore yields (1)$\iff$(2).

The equivalence (2)$\iff$(3) is [\cite{lurie2009higher}, Proposition~7.2.1.14] in the $\infty$-topos $\mathcal{P}(S)$, combined with $\pi_0|_{\mathcal{H}^b(S)}=\pi_0^{bir}$ (\Cref{birational truncation}(1)). For (3)$\iff$(4), a morphism $q\colon A\to B$ in $\mathrm{PSh}_S$ is sectionwise surjective if and only if its reflexive kernel pair $A\times_B A\rightrightarrows A\to B$ is a coequalizer diagram. Since reflexive coequalizers are sifted colimits and $\mathrm{Shv}^b_S\subset \mathrm{PSh}_S$ is closed under limits and sifted colimits (\Cref{bir is closed under sifted colimits}, and \Cref{birational truncation}(1)), this diagram is a coequalizer in $\mathrm{Shv}^b_S$ if and only if it is one in $\mathrm{PSh}_S$.

Part (a) applies (1)$\iff$(2)$\iff$(3) to $f=L_{bir}g$. For (b), if $g\colon \esc{U}\to \esc{V}$ is an external birational effective epimorphism, then $|L_{bir}\check{C}_\bullet(g)|\xrightarrow{\sim} L_{bir}\esc{V}$ holds in $\mathcal{H}^b(S)$ and hence sectionwise in $\mathcal{P}(S)$ (\Cref{bir is closed under sifted colimits}). Applying $\pi_0$ gives $\mathrm{coeq}\big(\pi_0^b(\esc{U}\times_{\esc{V}}\esc{U})\rightrightarrows \pi_0^b\esc{U}\big)\cong \pi_0^b\esc{V}$ in $\mathrm{PSh}_S$, so $\pi_0^b(g)$ is sectionwise surjective and (a) shows that $g$ is an internal birational effective epimorphism. Since any effective epimorphism in $\mathcal{P}(S)$ or $\mathcal{P}_{nis}(S)$ is automatically an external birational effective epimorphism (as $L_{bir}$ preserves colimits), $L_{bir}$ preserves effective epimorphisms.
\end{proof}

\begin{rem}[Contrast with $\mathcal{H}^{\mathbb{A}^1}(S)$]
\Cref{bir eff epi} exhibits a sharp structural reversal relative to the $\mathbb{A}^1$-motivic theory of [\cite{inhot}, \S3.4.2]:
\begin{enumerate}
    \item Because $\mathcal{H}^{\mathbb{A}^1}(S)\subset \mathcal{P}_{nis}(S)$ is not closed under geometric realizations (\Cref{sifted colimits in hmot}; [\cite{inhot}, Remark~3.2.44]), effective epimorphisms in $\mathcal{H}^{\mathbb{A}^1}(S)$ need not be Nisnevich (let alone presheaf) effective epimorphisms, and $\pi_0^{mot}$ only preserves them ([\cite{inhot}, Proposition~3.4.19(1)]) without detecting them sectionwise on $\pi_0$.
    \item Conversely, because $L_{mot}$ is an effective localization, it detects Nisnevich effective epimorphisms over motivic targets, and $\mathbb{A}^1$-$0$-connectivity coincides with Nisnevich $0$-connectivity ([\cite{inhot}, Proposition~3.4.23]). By contrast, $L_{bir}$ is not effective (\Cref{Lbir not effective coex}): the dense open immersion $U:=X\setminus\{0\}\hookrightarrow X$ into an abelian variety $X\in \mathcal{H}^b(k)$ is an internal birational effective epimorphism (indeed, a birational equivalence) without being a Nisnevich effective epimorphism.
\end{enumerate}
\end{rem}

\begin{thm}[Higher Birational Motivic Van Kampen; compare \textup{[\cite{morel2012a1}, Theorem~7.12]} and \textup{[\cite{inhot}, Corollary~3.4.27]}]\label{bvk}
Let $S$ be qcqs and $n\ge 0$. Let $\{\esc{U}_i\to \esc{X}\}_{i\in I}$ be a compatibly pointed internal (identically external) birational geometric cover in $\mathcal{P}(S)$ (for example, an scdh covering $\{U_i\to X\}_{i\in I}$ in $\mathrm{Sm}_S$) such that every $m$-fold fiber product $\esc{U}_{i_1\dots i_m}:=\esc{U}_{i_1}\times_{\esc{X}}\dots\times_{\esc{X}}\esc{U}_{i_m}$ for $1\le m\le n+2$ is birationally $(n-1)$-connected. Then $\esc{X}$ is birationally $(n-1)$-connected, and there is a coequalizer diagram in $n\text{-}\mathrm{Grp}_S^b$:
\[
\coprod_{i,j}\,\pi_n^b(\esc{U}_i\times_{\esc{X}}\esc{U}_j)\;\rightrightarrows\; \coprod_{i}\,\pi_n^b(\esc{U}_i)\longrightarrow \pi_n^b(\esc{X})\longrightarrow *,
\]
where $\coprod$ denotes the coproduct in $n\text{-}\mathrm{Grp}_S^b$ (namely $\bigvee$ in $\mathrm{Shv}_{*,S}^b$ for $n=0$, $\ast$ in $\mathrm{Grp}_S^b$ for $n=1$, and $\bigoplus$ in $\mathrm{Ab}_S^b$ for $n\ge 2$).
\end{thm}
\begin{proof}
By \Cref{birational truncation}, \Cref{bir em}, and \Cref{bir eff epi}(b), $L_{bir}$ is path-injective and $j$-fundamental on every object for all $j\ge 0$, satisfies $n\text{-}\mathrm{Grp}^{L_{bir}} = n\text{-}\mathrm{Grp}_S^b$s. The result therefore follows directly from [\cite{inhot}, Theorem~2.4.39].
\end{proof}

\begin{cor}[Local-to-global birational $n$-connectivity; compare \textup{[\cite{inhot}, Lemma~3.4.26 and Corollary~3.4.30]}]\label[cor]{local to global bir conn}
Let $S$ be qcqs and $\esc{X}\in \mathcal{P}(S)$.
\begin{enumerate}
    \item ($n=0$) $\esc{X}$ (in particular, a scheme $X/S$) is birationally connected if and only if it admits a nonempty external birational geometric covering $\{\esc{U}_i\to \esc{X}\}_{i\in I}$ (e.g., a Nisnevich, $\mathbb{A}^1$-, or scdh covering) such that each $\esc{U}_i$ is birationally connected and, for some $i_0\in I$ and every $i\in I$, $\pi_0^b(\esc{U}_{i_0}\times_{\esc{X}}\esc{U}_i)$ has an $S$-rational point.
    \item ($n\ge 1$) A pointed $S$-space $\esc{X}$ is birationally $n$-connected if and only if it admits a compatibly pointed external birational geometric covering $\{\esc{U}_i\to \esc{X}\}_{i\in I}$ (e.g., an scdh covering) such that each $\esc{U}_i$ is birationally $n$-connected and every $m$-fold fiber product $\esc{U}_{i_1\dots i_m}$ ($2\le m\le n+2$) is birationally $(n-1)$-connected.
\end{enumerate}
\end{cor}
\begin{proof}
The "only if" directions are witnessed by $\{\mathrm{id}_{\esc{X}}\colon \esc{X}\to \esc{X}\}$. For the converse in (1), any external birational geometric covering is internal by \Cref{bir eff epi}(b). Moreover, since $S\in \mathrm{Sm}_S$ is terminal, the existence of an $S$-rational point of $\pi_0^b(\esc{U}_{i_0}\times_{\esc{X}}\esc{U}_i)$ is equivalent to sectionwise surjectivity of $\pi_0^b(\esc{U}_{i_0}\times_{\esc{X}}\esc{U}_i)\to *$, hence to $L_{bir}(\esc{U}_{i_0}\times_{\esc{X}}\esc{U}_i)\to *$ being an effective epimorphism in $\mathcal{H}^b(S)$ and $\mathcal{P}(S)$ by \Cref{bir eff epi}(a). Part (1) therefore follows from [\cite{inhot}, Lemma~2.4.38] and \Cref{birational truncation}. For (2), \Cref{bvk} implies that $\esc{X}$ is birationally $(n-1)$-connected and yields an epimorphism $*\simeq \coprod_{i\in I}\pi_n^b(\esc{U}_i)\to \pi_n^b(\esc{X})$ in $n\text{-}\mathrm{Grp}_S^b$, forcing $\pi_n^b(\esc{X})\simeq *$ (compare [\cite{inhot}, Corollary~2.4.41]).
\end{proof}

\begin{rem}
Unlike their $\mathbb{A}^1$-motivic counterparts [\cite{inhot}, Corollaries~3.4.27 and~3.4.30], \Cref{bvk} and \Cref{local to global bir conn} hold over an arbitrary qcqs base scheme $S$ (not just a perfect field), apply to scdh coverings, require no replacement of $\pi_0^b$ by a distinct internal functor at $n=0$. Over a perfect field $k$, when all connectivity hypotheses in \Cref{bvk} are strengthened to $\mathbb{A}^1$-$(n-1)$-connectivity, \Cref{bvk} can also be deduced directly from the higher $\mathbb{A}^1$-Van Kampen theorem [\cite{inhot}, Corollary~3.4.27] by applying the colimit-preserving localization $\pi_0^b\colon n\text{-}\mathrm{Grp}_k^{L_{mot}}\to n\text{-}\mathrm{Grp}_k^b$ via \Cref{pi1b=pi0bpi1a1} (and \Cref{pi0b is birationalization} for $n=0$).
\end{rem}
\subsection{Birational homotopy groups}\label{3.6}
Starting with \Cref{birational truncation}(1), we have, on multiple occasions, specifically at the beginning of the previous section \Cref{3.5}, explained that, birational locally, internal and external notions of homotopy groups coincide. Moreover, in the last section, we have established various methods for computing birational fundamental groups from known ones, along with birational connectivity. In this section, we shall derive direct algebraic methods of computing birational homotopy groups. We start with a universal characterization:
\begin{thm}\label{higher pii as universal birationalization}
Let $\esc{X}\in \mathcal{P}_{nis}(S)_\bullet$ be a pointed Nisnevich-locally $0$-connected $S$-space such that for some $n\geq 0$ the truncation $\tau^{nis}_{\leq n}\esc{X}$ is birational-local (equivalently, $\pi_i^{nis}\esc{X}$ is birational for all $i\leq n$). Then the canonical morphism $\pi_{n+1}^{nis}\esc{X}\to \pi_{n+1}^{b}\esc{X}$ is universal among morphisms to a birational sheaf of groups when $n=0$, and to a birational sheaf of abelian groups when $n\geq 1$.
\end{thm}
\begin{proof}
Let $B\in (n+1)\text{-}\mathrm{Grp}_S^b$ be a birational sheaf of groups ($n=0$) or abelian groups ($n\geq 1$). Since $\tau^{nis}_{>n}\esc{X}$ is Nisnevich-locally $n$-connected, it admits an $(n+1)$-fold delooping $\tau^{nis}_{>n}\esc{X}\simeq \mathrm{B}_{nis}^{n+1}\Omega^{n+1}\tau^{nis}_{>n}\esc{X}$. Because $B$ is discrete, we have
\begin{align}\label{univer}
\mathrm{Map}_{\mathcal{P}_{nis}(S)}(\tau^{nis}_{> n}\esc{X}, \mathrm{B}^{n+1}_{nis} B)
&\simeq \mathrm{Map}_{\mathcal{E}_{n+1}\text{-}\mathrm{Mon}(\mathcal{P}_{nis}(S))}(\Omega^{n+1}\tau^{nis}_{>n}\esc{X}, B) \\
&\simeq \mathrm{Hom}_{(n+1)\text{-}\mathrm{Grp}_S}(\pi_0^{nis}\Omega^{n+1}\tau^{nis}_{>n}\esc{X}, B) \nonumber \\
&\cong \mathrm{Hom}_{(n+1)\text{-}\mathrm{Grp}_S}(\pi^{nis}_{n+1}\esc{X}, B).\label{uniend}
\end{align}
On the other hand, since $\mathrm{B}^{n+1}_{nis}B$ is birational-local (\Cref{B_bir is B}), \Cref{Lbir vs connective covers} identifies the left-hand side of~\eqref{univer} with
\[
\mathrm{Map}_{\mathcal{H}^b(S)}(L_{bir}\tau^{nis}_{> n}\esc{X}, \mathrm{B}_{nis}^{n+1}B) \simeq \mathrm{Map}_{\mathcal{P}_{nis}(S)}(\tau^{nis}_{> n}L_{bir}\esc{X}, \mathrm{B}^{n+1}_{nis} B).
\]
Repeating the identifications~\eqref{univer}-\eqref{uniend} for $L_{bir}\esc{X}$ equates this mapping space with
\begin{align}
\mathrm{Hom}_{(n+1)\text{-}\mathrm{Grp}_S}(\pi^{nis}_{n+1}L_{bir}\esc{X}, B) \cong \mathrm{Hom}_{(n+1)\text{-}\mathrm{Grp}^b_S}(\pi^{b}_{n+1}\esc{X}, B). 
\end{align}
Using (3) and (4), the claim follows from the Yoneda lemma.
\end{proof}

For $\pi_0$, no connectivity assumption is required:

\begin{prop}\label[prop]{pi0b is birationalization}
For any $\esc{X}\in \mathcal{P}(S)$, the natural maps $\pi_0\esc{X}\to \pi_0^b\esc{X}$, $\pi_0^{nis}\esc{X}\to \pi_0^b\esc{X}$, and $\pi_0^{\mathbb{A}^1}\esc{X}\to \pi_0^b\esc{X}$ are universal among morphisms to a birational sheaf of sets.
\end{prop}
\begin{proof}
For any discrete birational sheaf $B\in \mathrm{Shv}^b_S\subset \mathcal{H}^b(S)$, we have natural identifications
\[
\mathrm{Hom}_{\mathrm{PSh}_S}(\pi_0\esc{X}, B) \simeq \mathrm{Map}_{\mathcal{P}(S)}(\esc{X}, B) \simeq \mathrm{Map}_{\mathcal{H}^b(S)}(L_{bir}\esc{X}, B) \simeq \mathrm{Hom}_{\mathrm{Shv}^b_S}(\pi_0^b\esc{X}, B).
\]
Since every birational presheaf is an $\mathbb{A}^1$-invariant Nisnevich sheaf (\Cref{Motivic equivalences are Birational equivalences}(2)), the universal properties for $\pi_0^{nis}\esc{X}$ and $\pi_0^{\mathbb{A}^1}\esc{X}$ follow immediately.
\end{proof}

\begin{rem}\label[rem]{pi0b is birationalization of pi0}
For a discrete presheaf of sets $F\in \mathrm{PSh}_S$, \Cref{pi0b is birationalization} shows that $F\mapsto \pi_0^bF$ provides a left adjoint to the inclusion $\mathrm{Shv}^b_S\hookrightarrow \mathrm{PSh}_S$ (studied over fields in [\cite{Shimizu01022023}] under the notation $\pi_0^{br}$), yielding canonical isomorphisms
\[
\pi_0^b\esc{X}\simeq \pi_0^{b}\pi_0\esc{X}\simeq \pi_0^{b}\pi^{nis}_0\esc{X} \simeq\pi_0^{b}\pi^{\mathbb{A}^1}_0\esc{X}.
\]
Moreover, since both $\pi_0$ and $L_{bir}$ preserve finite products (\Cref{Lbir is cartesian}), $\pi_0^b=\pi_0L_{bir}$ preserves finite products and hence group and monoid structures. Thus \Cref{higher pii as universal birationalization} can be restated as $\pi_{n+1}^{b}\esc{X}\cong\pi_0^{b}(\pi_{n+1}^{nis}\esc{X})$.
\end{rem}

\begin{cor}\label[cor]{effective on cmon}
Let $\esc{M}$ be a commutative $H$-space in $\mathcal{P}(S)$. Then $\pi_0\esc{M}\to \pi_0^b\esc{M}$ is a surjection of presheaves.
\end{cor}
\begin{proof}
By \Cref{pi0b is birationalization} and \Cref{pi0b is birationalization of pi0}, $\pi_0\esc{M}\to \pi_0^b\esc{M}$ is the unit of the localization $\mathrm{CMon}(\mathrm{PSh}_S)\to \mathrm{CMon}(\mathrm{Shv}^b_S)$, hence an epimorphism in $\mathrm{CMon}(\mathrm{Shv}^b_S)$. By the same semiadditivity argument as in \Cref{bir is exact in presheaves}, the inclusion $\mathrm{CMon}(\mathrm{Shv}^b_S)\subset \mathrm{CMon}(\mathrm{PSh}_S)$ preserves colimits and hence epimorphisms, so $\pi_0\esc{M}\to \pi_0^b\esc{M}$ is an epimorphism in $\mathrm{PSh}_S$.
\end{proof}

\begin{rem}
\Cref{effective on cmon} is a birational version of the $\mathbb{A}^1$-motivic result [\cite{morel19991}, Corollary~2.3.22], which holds for all $S$-spaces; in the birational setting it fails without the $H$-space structure by \Cref{Lbir not effective coex}.
\end{rem}

\begin{cor}\label[cor]{pi1b=pi0bpi1a1}
Suppose $\esc{X}$ is a (motivically) pointed $\mathbb{A}^1$-$n$-connected $S$-space ($n\ge -1$). Then the canonical map $\pi_{n+1}^{\mathbb{A}^1}\esc{X}\to \pi_{n+1}^{b}\esc{X}$ is universal among maps to a birational sheaf of sets when $n=-1$; to a birational sheaf of groups when $n=0$; and to a birational sheaf of abelian groups when $n\geq 1$; i.e.,
\[
\pi_{n+1}^{b}\esc{X}\simeq \pi_0^b(\pi_{n+1}^{\mathbb{A}^1}\esc{X}).
\]
\end{cor}
\begin{proof}
The case $n=-1$ is \Cref{pi0b is birationalization}. For $n\geq 0$, apply \Cref{higher pii as universal birationalization} to $\esc{Y}:=L_{mot}\esc{X}$, which satisfies $\tau_{\leq n}^{nis}\esc{Y}\simeq *\in \mathcal{H}^b(S)$, and use $L_{bir}L_{mot}\simeq L_{bir}$ together with \Cref{pi0b is birationalization of pi0}.
\end{proof}

\begin{rem}
When $S=\operatorname{Spec} k$ for a perfect field $k$ and $n\geq 0$, \Cref{gm contraction 0} allows us to restate \Cref{pi1b=pi0bpi1a1} as saying that $\pi_{n+1}^{\mathbb{A}^1}\esc{X}\to \pi_{n+1}^{b}\esc{X}$ is universal among maps to a strongly (resp.\ strictly, for $n\ge 1$) $\mathbb{A}^1$-invariant Nisnevich sheaf of groups (resp.\ abelian groups) with trivial $\mathbb{G}_m$-contraction.
\end{rem}

\medskip
\noindent{\small{\textbf{Computing the birational path component $\pi_0^b$:}}}\label{cbpc} Therefore, the transition from motivic homotopy (groups) to birational motivic homotopy (groups) is, to some extent, controlled by $\pi_0^b$ of discrete $S$-spaces. We want to observe that in sufficiently nice scenarios, this behaves quite well. In fact, it turns out that computing the birational path component of any presheaf of spaces is easier than computing its $\mathbb{A}^1$ counterpart.
\begin{prop}\label[prop]{pi0b preserves sifted colimit}
The functor $\pi^b_0\colon \mathcal{P}(S)\to \mathrm{PSh}(S)$ preserves sifted colimits. When $S$ is geometrically unibranch and has finitely many generic points, then $\pi^b_0\colon \mathcal{P}_\Sigma(S)\to \mathrm{PSh}_\Sigma(S)$ and $\pi^b_0\colon \mathcal{P}_{nis}(S)\to \mathrm{PSh}_{nis}(S)$ preserve all colimits.
\end{prop}
\begin{proof}
Since $\pi_0\colon \mathcal{P}(S)\to \mathrm{PSh}_S$ preserves all colimits, this follows immediately from the fact that $L_{bir}\colon \mathcal{P}(S)\to \mathcal{P}(S)$ preserves sifted colimits (\Cref{l bir preserve sifted colimit}). The second statement follows from the analogous statement about $\pi_0\colon \mathcal{P}_\Sigma(S)\to \mathrm{PSh}_\Sigma(S)$ and $\pi_0^{nis}\colon \mathcal{P}_{nis}(S)\to \mathrm{PSh}_{nis}(S)$ in presence of \Cref{closed under coproducts}.
\end{proof}
Since every $\esc{X}\in \mathcal{P}_\Sigma(S)$ can be written as a sifted colimit of smooth schemes, we have
\begin{flalign}\label{pi0b formula}
\pi_0^b\esc{X}=\pi_0L_{bir}\esc{X}\simeq \colim_{\substack{X \to \mathcal{X} \\ X/S \in \mathrm{Sm}_S}}\pi_0^bX    
\end{flalign}
where the colimit is computed in the category of presheaves of sets on $\mathrm{Sm}_S$. Thus, to compute $\pi_0^b\esc{X} $ for any $\esc{X}$, it suffices to know $\pi_0^bX$ for $X\in \mathrm{Sm}_S$. Moreover, when the base is a field $k$ of characteristic zero, by \Cref{char zero smooth to smoothproper} it suffices to know $\pi_0^bX$ for every smooth proper scheme $X/k$. It turns out that there is a classical way to compute this.

To see this, let us fix an arbitrary base field $k$. Given a proper scheme $X/k$, there is a birational sheaf, denoted $\pi_0^{b\mathbb{A}^1}X$, whose sections over a finitely generated separable field extension $L/k$ are given by the set of Manin's $R$-equivalence classes of $L$-points of $X$ [\cite{asok2011smooth}, \S6]. To construct it, the authors of \textit{loc.\ cit.} introduce the category $\mathrm{Shv}^{b\mathbb{A}^1}_k$ of $\mathbb{A}^1$-invariant birational sheaves and prove that evaluation at function fields defines a fully faithful functor $\mathrm{Shv}^{b\mathbb{A}^1}_k\to \mathcal{F}^r_k\text{-}\mathrm{Sets}$ with essential image consisting of sheaf-like transcendental-invariant data [\cite{asok2011smooth}, Theorem~6.1.7]. In view of \Cref{Motivic equivalences are Birational equivalences}, the $\mathbb{A}^1$-invariance assumption is redundant:

\begin{lem}[\textup{[\cite{asok2011smooth}, Theorem~6.1.7]}]
For a field $k$, the inclusion $\mathrm{Shv}^{b\mathbb{A}^1}_k\subset \mathrm{Shv}^{b}_k$ is an equality, and the functor $\mathrm{Shv}^{b}_k\to \mathcal{F}^r_k\text{-}\mathrm{Sets}$ given by evaluation at the generic point is fully faithful, with essential image consisting of sheaf-like transcendental-invariant data \textup{(}see \textup{[\cite{asok2011smooth}, \S6]} whee it is origianlly termed as $\mathbb{A}^1$-invariant data\textup{)}.
\end{lem}

For any proper scheme $X/k$, the assignment $L\mapsto X(L)/\!\sim$ of naive $\mathbb{A}^1$-homotopy classes of $L$-points of $X$ is a sheaf-like transcendental-invariant $\mathcal{F}^r_k$-datum defining $\pi_0^{b\mathbb{A}^1}X$ [\cite{asok2011smooth}, Theorem~6.2.1]. The identification $\pi_0^bX\simeq \pi_0^{b\mathbb{A}^1}X$ was previously obtained for \textbf{smooth proper varieties} in [\cite{choudhury2022characterisation}, Theorem~3.4] and [\cite{cisinski2023homotopy}, Proposition~1.9] using Kahn-Sujatha's localized category $S_b^{-1}\mathrm{Sm}_k$. The proof below is formal, avoids $S_b^{-1}\mathrm{Sm}_k$, and works for all proper schemes:

\begin{thm}[Generalizing \textup{[\cite{choudhury2022characterisation}, Theorem~3.4]} and \textup{[\cite{cisinski2023homotopy}, Proposition~1.9]} to all proper schemes]\label{pi0bX=pi0ba1X}
Let $k$ be a field and $X$ a proper scheme over $k$. Then there is a canonical isomorphism $\pi^{b}_0X\simeq \pi_0^{b\mathbb{A}^1}X$ of sheaves. Consequently, the assignment $\pi_0^{b\mathbb{A}^1}(-)\colon \mathrm{SmProp}_k\to \mathrm{Sh}_k$ is a stable birational invariant, i.e., if $X$ and $Y$ are stably birational smooth proper schemes, then there is an isomorphism of sheaves $\pi^{b\mathbb{A}^1}_0X\simeq \pi^{b\mathbb{A}^1}_0Y$.
\end{thm}
\begin{proof}
By \Cref{pi0b is birationalization}, $X\cong\pi_0X\to \pi_0^bX$ provides the universal birationalization of $X$ in $\mathrm{Shv}^b_k$. On the other hand, by [\cite{asok2011smooth}, Remark~6.2.2] (which can be obtained easily from Theorem 6.1.7 in loc.cit.), for any proper scheme $X$, the canonical morphism $X \to \pi_0^{b\mathbb{A}^1}X$ is the universal object of $\mathrm{Shv}^{b\mathbb{A}^1}_k$ associated to $X$. Because $\mathrm{Shv}^{b\mathbb{A}^1}_k=\mathrm{Shv}^b_k$, comparing universal properties yields the canonical isomorphism $\pi^{b}_0X\xrightarrow{\sim} \pi_0^{b\mathbb{A}^1}X$.

For the last statement it suffices to know that the assignment $\pi_0^b(-)|_{\mathrm{Sm}_k}\colon \mathrm{Sm}_k\to \mathrm{Sh}_k$, $X\mapsto \pi^b_0(X)=\pi_0L_{bir}X$, is, by construction, a stable birational invariant.
\end{proof}

\begin{cor}\label[cor]{uni bir of pia1x}
For $X\in \mathrm{Prop}_k$, the canonical map
$ \pi_0^{\mathbb{A}^1}X\longrightarrow \pi_0^{b\mathbb{A}^1}X$ 
\textup{[\cite{asok2011smooth}, Proposition~6.2.6]} is initial among maps to a birational sheaf of sets.
\end{cor}
\begin{proof}
Combine \Cref{pi0b is birationalization} and \Cref{pi0bX=pi0ba1X}.
\end{proof}

\underline{\textit{$\pi_0$ and $\mathbb{H}_0$ of proper motivic spaces $\mathcal{H}_{prop}^{\mathbb{A}^1}(k)$: motivic vs birational:}} Over a field $k$ of characteristic $0$, the $\infty$-topos $\mathcal{H}^b(k)$ (\Cref{Hbir a topos}) is generated under sifted colimits by smooth proper schemes (\Cref{generation of H^b}(c)), on which $\pi_0^b$ and $\pi_0^{b\mathbb{A}^1}$ coincide (\Cref{pi0bX=pi0ba1X}). More generally, over an arbitrary field $k$, let $\mathcal{H}_{\mathrm{prop}}^{\mathbb{A}^1}(k) \subset \mathcal{H}^{\mathbb{A}^1}(k)$ denote the full subcategory generated under small colimits by proper $k$-schemes, whose objects we call \emph{proper motivic spaces} (and write $\mathcal{H}_{\mathrm{smprop}}^{\mathbb{A}^1}(k) \subset \mathcal{H}_{\mathrm{prop}}^{\mathbb{A}^1}(k)$ for the subcategory generated by smooth proper $k$-schemes; recall that the distinction between $\mathcal{H}_{\mathrm{smprop}}^{\mathbb{A}^1}(k)$, $\mathcal{H}_{\mathrm{prop}}^{\mathbb{A}^1}(k)$, and $\mathcal{H}^{\mathbb{A}^1}(k)$ vanishes $\mathbb{P}^1$-stably [\cite[Lemme~2.2.23 and Proposition~2.2.27]{MR2423375}]). Below, we use \Cref{pi0bX=pi0ba1X} to show that on $\mathcal{H}_{\mathrm{prop}}^{\mathbb{A}^1}(k)$ (and hence on $\mathcal{H}_{\mathrm{smprop}}^{\mathbb{A}^1}(k)$), $L_{\mathrm{bir}}$ detects connectivity and Morel's $0$-th $\mathbb{A}^1$-homology sheaf $\mathbb{H}_0^{\mathbb{A}^1}$ [\cite[Definition~6.29]{morel2012a1}] is automatically birational.

\begin{thm}\label{Hprop connectivity detection}
Let $k$ be a field and $\mathcal{H}_{prop}^{\mathbb{A}^1}(k)$ be the full subcategory of $\mathcal{H}^{\mathbb{A}^1}(k)$ generated under small colimits by proper schemes. Then for every $\esc{X}\in \mathcal{H}_{prop}^{\mathbb{A}^1}(k)$, the canonical map $\pi_0^{\mathbb{A}^1}\esc{X}\to \pi_0^b\esc{X}$ is a bijection over $\mathcal{F}_k$.
\end{thm}
\begin{proof}
Let $D:\mathcal{J}\to \mathcal{H}^{\mathbb{A}^1}(k)$ be a diagram defining $\esc{X}\in \mathcal{H}^{\mathbb{A}^1}_{prop}(k)$ as the colimit $\esc{X}=\colim_{j\in\mathcal{J}}Dj$, where each $Dj=L_{mot}X_j\in L_{mot}\mathrm{Prop}_k$. Let $\iota$ denote the inclusion $\mathcal{H}^{\mathbb{A}^1}(k)\subset \mathcal{P}_{nis}(k)$. Let $\esc{Y}$ be the colimit computed in $\mathcal{P}_{nis}(k)$, i.e., $\esc{Y}=\colim_{j\in \mathcal{J}}\iota Dj$, so that $\esc{X}\simeq L_{mot}\esc{Y}$. Consider the unit of the birational motivic localization $\eta:\esc{Y}\to L_{bir}\esc{Y}\simeq L_{bir} \esc{X}$, and apply the epimorphism $\pi_0^{nis}\to \pi_0^{\mathbb{A}^1}$ [\cite{morel19991}, Corollary~2.3.22] to obtain the following commutative square:
\[
\xymatrix{
\pi_0^{nis}\esc{Y}\ar@{->>}[r]\ar[d]&\pi_0^{\mathbb{A}^1}\esc{Y}\ar[d]\ar@{=}[r]&\pi_0^{\mathbb{A}^1}\esc{X}\\
\pi_0^{nis}L_{bir}\esc{X}\ar[r]^{\simeq}&\pi_0^{\mathbb{A}^1}L_{bir}\esc{X}\ar@{=}[r]&\pi_0^b\esc{X}
}
\]
Since $\pi_0^{nis}$ commutes with Nisnevich colimits, \Cref{pi0b preserves sifted colimit} identifies the left vertical map with the colimit
$\colim_{j\in \mathcal{J}}\big( \pi_0^{nis}\iota Dj\longrightarrow \pi_0^{b}Dj\big)=\colim_{j\in \mathcal{J}}\big( \pi_0^{\mathbb{A}^1}X_j\longrightarrow \pi_0^{b}X_j\big)$.
Hence, by \Cref{pi0bX=pi0ba1X} and [\cite{asok2011smooth}, Proposition~6.2.6], this map is an isomorphism over $\mathcal{F}_k$. Since the top horizontal map is a Nisnevich epimorphism, the commutative square implies that the right vertical map is an isomorphism over $\mathcal{F}_k$.
\end{proof}

\begin{cor}[\textbf{$\mathbb{A}^1$-connectivity vs.\ birational connectivity on $\mathcal{H}_{prop}^{\mathbb{A}^1}(k)$}]\label[cor]{proper schemes are a1 connected iff bir connected}
Let $k$ be a perfect field. Then $L_{bir}:\mathcal{H}^{\mathbb{A}^1}_{prop}(k)\to \mathcal{H}^b(k)$ preserves and detects connectivity. Consequently, a proper motivic space, and in particular every (ind-)proper scheme over $k$, is $\mathbb{A}^1$-connected if and only if it is birationally connected.
\end{cor}
\begin{proof}
Preservation of connectivity follows from \Cref{bir preserves a1 conn}(2).

Detection follows from the above theorem, using the fact that over perfect fields, triviality of $\pi^{\mathbb{A}^1}_0\esc{X}$ is detected over $\mathcal{F}_k$ [\cite{bachmann2024strongly}, Corollary~2.10].
\end{proof}

\begin{rem}
\Cref{Hprop connectivity detection} and \Cref{proper schemes are a1 connected iff bir connected} fail on non-proper spaces: $\mathbb{G}_m$ is $\mathbb{A}^1$-disconnected despite being birationally contractible. They also fail for higher connectivity: $\mathbb{P}^2_k$ is birationally contractible over $k$, but has nontrivial $\mathbb{A}^1$-fundamental group $\pi_1^{\mathbb{A}^1}(\mathbb{P}^2_k)\cong\mathbb{G}_m$ [\cite{morel2012a1}, Theorem~7.13]. Therefore, $L_{bir}:\mathcal{H}^{\mathbb{A}^1}_{prop}(k)\to \mathcal{H}^b(k)$ is a nontrivial localization.\end{rem}

\begin{thm}[\textbf{Birationality of $\mathbb{H}_0^{\mathbb{A}^1}$ on $\mathcal{H}_{prop}^{\mathbb{A}^1}(k)$}; generalizing \textup{[\cite{koizumi2021zeroth}, Theorem~4.5]}]\label[cor]{Ha1 is bir loc for prop}
Let $k$ be a field. The restriction of the $0$-th $\mathbb{A}^1$-homology functor
\[
\mathbb{H}_0^{\mathbb{A}^1}(-)\colon \mathcal{H}_{prop}^{\mathbb{A}^1}(k)\longrightarrow \mathrm{Ab}_k^{\mathbb{A}^1}
\]
factors through $\mathrm{Ab}_k^b$ via a canonical isomorphism $\theta_{\esc{X}}\colon \mathbb{H}_0^{\mathbb{A}^1}(\esc{X})\xrightarrow{\;\sim\;} \mathbb{Z}_{\Sigma}(\pi_0^b\esc{X})$, and is itself a birational motivic invariant \textup{(}and in particular, on $\mathrm{SmProp}_k$, a stable birational invariant\textup{)}.
\end{thm}
\begin{proof}
Recall that $\mathbb{H}_0^{\mathbb{A}^1}\colon \mathcal{H}^{\mathbb{A}^1}(k)\to \mathrm{Ab}_k^{\mathbb{A}^1}$ is left adjoint to the inclusion $\mathrm{Ab}_k^{\mathbb{A}^1}\hookrightarrow \mathcal{H}^{\mathbb{A}^1}(k)$. Since every object of $\mathrm{Ab}_k^{\mathbb{A}^1}$ is discrete, $\mathbb{H}_0^{\mathbb{A}^1}(\esc{X})\simeq \mathbb{H}_0^{\mathbb{A}^1}(\pi_0^{\mathbb{A}^1}\esc{X})$. Moreover, since $\pi_0^b\esc{X}\in \mathrm{Shv}_k^b$, the free abelian Nisnevich sheaf $\mathbb{Z}_\Sigma(\pi_0^b\esc{X})$ lies in $\mathrm{Ab}_k^b$ and hence in $\mathrm{Ab}_k^{\mathbb{A}^1}$ (\Cref{bir is strong a1}), so $\mathbb{H}_0^{\mathbb{A}^1}(\pi_0^b\esc{X})\simeq \mathbb{Z}_\Sigma(\pi_0^b\esc{X})$. Applying $\mathbb{H}_0^{\mathbb{A}^1}$ to the birationalization map $q_{\esc{X}}\colon \pi_0^{\mathbb{A}^1}\esc{X}\to \pi_0^b\esc{X}$ therefore yields a factorization of $\mathbb{Z}_\Sigma(q_{\esc{X}})$ in $\mathrm{Ab}_k$:
\[
\mathbb{Z}_\Sigma(\pi_0^{\mathbb{A}^1}\esc{X}) \xrightarrow{\;\eta_{\esc{X}}\;} \mathbb{H}_0^{\mathbb{A}^1}(\pi_0^{\mathbb{A}^1}\esc{X}) = \mathbb{H}_0^{\mathbb{A}^1}(\esc{X}) \xrightarrow{\;\theta_{\esc{X}}\;} \mathbb{Z}_\Sigma(\pi_0^b\esc{X}).
\]
 Because $\mathbb{H}_0^{\mathbb{A}^1}$, $\mathbb{Z}_\Sigma(\pi_0^b(-))$, and the inclusion $\mathrm{Ab}_k^b\hookrightarrow \mathrm{Ab}_k^{\mathbb{A}^1}$ (\Cref{bir is exact in presheaves}) all preserve small colimits, it suffices to show that $\theta_X$ is an isomorphism for every $X\in \mathrm{Prop}_k$.

For $X\in \mathrm{Prop}_k$ and any $L\in \mathcal{F}_k$, $q_X(L)\colon \pi_0^{\mathbb{A}^1}X(L)\to \pi_0^b X(L)$ is a bijection (\Cref{Hprop connectivity detection}), so $\theta_X(L)\circ \eta_X(L) = \mathbb{Z}[q_X(L)]$ is an isomorphism. (Note that the epicity of $\eta_{\esc{X}}$ is unclear, so a similar argument as in \Cref{Hprop connectivity detection} cannot be applied immediately, and we have to construct an inverse by hand). For any connected $U\in \mathrm{Sm}_k$ with function field $L=k(U)$ and any $x\in U^{(1)}$, the valuative criterion of properness gives $X(\mathcal{O}_{U,x})\xrightarrow{\sim} X(L)$, so $\pi_0^{\mathbb{A}^1}X(\mathcal{O}_{U,x})\twoheadrightarrow \pi_0^{\mathbb{A}^1}X(L)$ is surjective [\cite{morel19991}, Corollary 2.3.22]. By naturality of $\eta_X$ on the essentially smooth local scheme $\operatorname{Spec}\mathcal{O}_{U,x}$, the image of $$\psi_L := \eta_X(L)\circ \mathbb{Z}[q_X(L)]^{-1}\colon \mathbb{Z}(\pi_0^b X)(L)\to \mathbb{H}_0^{\mathbb{A}^1}(X)(L)$$ equals
\[
\eta_X(L)\big(\mathbb{Z}[\pi_0^{\mathbb{A}^1}X(L)]\big) = \mathrm{im}\Big(\mathbb{Z}[\pi_0^{\mathbb{A}^1}X(\mathcal{O}_{U,x})] \xrightarrow{\eta_X(\mathcal{O}_{U,x})} \mathbb{H}_0^{\mathbb{A}^1}(X)(\mathcal{O}_{U,x})\Big) \subset \mathbb{H}_0^{\mathbb{A}^1}(X)(\mathcal{O}_{U,x}).
\]
Since $\mathbb{H}_0^{\mathbb{A}^1}(X)\in \mathrm{Ab}_k^{\mathbb{A}^1}$ is unramified [\cite{morel2005stable}, Lemma 6.4.11], we have $\mathbb{H}_0^{\mathbb{A}^1}(X)(U) = \mcap_{x\in U^{(1)}} \mathbb{H}_0^{\mathbb{A}^1}(X)(\mathcal{O}_{U,x})$ inside $\mathbb{H}_0^{\mathbb{A}^1}(X)(L)$. Because $\mathbb{Z}(\pi_0^b X)(U)\xrightarrow{\sim}\mathbb{Z}(\pi_0^b X)(L)$, the map $\psi_L$ restrict to a morphism $\psi_X(U)\colon \mathbb{Z}(\pi_0^b X)(U)\to \mathbb{H}_0^{\mathbb{A}^1}(X)(U)$. This gives us a sheaf morphism $\psi_X\colon \mathbb{Z}_\Sigma(\pi_0^b X)\to \mathbb{H}_0^{\mathbb{A}^1}(X)$. Clearly,  by construction $\theta_X\psi_X=id$, whereas the identity $\psi_X\theta_X=id$ follows from the universal property of $\mathbb{H}_0^{\mathbb{A}^1}$.
\end{proof}
\begin{cor}[\textbf{Universal property of $\mathbb{H}_0^{\mathbb{A}^1}$ on $\mathcal{H}_{prop}^{\mathbb{A}^1}(k)$}]\label[cor]{ha1 is universal}
For every $\esc{X}\in \mathcal{H}_{prop}^{\mathbb{A}^1}(k)$ (in particular, for every proper scheme $X/k$), the canonical maps $\esc{X}\to \mathbb{H}_0^{\mathbb{A}^1}(\esc{X})$ and $\pi_0^{\mathbb{A}^1}\esc{X}\to \mathbb{H}_0^{\mathbb{A}^1}(\esc{X})$ are initial among maps to a birational sheaf of abelian groups. In other words, $\mathbb{H}_0^b(\esc{X})\simeq \mathbb{H}_0^{\mathbb{A}^1}(\esc{X})$.
\end{cor}
\begin{proof}
This is immediate from the canonical isomorphism $\mathbb{H}_0^{\mathbb{A}^1}(\esc{X})\simeq \mathbb{Z}_\Sigma(\pi_0^{b}\esc{X})$ from \Cref{Ha1 is bir loc for prop} and the universal property of $\pi_0^b$ from \Cref{pi0b is birationalization}.
\end{proof}
\begin{cor}[K\"unneth isomorphism of $\mathbb{H}_0^{\mathbb{A}^1}$ for proper motivic spaces]
 The functor $\mathbb{H}^{\mathbb{A}^1}_0:\mathcal{H}_{prop}^{\mathbb{A}^1}(k)^{\times }\to Ab_k^{\otimes}$ is symmetric monoidal.
\end{cor}
\begin{cor}\label[cor]{Ha1 is bir loc for prop scheme}
  The assignment $\mathbb{H}_0^{\mathbb{A}^1}(-): \mathrm{SmProp}_k\to Ab_k^{\mathbb{A}^1}$ factors through $Ab_k^b$ and is itself a stable birational invariant. 
\end{cor}
\begin{proof}
 From \Cref{Ha1 is bir loc for prop}, $\mathbb{H}_0^{\mathbb{A}^1}(X)\simeq \mathbb{Z}(\pi_0^{b}X)$ for any proper $X$. The claim follows from the birational invariance of $X\mapsto \pi_0^bX$ on $\mathrm{Sm}_k$.
\end{proof}
\begin{cor}
The Albanese map $\alpha_X \colon X \to \mathrm{Alb}(X)$ of a pointed proper variety $(X, x_0)$ over $k$, factors uniquely through a morphism $\widetilde{\mathbb{H}}_0^{\mathbb{A}^1}(X) \to \mathrm{Alb}(X)$ in $\mathrm{Ab}_k^b$, which exhibits $\mathrm{Alb}(X)$ as the universal abelian-variety quotient of $\widetilde{\mathbb{H}}_0^{\mathbb{A}^1}(X)$. If a proper variety has trivial reduced $\mathbb{A}^1$-homology, then its Albanese variety is trivial.
\end{cor}
\begin{proof}
  Using \Cref{ha1 is universal}, this follows immediately from the fact that $\mathrm{Alb}(X)$, being an abelian variety, is birational local (\Cref{abelian variety}).
\end{proof}
\medskip
\underline{\textit{Computing $\pi_0^b$ on $\mathrm{Grp}^{\mathbb{A}^1}_k$:}}
We have seen that for a pointed $\mathbb{A}^1$-$n$-connected space $\esc{X}$ ($n \ge 0$), computing the higher birational homotopy group sheaf $\pi_{n+1}^b\esc{X} \simeq \pi_0^b(\pi_{n+1}^{\mathbb{A}^1}\esc{X})$ (\Cref{pi1b=pi0bpi1a1}) reduces to computing the universal birationalization $\pi_0^b(G)$ of a strongly $\mathbb{A}^1$-invariant sheaf of groups $G \in \mathrm{Grp}_k^{\mathbb{A}^1}$. By \Cref{gm contraction 0}, when $k$ is a perfect field, $\mathrm{Grp}_k^b \subset \mathrm{Grp}_k^{\mathbb{A}^1}$ (resp.\ $\mathrm{Ab}_k^b \subset \mathrm{Ab}_k^{\mathbb{A}^1}$) is precisely the kernel of the contraction functor $(-)_{-1} = \Omega_{\mathbb{G}_m}$. Using the adjunction between $\mathbb{G}_m$-suspension and $\mathbb{G}_m$-loops, we shall now deduce that $\pi_0^b(G)$ is obtained by killing the image of $G_{-1}$ in a controlled way. For this, note that the contraction functor $(-)_{-1} \colon \mathrm{Grp}_k^{\mathbb{A}^1} \to \mathrm{Grp}_k^{\mathbb{A}^1}$ admits a left adjoint $(-)^{(1)} := \pi_1^{\mathbb{A}^1}(\mathbb{P} \wedge -)\simeq F_{\mathbb{A}^1}(\mathbb{G}\wedge -)$. Let $\varepsilon_G \colon (G_{-1})^{(1)}\to G$ denote the corresponding counit.

\begin{thm}\label{pi0b-presentation}
Let $k$ be a perfect field.
\begin{enumerate}
    \item Define the endofunctor $Q_b \colon \mathrm{Grp}_k^{\mathbb{A}^1} \to \mathrm{Grp}_k^{\mathbb{A}^1}$ by the right exact sequence in $\mathrm{Grp}_k^{\mathbb{A}^1}$:
    \[
     (G_{-1})^{(1)} \xrightarrow{\varepsilon_G} G \xrightarrow{\;q_G\;} Q_b(G)\xrightarrow{}1.
    \]
    Then $\pi_0^b(G) \cong \varinjlim_{n} Q_b^n(G)$ in $\mathrm{Grp}_k^{\mathbb{A}^1}$. (If the Nisnevich cokernel of $\varepsilon_G$ is already in $\mathrm{Grp}_k^{\mathbb{A}^1}$, then a single step suffices: $\pi_0^b(G) \cong Q_b(G) \cong G/\langle\operatorname{im}(\varepsilon_G)\rangle_{{nis}}$).
    
    \item When $M \in \mathrm{Ab}_k^{\mathbb{A}^1}$, there is a canonical right exact sequence in $\mathrm{Ab}_k^{\mathbb{A}^1}$ (and in $\mathrm{Ab}_k$):
    \[
    M_{-1} \otimes^{\mathbb{A}^1} \underline{K}_1^{MW} \xrightarrow{\;\mu_M\;} M \longrightarrow \pi_0^b(M) \longrightarrow 0
    \]
    Where $\mu_M$ is the ($n=1$) bilinear pairnig of \textup{[\cite{morel2012a1}, Lemmal 3.48]}.
\end{enumerate}
\end{thm}

\begin{proof}
(1) By the triangle identity for $(-)^{(1)} \dashv (-)_{-1}$, the composite
\[
G_{-1} \xrightarrow{\;\eta_{G_{-1}}\;} \big((G_{-1})^{(1)}\big)_{-1} \xrightarrow{\;(\varepsilon_G)_{-1}\;} G_{-1}
\]
is the identity on $G_{-1}$. On the other hand, since $q_G \circ \varepsilon_G = 1$, we have $(q_G)_{-1} \circ (\varepsilon_G)_{-1} = 1$. So composing with $\eta_{G_{-1}}$ we deduce  $(q_G)_{-1} = 1 \colon G_{-1} \to Q_b(G)_{-1}$. 

Consequently, in the sequential diagram $G \to Q_b(G) \to Q_b^2(G) \to \cdots$, every transition map on $(-)_{-1}$ is trivial. Since $\mathrm{Grp}_k^{\mathbb{A}^1} \subset \mathrm{Grp}_k$ is closed under filtered colimits and $(-)_{-1}$ preserves filtered colimits, $Q^\infty _bG:= \varinjlim_n Q_b^n(G)$ lies in $\mathrm{Grp}_k^{\mathbb{A}^1}$ and satisfies $(Q^\infty_b)_{-1} \cong \varinjlim_n (Q^n_b(G))_{-1} = 1$, so $Q_b^\infty G \in \mathrm{Grp}_k^b$ by \Cref{gm contraction 0}. Conversely, for any $B \in \mathrm{Grp}_k^b$,
\[
\operatorname{Hom}_{\mathrm{Grp}_k^{\mathbb{A}^1}}\big((G_{-1})^{(1)}, B\big) \cong \operatorname{Hom}_{\mathrm{Grp}_k^{\mathbb{A}^1}}(G_{-1}, B_{-1}) = *,
\]
so every morphism $G \to B$ factors uniquely through $Q_b(G)$ and hence through $Q_b^\infty$, identifying $Q_b^\infty G \xrightarrow{\sim} \pi_0^b(G)$. Finally, if $q_G \colon G \to Q_b(G)$ is a Nisnevich epimorphism in $\mathrm{Grp}_k$, then $(q_G)_{-1} \colon G_{-1} \twoheadrightarrow Q_b(G)_{-1}$ is also an epimorphism (as a retract of $q_G(\mathbb{G}_m \times -)$); since $(q_G)_{-1} = 1$, this forces $Q_b(G)_{-1} = 1$ and hence the universal property is already obtained by $Q_bG$, yielding $Q_b(G) \xrightarrow{\sim} \pi_0^b(G)$.

(2) On $\mathrm{Ab}_k^{\mathbb{A}^1}$, the functor $(-)_{-1} = \underline{\operatorname{Hom}}_{\mathrm{Ab}_k^{\mathbb{A}^1}}(\underline{K}_1^{MW}, -)$ is right adjoint to $- \otimes^{\mathbb{A}^1} \underline{K}_1^{MW}$. Since $\mathrm{Ab}_k^{\mathbb{A}^1} \subset \mathrm{Ab}_k$ is an exact abelian subcategory [\cite{morel2012a1}, Theorem~6.24], the coequalizer $q_M \colon M \twoheadrightarrow \operatorname{coker}_{\mathrm{Ab}_k^{\mathbb{A}^1}}(\varepsilon_M)$ is a Nisnevich epimorphism in $\mathrm{Ab}_k$ and $(-)_{-1}$ is exact on $Ab^{\mathbb{A}^1}_k$. So we deduce that $(q_M)_{-1} \colon M_{-1} \twoheadrightarrow \operatorname{coker}(\varepsilon_M)_{-1}$ is both zero (by the triangle identity as in (1)) and an epimorphism, forcing $\operatorname{coker}(\varepsilon_M)_{-1} = 0$ and thus, similar to the last paragraph, $\operatorname{coker}(\varepsilon_M) \xrightarrow{\sim} \pi_0^b(M)$.
\end{proof}

\begin{ex}\label[ex]{ex:GW-and-H0-obstruction}
For an arbitrary motivic space $\esc{X} \in \mathcal{H}^{\mathbb{A}^1}(k)$, applying \Cref{pi0b-presentation}(2) to $M = \mathbb{H}_0^{\mathbb{A}^1}(\esc{X})$ yields a canonical right exact sequence in $\mathrm{Ab}_k^{\mathbb{A}^1}$:
\[
\big(\mathbb{H}_0^{\mathbb{A}^1}(\esc{X})\big)_{-1} \otimes^{\mathbb{A}^1} \underline{K}_1^{MW} \longrightarrow \mathbb{H}_0^{\mathbb{A}^1}(\esc{X}) \longrightarrow \mathbb{H}_0^b(\esc{X}) \longrightarrow 0,
\]
which identifies $\big(\mathbb{H}_0^{\mathbb{A}^1}(\esc{X})\big)_{-1}$ as the exact obstruction to \Cref{Ha1 is bir loc for prop} holding outside $\mathcal{H}^{\mathbb{A}^1}_{prop}(k)$.
\end{ex}
\begin{ex}\label[ex]{formula to compute piib}
    For any pointed $\mathbb{A}^1$-$n$-connected space $\mathcal{X}$ ($n \ge 1$), \Cref{pi0b-presentation}(2) and \Cref{pi1b=pi0bpi1a1} yield the right exact sequence:
   $$(\pi_{n+1}^{\mathbb{A}^1}\mathcal{X})_{-1} \otimes^{\mathbb{A}^1} \underline{K}_1^{MW} \xrightarrow{\;\varepsilon\;} \pi_{n+1}^{\mathbb{A}^1}\mathcal{X}\longrightarrow \pi_{n+1}^b(\mathcal{X}) \to 0.$$ One can similarly write down the $n=0$ case using \Cref{pi0b-presentation}(1) instead.
\end{ex}
\section{Examples}\label{examples}
\subsection{Examples of Birational Motivic spaces}\label{examples of birational local spaces}
In this subsection, let $k$ be a field.

\begin{defn}[\textit{Birationally rigid schemes\protect\footnote{This should not be confused with the notion of birational rigidity used in classical birational geometry [\cite{MR3060242}]. In fact, the geometric concept is quite the opposite. For instance, being rationally connected, Fano varieties are never rigid in our sense, yet many Fano varieties are rigid in the other sense (see [\cite{MR2195677}]).}}]\label{birationally rigid schemes}
    We say that a scheme $T/k$ is birationally rigid if, for every dense open immersion $U\subset X$ of smooth schemes, the restriction maps $\mathrm{Sch}_k(X,T)\to \mathrm{Sch}_k(U,T)$ are isomorphisms, i.e., if it admits no nontrivial rational maps. Since schemes are sheaves for the $\sqcup$-topology, these are precisely the schemes that are birational local (see \Cref{bir local}).
\end{defn}
\begin{prop}\label[prop]{dim 0 is bir rigid}
Let $T$ be a Noetherian scheme over $k$.
\begin{enumerate}
    \item If $\dim T = 0$, then $T$ is affine and birationally rigid.
    \item Conversely, if $k$ is perfect and $T \in \mathrm{SmAff}_k$ is birationally rigid, then $\dim T =0$.
\end{enumerate}
\end{prop}
\begin{proof}
(1) By [\cite[\href{https://stacks.math.columbia.edu/tag/0AAX}{Tag 0AAX}]{stacks-project}] we are reduced to the case $T=\operatorname{Spec} L$. For any dense open immersion $U \hookrightarrow X$ of connected schemes in $\mathrm{Sm}_k$ and any $k$-morphism $U \to \operatorname{Spec} L$, evaluating at a closed point $u \in U$ gives an injection $L \hookrightarrow \kappa(u)$, so $L/k$ is algebraic. Because $X$ is normal, $\mathcal{O}_X(X)$ is integrally closed in $\mathcal{O}_X(U)$, so $L \hookrightarrow\mathcal{O}_X(U)$ must factor (uniquely) through $\mathcal{O}_X(X)$.

(2) Conversely, suppose $\dim T \geq 1$. Choosing a smooth curve $C$ in $T$, birationality would demand that $T$ contains a smooth projective compactification of $C$ (which exists by [\cite[\href{https://stacks.math.columbia.edu/tag/0H1F}{Remark 0H1F}]{stacks-project}]). This is impossible since $T$ is affine. 
\end{proof}

\begin{ex}[Abelian varieties]\label[ex]{abelian variety}
The standard example of a birationally rigid variety is an abelian variety [\cite{milne1986abelian}, Theorem 3.2]. In particular, elliptic curves over $k$ are birationally rigid.     
\end{ex}

 \begin{cor}\label[cor]{A1 rigid proper iff bir rigid} 
 A proper scheme over $k$ is birationally rigid if and only if it is $\mathbb{P}^1$-rigid, if and only if it is $\mathbb{A}^1$-rigid (for example, smooth proper curves of genus $\geq 1$ \textup{[\cite{levine2010slices}, page 399]}, and finite limits thereof).
 \end{cor}
 \begin{proof}
    For a proper scheme, the equivalence of $\mathbb{P}^1$-rigidity and $\mathbb{A}^1$-rigidity follows from the valuative criterion of properness. By \Cref{Motivic equivalences are Birational equivalences}(2), birational local schemes are $\mathbb{A}^1$-rigid and $\mathbb{P}^1$-rigid. The remaining implication follows from \Cref{rationally rigid is bir rigid} (note that the absence of rational curves is equivalent to $\mathbb{P}^1$-rigidity). 
 \end{proof}

\begin{ex}[Chow groups of zero cycles of smooth proper varieties]
 Let $Y$ be a smooth proper variety over $k$. The presheaf with transfers associated to the naive $\mathbb{A}^1$-homotopy quotient of a proper variety $Y$, i.e., $\mathbb{H}^{tr}_0(Y):=\pi_0^{nis}L_{\mathbb{A}^1}\mathbb{Z}^{tr}(Y),$ is a birational sheaf. Indeed, by [\cite{kahn2017birational}, Theorem 3.1.2] $\pi_0^{nis}L_{\mathbb{A}^1}\mathbb{Z}^{tr}(Y)\cong \mathrm{CH}_0(Y_{K(-)}),$ where the right-hand side is clearly birational (note that the only thing that requires properness of $Y$ is in providing the right-hand side with a presheaf structure, otherwise $X\mapsto \mathrm{CH}_0(Y_{K(X)})$ is not functorial on all morphisms of $\mathrm{Sm}_k$).
\end{ex}
    
\begin{rem} \label{GW-non-birational}
A similar claim does not apply to Milnor-Witt or framed correspondences: when $Y=\operatorname{Spec}k$, $\pi_0^{nis}L_{\mathbb{A}^1}\mathbb{Z}^{fr}(k)\simeq \pi_0^{nis}L_{\mathbb{A}^1}\tilde{\mathbb{Z}} ^{}(k)\simeq \underline{\mathrm{GW}}$ [\cite{MR3925101}, Theorem 8.8], which is not birational local since $\underline{\mathrm{GW}}(\mathbb{G}_{m, k}) \simeq GW(k) \oplus W(k)$. In fact, taking $M = \underline{\mathrm{GW}}$ in \Cref{pi0b-presentation}(2), we have $(\underline{\mathrm{GW}})_{-1} = \underline{W}$, and the counit $\varepsilon_{\underline{\mathrm{GW}}} \colon \underline{W} \otimes^{\mathbb{A}^1} \underline{K}_1^{MW} \to \underline{\mathrm{GW}}$ is identified with the multiplication map $$(\eta<a>,[b])\mapsto \eta<a>[b]=<a>(1-<b>)=<a>\cdot <<b>>,$$ so that the image of $\varepsilon_{\underline{\mathrm{GW}}} $ is the fundamental ideal sheaf $\underline{I} \subset \underline{\mathrm{GW}}$. Hence the universal birationalization is given by the rank map $\operatorname{rk} \colon \underline{\mathrm{GW}} \twoheadrightarrow \underline{\mathrm{GW}}/\underline{I} \cong \underline{\mathbb{Z}}=L_\Sigma\mathbb{Z} $, identifying  $\pi_0^b(\underline{\mathrm{GW}})=\underline{\mathbb{Z}}$. Similarly, $\pi_0^b(\underline{K}_n^{MW}) = \pi_0^b(\underline{K}_n^M) = 0$ for all $n \ge 1$.
\end{rem}
\begin{ex}\label[ex]{const sheaf of sets}
By \Cref{unibranch lemma}(3), over geometrically unibranch schemes with finitely many generic points, every constant sheaf of sets (and hence every constant sheaf of spaces, by \Cref{pi is birational}) is birational local. By \Cref{bir is closed under sifted colimits}, geometric realizations and filtered colimits of birational local spaces (such as Ind abelian varieties) are birational local. Further examples studied in companion works include $\Omega _{\mathbb{P}^1}^n\esc{X}$ for an $n$-birational motivic space $\esc{X}$ [\cite{sfbat}, Corollary 4.5.2], flat pushforwards and smooth pullbacks of birational local objects [\cite{sfbat}].
\end{ex}

\subsection{Birationally contractible and birationally connected spaces} 
\begin{defn}
We say that an $S$-space $\esc{Y}$ is birationally contractible if the projection map $\esc{Y}\to h_S(S)$ is a birational equivalence. In other words, $L_{bir}\esc{Y}$ is the contractible space. 
\end{defn}
\begin{ex}\label[ex]{many bir contractible schemes}
Every smooth rational (or stably rational, or retract-rational) scheme is birationally contractible (e.g., $\mathbb{A}^n, \mathbb{P}^n, GL_n, PGL_n, Gr_{m,n}, SL_n$, and their filtered colimits $GL, Gr_\infty, \mathbb{A}^\infty, \mathbb{P}^\infty$). A blowup of a birationally contractible smooth scheme along a smooth center, or a vector/projective bundle over a birationally contractible scheme, is birationally contractible. By \Cref{bir bar commute} (2), classifying spaces $\mathrm{B}_{nis}\esc{M}$ of birationally contractible monoids $\esc{M}$ (such as $\mathrm{B}_{nis}GL$) are birationally contractible, since $L_{bir}(\mathrm{B}_{nis}\esc{M})\simeq \mathrm{B}(L_{bir}\esc{M})\simeq \mathrm{B}(*)\simeq *$.
\end{ex} 

\begin{prop}\label[prop]{severi}
    A Severi-Brauer variety over a field $k$ is $\mathbb{A}^1$-connected iff it is birationally connected, iff it is birationally contractible, iff it is isomorphic to $\mathbb{P}^n_k$.
\end{prop}
\begin{proof}
Clearly, $\mathbb{P}^n_k$ is birationally contractible, and a birationally contractible space is birationally connected. The identification of birational connectivity with $\mathbb{A}^1$-connectivity is \Cref{proper schemes are a1 connected iff bir connected}.

Let $X$ be a Severi-Brauer variety over $k$ such that $\pi_0^bX\cong*$. Since $X$ is proper, by \Cref{pi0bX=pi0ba1X} $X(k)/_\sim \cong \pi_0^bX(k)\cong *$, so that $X(k)\neq \emptyset$. Since $X$ is a Severi-Brauer variety, this implies $X\cong \mathbb{P}^n_k$ [\cite{gille2017central}, Theorem~5.1.3].
\end{proof}
\begin{cor}\label[cor]{bir conn curve}
A smooth, proper curve over $k$ is birationally connected if and only if it is $\mathbb{A}^1$-connected if and only if it is birationally contractible if and only if it is isomorphic to $\mathbb{P}^1_k$.
\end{cor}
\begin{proof}
The proof is identical to the one above: such curves must be of genus $0$ (\Cref{A1 rigid proper iff bir rigid}) and $\mathbb{P}^1_k$ is the unique smooth proper curve of genus $0$ with a $k$-rational point.
\end{proof} 
\begin{rem}\label[rem]{bontr surfaces}
A similar characterization of smooth, proper, birationally contractible surfaces does not hold, unlike the $\mathbb{A}^1$-local characterization (over fields of characteristic zero) of the affine plane as the unique smooth $\mathbb{A}^1$-contractible surface [\cite{A1contractibledubouloz}, Theorem 3.12]. Indeed, $\mathbb{P}^1_k\times \mathbb{P}^1_k$, $\mathbb{P}^2_k$, their blow-ups, and the Hirzebruch surfaces $\Sigma_n:=\mathbb{P}_{\mathbb{P}^1}(\mathcal{O}_{\mathbb{P}^1}(-n)\oplus \mathcal{O}_{\mathbb{P}^1})$ are all smooth, proper, birationally contractible surfaces (\Cref{many bir contractible schemes}), but none isomorphic to the other.
\end{rem}

\vspace{1cm}
\noindent{\large\textbf{\textit{Acknowledgments}:}} I thank my supervisor, Dr. Chetan Tukaram Balwe, for carefully reviewing the first draft and Dr. Arun Kumar for his initial comments. I also thank Dr. Utsav Choudhury for his insightful suggestions on the second draft.

\phantomsection
\bibliographystyle{amsalpha}	
\renewcommand\refname{Bibliography}
\bibliography{references}
\noindent\rule{\textwidth}{0.4pt} 
\end{document}